\documentclass[review,onefignum,onetabnum]{siamonline220329}

\usepackage{lipsum}
\usepackage{amsfonts}
\usepackage{graphicx}
\usepackage{epstopdf}
\usepackage{xcolor}

\usepackage{makecell}

\def\sc{\scriptsize}
\def\sl{\small}

\ifpdf
  \DeclareGraphicsExtensions{.eps,.pdf,.png,.jpg}
\else
  \DeclareGraphicsExtensions{.eps}
\fi

\graphicspath{{Images/}}

\usepackage{enumitem}
\setlist[enumerate]{leftmargin=.5in}
\setlist[itemize]{leftmargin=.5in}

\newsiamremark{remark}{Remark}
\newsiamremark{hypothesis}{Hypothesis}
\crefname{hypothesis}{Hypothesis}{Hypotheses}
\newsiamthm{claim}{Claim}

\headers{AM2 model with interconnected chemostats}{T. Hmidhi, R. FEKIH-SALEM and J. Harmand}

\title{Analysis of anaerobic digestion model with two serial interconnected chemostats
\thanks{Submitted to the editors 2024-08-05.
\funding{
This work was supported by the Euro-Mediterranean research network TREASURE (\url{http://www.inra.fr/treasure}).}
}}

\author{
     THAMER HMIDHI  \thanks{University of Tunis El Manar, National Engineering School of Tunis, LAMSIN, Tunisia (\email{thamer.hmidhi@enit.utm.tn}).}
\and RADHOUANE FEKIH-SALEM \footnotemark[1]~\footnotemark[2]~\thanks{University of Monastir, Higher Institute of Computer Science of Mahdia, Tunisia     (\email{radhouane.fekih-salem@enit.utm.tn}).}
\and J\'ER\^OME HARMAND \thanks{LBE, University of Montpellier, INRAE, Narbonne, France (\email{jerome.harmand@inrae.fr}).}
}

\usepackage{amsopn}

\ifpdf
\hypersetup{
  pdftitle={HFH-AM2SerChem-SIADS},
  pdfauthor={HMIDHI, FEKIH-SALEM, HARMAND}
}
\fi

\begin{document}

\maketitle
\begin{abstract}
In this paper, we study a well known two-step anaerobic digestion model in a configuration of two chemostats in series. This model is an eight-dimensional system of ordinary differential equations. Since the reaction system has a cascade structure, we show that the eight-order model can be reduced to a four-dimensional one. Using general growth rates, we provide an in-depth mathematical analysis of the asymptotic behavior of the system. First, we determine all the steady states of the model where there can be more than fifteen equilibria with a non-monotonic growth rate. Then, the necessary and sufficient conditions of existence and local stability of all steady states are established according to the operating parameters: the dilution rate, the input concentrations of the two nutrients, and the distribution of the total process volume considered. The operating diagrams are then analyzed theoretically to describe the asymptotic behavior of the process according to the four control parameters. There can be seventy regions with rich behavior where the system may exhibit bistability or tristability with the coexistence of both microbial species in the two bioreactors.
\end{abstract}

\section{Introduction}
Anaerobic Digestion (AD) is a biological process in which organic matter is transformed by a complex microbial ecosystem into biogas in the absence of oxygen. It is used for the treatment of wastewater and the production of energy in the form of biogas containing methane. Under specific conditions, it may also be used to produce hydrogen from waste: the process is then named `black fermentation'. It is thus a fundamental process in the field of circular economy and green energy production from waste, \cite{ParuER2023}.

On the one hand, biological are attractive: based on natural processes, they are relatively easy and simple to manage. On the other hand, they are usually considered as low yield and rate processes when compared to other energy producing processes. This why it is particularly important to optimize their design. In particular, in (bio)chemical engineering, one very important and actually quite old question is whether a biological process
must be implemented in a homogeneous medium, or not. To study this question and characterize the mixing characteristics of the medium in a reactor, chemical engineers have proposed formal representations of chemical and biochemical devices known as `ideal' (or `perfect') reactors. On the one side of the scale on which reactors would be classified according to their homogeneity, there are the Completely Stirred Tank Reactors: they are completely mixed systems where the medium is considered as being perfectly homogeneous (CSTR). On the other side, we define the Plug Flow Reactors (PFR) which are nonhomogeneous systems in which small volumes of matter progress from the input to the output of the process independently of what happens around them. Practically, we can represent this situation
with a long pipe of small diameter in which slices of little volumes of medium progress at a given speed without any exchange with those which go before and those which come after, the bioreaction taking place only inside each slice of liquid. Notice that a PFR may be approximated by an infinite number of CSTR in series.
A limited number of CSTR is thus suited to represent situations between these two ideal cases, see for instance \cite{LevenspielBook1998}.

In microbiology/biotechnology, a CSTR is named `chemostat'. It consists in a tank filled with an aqueous solution in which microorganisms, generically named `biomass', grow on one or more resources, named `substrates'. The first introduction of the chemostat dates back to 1950 by its inventors Novick and Szilard \cite{NovickSience1950} simultaneously with Monod \cite{Monod1950}. More than a common laboratory apparatus used for the continuous cultivation of microorganisms, its model has been studied as a mathematical object since its origins in the fifties \cite{HarmandBook2017,SmithBook1995}. As such, it plays an important role as a model to study the growth of microbes and their interactions in mathematical biology and mathematical ecology. It has been long time used as a model of wastewater biological treatment processes and has been widely used to model industrial bioreactors.

Following the increasing knowledge accumulated over time about the anaerobic microbial ecosystems, the AD process is schematically represented as a set of bioreactions in cascade realized by several groups of microbes, the output of a reaction step being considered to be the input of the next one. Depending on the nature of its inputs, the AD is represented either as a one, two, three, or even four-step process, \emph{cf.} for instance \cite{WeedermanJBD2013}.
It is recognized that the most complete actual model is the ADM1 developed initially proposed by Batstone \emph{et al.} \cite{BatstoneWST2002,IWAReport2002}.
Of course, greater the complexity, lesser the possibility of analytically analyzing the properties of the associated model. Indeed, because of its complexity, and notably its relatively high dimension, the ADM1 is not easy to study from a mathematical viewpoint: very few studies were able to investigate its properties. From the best of our knowledge, there exists only one paper studying its steady states and their stability (see \cite{BornhoftND2013}).
The AM2 model proposed by \cite{Bernard2001} is a two-step AD model. It is particularly interesting because it is of medium complexity.
It has been shown to be able to represent real data in capturing the main dynamical features of different full-scale AD processes and
is thus well suited for control design. During the last decade, many papers in the literature have studied the AM2 model and some of its extensions (see \cite{BenyahiaMBE2020,BenyahiaJPC2012,FekihSIADS2021,SariNonLinDyn2021}).

The `optimal design' of a biological system consists in finding the `best configuration' - in the sense of a given performance criterion - of this system.
In particular, a major question in process engineering is to determine which type of reactor or interconnection of reactors - typically perfect reactors - is the most suitable to maximize such a performance for a given biological process, \cite{GooijerEMT1996}. By `type or reactor' or `interconnection of reactors', we essentially mean determining whether it is important to favor a homogeneous medium (and thus use a chemostat), or not
(and thus rather use an interconnection of chemosts, or even a PFR).
Notice that there exist many alternative configurations, such as several chemostats working in parallel with or without convective or diffusive connections.
Several mathematical studies have addressed these issues in considering what is named the `minimal chemostat model', that is a single-step biological reaction and different biomass growth rates, see \cite{DaliYoucefIFAC2022, DaliYoucefMBE2020, DaliYoucefBMB2022, DaliYoucefIJB2023, HaidarMB2011, HarmandJPC2004, HarmandAIChE2003} for configurations involving chemostats in series, and \cite{HaidarMB2011, LevenspielBook1998} for configurations of chemostats in parallel.
For instance, considering the maximization of conversion yield, Haidar \emph{et al.} \cite{HaidarMB2011} have shown that there exists a threshold on the concentration of nutrient supply for which the series and parallel configurations are better than a single chemostat at steady state. In \cite{DaliYoucefMBE2020}, authors have compared the performances between the two serial interconnected chemostats and of a single chemostat and showed the most effective serial configuration of two chemostats against one to maximize the conversion of the substrate.
Notice that if the objective is to maximize the productivity of biomass or, equivalently the production of biogas, the results may be different. And indeed, in the most general case, the results highly depend on both the growth rate considered, the input substrate concentration, and on the objective pursued.
The papers cited hereabove are very general studies that did not consider specifically the AD process and only concentrated on the optimal design of single-step biological reactional systems.
What about the optimal design of multiple-step processes as it is the case for the AD? In particular, do the results establish for a one-step system hold for a two-step process? In addition, one may further question whether a configuration in which only the first reaction would take place in one reactor (in which case the first reactor may be used to produce hydrogen) and the second reaction in another one would be better than any other configuration of reactors, or not.
Several authors have begun to address these questions either experimentally (\emph{cf.} for instance \cite{SchievanoAE2014}) or theoretically in using the minimal chemostat model in each stage (\emph{cf.} for instance \cite{ChorukovaECES2021,DonosoMBE2019,PanIFAC2022}).
However, before addressing the problem of optimally designing multi-step biological processes, it is essential to study the properties of such a process in a cascade of chemostats, that are the existence and the steady states properties of the steady states.


In this work, we propose an original model describing the two-stage AD process in two series-connected chemostats. Using the same removal rates of two species in the two bioreactors, the eight-dimensional system with a cascade structure can be reduced to a four-dimensional system. Using monotonic growth rates for the first species and non-monotonic growth rates for the second species, we investigate the existence and stability of all steady states according to the control parameters. Moreover, we theoretically describe the operating diagrams and show that the model exhibits a very rich behavior, where there can be bistability and tristability with the coexistence of two populations of micro-organisms in each bioreactor.


Indeed, the operating diagram is a very useful tool for engineers, enabling them to graphically synthesize the asymptotic behavior of the process as a function of the control parameters.
It was established in the book \cite{HarmandBook2017} for the classical chemostat model with Monod and Haldane growth rates.
Three methods of analysis and construction of the operating diagram in the existing literature have been reported and summarized in Mtar et al \cite{MtarDcdsB2022}: purely numerical point-by-point \cite{WadePloS2017,WadeJTB2016}, by the numerical continuation of bifurcations through software packages such as MatCont (see \cite{DhoogeMCMDS2008} and the reference therein) and theoretical from the existence and stability conditions. We can cite the following papers based on the numerical study of the operating diagram as in \cite{HanakiProc2021,KhedimAMM2018,SbarciogBEJ2010,WadeJTB2016,WeedermannNonlinDyn2015,XuJTB2011}.
Several other recent papers have been based on the theoretical study of the operating diagram with various biological and ecological applications:
anaerobic digestion \cite{DaoudMMNP2018,FekihSIADS2021,SariProcess22,SariNonLinDyn2021,SariMB2016,WeedermanJBD2013},
microbial food webs \cite{NouaouraSiap2021,NouaouraDcDsB2021,NouaouraJMB2022,SariMB2017,SobieszekMBE2020}
serially interconnected chemostats \cite{DaliYoucefIFAC2022,DaliYoucefMBE2020,DaliYoucefBMB2022,DaliYoucefIJB2023},
competition with inhibition and allelopathy effects \cite{BarDCDSB2020,DellalDcDsB2021,DellalDcDsB2022,DellalMB2018,ZitouniJMB2023}
and density-dependence models \cite{FekihMB2017,MtarIJB2020,MtarDcdsB2022}.


This paper is organized as follows:
in \cref{SecAnalysisMod}, we propose a original model describing the AD process with a serial configuration of two chemostats.
In \cref{SecRedModel}, we show that the study of the eight-dimension system can be reduced to a four-dimension one.
\Cref{SecExisMulti} is devoted to analyzing the existence conditions and multiplicity of all steady states of the reduced model.
\Cref{SecStabSS} is dedicated to establishing the local stability conditions of all steady states.
Then, in \cref{Sec_DO}, we establish theoretically the operating diagram from the existence and stability conditions.
In \cref{Sec-Conc}, we discuss our results by comparing them with those of the existing literature.
The proofs of all results are reported in \cref{AppProof}. The parameter values used in all figures and the tables defining the various regions of the operating diagram are provided in \cref{Sec_AppTableVal}.
\section{Modeling the configuration of two interconnected chemostats} \label{SecAnalysisMod}
We first recall the general two-step model in a single chemostat with a cascade of the following two biological reactions (see \cite{Bernard2001}):
\begin{equation}                                          \label{r1Ch1}
k_1 S_1 \xrightarrow{\quad{\mu_1(\cdot)\quad}} X_1+ k_2 S_2 + k_4 CO_2, \qquad  k_3 S_2 \xrightarrow{\quad{\mu_2(\cdot)\quad}} X_2 + k_5 CO_2 + k_6 CH_4
\end{equation}
where the organic substrate $S_1$ is consumed by acidogenic bacteria $X_1$ to produce a product $S_2$ (Volatile Fatty Acid).
In the second reaction, the methanogenic bacteria $X_2$ consumes the substrate $S_2$ and produces biogas.
$\mu_1$ and $\mu_2$ are the kinetics of the two reactions, $k_i$, $i=1,\ldots,6$ are the pseudo-stoichiometric coefficients.
The substrates $S_1$ and $S_2$ are introduced in the reactor with the inflowing concentrations $S_1^{in}$ and $S_2^{in}$, respectively, with the input flow rate $Q$ so that the dilution rate $D=Q/V$.
Thus, we obtain the so-called AM2 model which is described by the following ordinary differential equations:
\begin{equation}                                        \label{ModelAM2}
\left\{ \begin{array}{lll}
\dot{S}_1 &=& D \left(S_1^{in}- S_1\right)- k_1\mu_1(\cdot) X_1,                 \\
\dot{X}_1 &=& \left(\mu_1(\cdot)-\alpha D\right) X_1,                              \\
\dot{S}_2 &=& D \left(S_2^{in}-S_2\right)+ k_2 \mu_1(\cdot) X_1-k_3\mu_2(\cdot) X_2, \\
\dot{X}_2 &=& \left(\mu_2(\cdot)-\alpha D\right) X_2,
\end{array} \right.
\end{equation}
where only a fraction $\alpha$ of the biomass is in the liquid phase.
More precisely, $\alpha\in [0, 1]$ is a parameter allowing us to decouple the Hydraulic Retention Time (HRT) and the Solid Retention Time (SRT) (see \cite{Bernard2001} for more detail on the model development).

The AM2 model \cref{ModelAM2} has been studied in \cite{BenyahiaJPC2012,SariNonLinDyn2021,SbarciogBEJ2010} when the growth rate $\mu_i$ for $i=1,2$ depends only on substrate $S_i$ ($\mu_i(\cdot)=\mu_i(S_i))$, that is, \cref{ModelAM2} has a cascade-connected subsystems describing a commensalistic relationship.
In \cite{FekihSIADS2021}, an in-depth mathematical study of model \cref{ModelAM2} has been established by considering a syntrophic relationship between the species, that is, $\mu_1(\cdot)$ depends on both substrates $S_1$ and $S_2$ while $\mu_2(\cdot)$ depends on the substrate $S_2$ ($\mu_1(\cdot)=\mu_1(S_1,S_2)$, $\mu_2(\cdot)=\mu_2(S_2)$). A summary review of the AM2 model for commensalistic and syntrophic relationships describing the modeling assumptions and the growth rates is provided in \cite{FekihSIADS2021}.
Recently, the maximization of biogas production of the AM2 model \cref{ModelAM2} was established theoretically according to the bioreactor control parameters \cite{SariProcess22}.

In this paper, we consider a serial configuration of two interconnected chemostats where the volume $V$ is divided into two volumes, $r_1V=rV$ and $r_2V=(1-r)V$ with $r\in(0,1)$ (see \cref{Figchemo}). The two-tiered model we consider here describes the same biological reactions \cref{r1Ch1}.
However, the concentrations of substrate $S_1$ and $S_2$ and biomass $X_1$ and $X_2$ are different in each bioreactor, which we'll note as $S_i^j$ and $X_i^j$, $i,j=1,2$ in each bioreactor $j$.
Moreover, $S_1^{in}$ and $S_2^{in}$ are the concentrations of input substrate $S_1^1$ and $S_2^1$ in the first chemostat at a flow rate $Q$. The substrates $S_1^2$ and $S_2^2$ are introduced in the second chemostat with the inflowing concentrations $S_1^1$ and $S_2^1$, respectively, and a flow rate $Q$.
\begin{figure}[!ht]
\setlength{\unitlength}{1.0cm}
\begin{center}
\begin{picture}(7,25.3)(0,0)		
\put(-6.7,0){{\includegraphics[width=20cm,height=25cm]{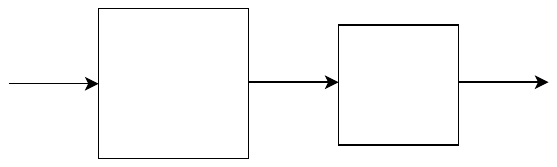}}}
\put(-2.5,23.7){\sc $S_1^{in},S_2^{in}$}
\put(0.8,24){\sc $S_1^1,X_1^1$}
\put(0.8,23.3){\sc $S_2^1,X_2^1$}
\put(-0.8,23.8){\sc $Q$}
\put(0.8,24.9){\sc $r_1V=rV$}
\put(4.4,24){\sc $S_1^2,X_1^2$}
\put(4.4,23.3){\sc $S_2^2,X_2^2$}
\put(4,24.7){\sc $r_2V=(1-r)V$}
\put(7.5,23.7){\sc $S_1^2,X_1^2,S_2^2,X_2^2$}
\put(3.1,23.8){\sc $Q$}
\put(6.4,23.8){\sc $Q$}
\end{picture}
\end{center}
\vspace{-22.7cm}
\caption{The serial configuration of two interconnected chemostats.} \label{Figchemo}
\end{figure}

Using Antoine Lavoisier's law: ``Nothing is lost, nothing is created, everything is transformed'' and establishing the mass balance between the instants $t$ and $t+dt$, we obtain the mathematical model given by the following eight-dimensional system of ordinary differential equations:
\begin{equation}                                  \label{ModelAM2S2C}
\left\{\begin{array}{lll}
\dot{S}_1^1  &=& D_1\left(S_1^{in}- S_1^1\right)-k_1\mu_1\left(S_1^1\right) X_1^1, \\
\dot{X}_1^1  &=&\left(\mu_1\left(S_1^1\right) - D_1\right)X_1^1, \\
\dot{S}_2^1  &=& D_1\left(S_2^{in}- S_2^1\right)+k_2\mu_1\left(S_1^1\right)X_1^1-k_3\mu_2\left(S_2^1\right) X_2^1,  \\
\dot{X}_2^1  &=&\left(\mu_2\left(S_2^1\right) - D_1\right)X_2^1,\\
\dot{S}_1^2  &=& D_2\left(S_1^1- S_1^2\right)-k_1\mu_1\left(S_1^2\right) X_1^2, \\
\dot{X}_1^2  &=& D_2\left(X_1^1- X_1^2\right)+\mu_1\left(S_1^2\right)X_1^2,  \\
\dot{S}_2^2  &=& D_2\left(S_2^1- S_2^2\right)+k_2\mu_1\left(S_1^2\right)X_1^2-k_3\mu_2\left(S_2^2\right) X_2^2 ,\\
\dot{X}_2^2  &=& D_2\left(X_2^1- X_2^2\right)+\mu_2\left(S_2^2\right)X_2^2,
\end{array}\right.
\end{equation}
where $D_1=D/r_1$ and $D_2=D/r_2$ are the dilution rates corresponding to the first and second chemostats, respectively.
In this paper, we assume that the growth functions $\mu_1(S_1)$ and $\mu_2(S_2)$ belongs to $\mathcal{C}^1({R}_+)$ and satisfy the following hypotheses.
\begin{hypothesis}                                    \label{Hypoth1}
$\mu_1(0)=0$, $\mu_1(+\infty)=m_1$, and $\mu_1'(S_1)>0$ for all $S_1>0$.
\end{hypothesis}
\begin{hypothesis}                                    \label{Hypoth2}
 $\mu_2(0)=0$, $\mu_2(+\infty)=0$, and there exists $S_2^m>0$ such that $\mu_2 '(S_2)>0$, for all $0<S_2<S_2^m$ and $\mu_2 '(S_2)<0$, for all $S_2>S_2^m$.
\end{hypothesis}

\cref{Hypoth1} means that the substrate is necessary for the growth of the first species. In addition, the growth rate of the first species is favored by the concentration of the substrate $S_1$.
\cref{Hypoth2} means that the growth of the second species takes into account the growth-limiting
for low concentrations of substrate $S_2$ and the growth-inhibiting for high concentrations.

To preserve the biological meaning of our model \cref{ModelAM2S2C}, we prove the following result. Since the proof is classical, it is left to the reader.
\begin{proposition}
All solutions of model \cref{ModelAM2S2C} with nonnegative initial condition remain nonnegative
and bounded for all positive times. Let $\xi=\left(S_1^1,X_1^1,S_2^1,X_2^1,S_1^2,X_1^2,S_2^2,X_2^2\right)$ be the state vector. The set
$$
\Omega=\left\{ \xi\in \mathbb{R}_+^8: S_1^i+k_1X_1^i=S_1^{in}~\mbox{and}~S_2^i+k_3X_2^i-k_2X_1^i=S_2^{in},i=1,2\right\}
$$
is positively invariant and is a global attractor for \cref{ModelAM2S2C}.
\end{proposition}
\section{Reduced model}                           \label{SecRedModel}
In the following, we show that the existence and asymptotic stability of the steady states of the eight-dimensional system \cref{ModelAM2S2C} can be deduced from that of the four-dimensional system \cref{ModelAM2SR4}.
The change of variable
$$
Z_1^i=S_1^i+k_1X_1^i \quad\mbox{and}\quad Z_2^i=S_2^i+k_3X_2^i-k_2X_1^i, \quad i=1,2,
$$
transforms system \cref{ModelAM2S2C} into a system having the same structure, namely a system of the form
\begin{equation}                                   \label{ModelAM2SZi}
\left\{
\begin{array}{lll}
\dot{Z}_1^1  &=& D_1\left(S_1^{in}- Z_1^1\right), \\
\dot{Z}_2^1  &=& D_1\left(S_2^{in}- Z_2^1\right),  \\
\dot{Z}_1^2  &=& D_2\left(Z_1^1- Z_1^2\right), \\
\dot{Z}_2^2  &=& D_2\left(Z_2^1- Z_2^2\right), \\
\dot{X}_1^1  &=&\left(\mu_1\left(Z_1^1-k_1X_1^1\right) - D_1\right)X_1^1,  \\
\dot{X}_2^1  &=&\left(\mu_2\left(Z_2^1+k_2X_1^1-k_3X_2^1\right) - D_1\right)X_2^1,  \\
\dot{X}_1^2  &=&D_2\left(X_1^1- X_1^2\right)+\mu_1\left(Z_1^2-k_1X_1^2\right)X_1^2, \\
\dot{X}_2^2  &=&D_2\left(X_2^1- X_2^2\right)+\mu_2\left(Z_2^2+k_2X_1^2-k_3X_2^2\right)X_2^2.
\end{array}\right.
\end{equation}
Thus, the steady states of the eight-dimension system \cref{ModelAM2S2C} are given by the steady state of \cref{ModelAM2SZi}.
From the first four equations of \cref{ModelAM2SZi}, we see that a steady state of \cref{ModelAM2SZi} must satisfy $Z_i^j=S_i^{in}$, $i,j=1,2$.
Setting $Z_i^j=S_i^{in}$ in the lower subsystem of \cref{ModelAM2SZi}, we obtain the following reduced system
\begin{equation}                                   \label{ModelAM2SR4}
\left\{
\begin{array}{lll}
\dot{X}_1^1 &=& \left(\mu_1\left(S_1^{in}-k_1X_1^1\right) - D_1\right)X_1^1,  \\
\dot{X}_2^1 &=& \left(\mu_2\left(S_2^{in}+k_2X_1^1-k_3X_2^1\right) - D_1\right)X_2^1,  \\
\dot{X}_1^2 &=& D_2\left(X_1^1- X_1^2\right)+\mu_1\left(S_1^{in}-k_1X_1^2\right)X_1^2 , \\
\dot{X}_2^2 &=& D_2\left(X_2^1- X_2^2\right)+\mu_2\left(S_2^{in}+k_2X_1^2-k_3X_2^2\right)X_2^2.
\end{array}
\right.
\end{equation}
Hence, the last four components of a steady state of the eight-dimensional system \cref{ModelAM2SZi} are given by putting the left-hand side of the reduced system \cref{ModelAM2SR4} equal to zero. Thus, for each steady state of the four-dimensional system \cref{ModelAM2SR4} denoted by
$$
\mathcal{E}_{ij}^{kl}\left(X_1^1,X_2^1,X_1^2,X_2^2\right), \quad \{i,j,k,l\}\in \{0,1,2\}
$$
corresponds a steady state of the eight-dimensional system \cref{ModelAM2S2C} denoted by
$$
\mathcal{F}_{ij}^{kl}\left(S_1^1,X_1^1,S_2^1,X_2^1,S_1^2,X_1^2,S_2^2,X_2^2\right)
$$
where
$$
S_1^i=S_1^{in}-k_1X_1^i \quad\mbox{and}\quad S_2^i=S_2^{in}-k_3X_2^i+k_2X_1^i,i=1,2.
$$
To analyze the local stability of each steady state, system \cref{ModelAM2SZi} possesses a pair of cascade-connected subsystems in the $(\xi_1,\xi_2)$ coordinates of the form (see e.g. \cite{IsidoriBook1999})
\begin{eqnarray*}
\left\{\begin{array}{l}
 \dot{\xi}_1 = f(\xi_1) \\
 \dot{\xi}_2 = g(\xi_2,\xi_1),
\end{array}\right.
\end{eqnarray*}
where
\begin{displaymath}
f(\xi_1)=
\left[\begin{array}{c}
   D_1\left(S_1^{in}- Z_1^1\right) \\
   D_1\left(S_2^{in}- Z_2^1\right)  \\
   D_2\left(Z_1^1- Z_1^2\right) \\
   D_2\left(Z_2^1- Z_2^2\right)
\end{array}\right],\quad
g(\xi_2,\xi_1)=
\left[\begin{array}{l}
 \left(\mu_1\left(Z_1^1-k_1X_1^1\right) - D_1\right)X_1^1  \\
 \left(\mu_2\left(Z_2^1+k_2X_1^1-k_3X_2^1\right) - D_1\right)X_2^1  \\
 D_2\left(X_1^1- X_1^2\right)+\mu_1\left(Z_1^2-k_1X_1^2\right)X_1^2, \\
 D_2\left(X_2^1- X_2^2\right)+\mu_2\left(Z_2^2+k_2X_1^2-k_3X_2^2\right)X_2^2
\end{array}\right]
\end{displaymath}
where  $\xi_1=\left(Z_1^1,Z_2^1,Z_1^2,Z_2^2\right) \in \mathbb{R}_+^4$ and $\xi_2=\left(X_1^1,X_2^1,X_1^2,X_2^2\right) \in \mathbb{R}_+^4$. So that
$$
g\left(\xi_1^*\right)=0, \quad g\left(\xi_2^*,\xi_1^*\right)=0,
$$
Note that $\xi_1^*$ is LES of the upper subsystem of \cref{ModelAM2SZi} since at $\xi_1^*$, the characteristic polynomial is
$$
P(\lambda)=\left(D_1+\lambda\right)^2 \left(D_2+\lambda\right)^2.
$$
Using \cite[Corollary 10.3.2]{IsidoriBook1999}, we show that the local stability of $\xi_2^*$ of the lower subsystem $\dot{\xi}_2=g\left(\xi_2,\xi_1^*\right)$ is the same as the one of the cascade \cref{ModelAM2SZi} at $(\xi_1,\xi_2)=\left(\xi_1^*,\xi_2^*\right)$.

Since the components $S_1^1$, $S_2^1$, $S_1^2$ and $S_2^2$ are nonnegative, system \cref{ModelAM2SR4} is defined on the set:
\begin{equation}                                              \label{M}
 M = \left\{ \left(X_1^1,X_2^1,X_1^2,X_2^2\right)\in \mathbb{R}_+^4: X_1^i\leq \frac{S_1^{in}}{k_1}\mbox{ and } X_2^i\leq \frac{S_2^{in}+k_2X_1^i}{k_3},i=1,2\right\}.
 \end{equation}
\subsection{Existence and multiplicity of the steady states}        \label{SecExisMulti}
A steady state exists if and only if all its components are nonnegative.
Each steady state is preceded by a subscript and an exponent, indicating which population of microorganisms survives in the first and the second chemostat, respectively.
For example, the subscript in $\mathcal{E}_{0i}^{11}$, $i=1,2$ implies that $X_1^1=0$ and $X_2^1>0$ with the possibility of two different concentrations of $X_2^1$ at the steady state.
However, the exponent in $\mathcal{E}_{0i}^{11}$ implies that $X_1^2>0$ and $X_2^2>0$.
Indeed, system \cref{ModelAM2SR4} has the following nine types of steady states with $i=1,2$.
\begin{itemize}[leftmargin=*]
\item $\mathcal{E}_{00}^{00}\left(X_1^1=0,X_2^1=0,X_1^2=0,X_2^2=0\right):$ washout of the two species in each bioreactor.
\item $\mathcal{E}_{00}^{0i}\left(X_1^1=0,X_2^1=0,X_1^2=0,X_2^2>0\right):$ only the second species is in the second bioreactor.
\item $\mathcal{E}_{00}^{10}\left(X_1^1=0,X_2^1=0,X_1^2>0,X_2^2=0\right):$ only the first species is in the second bioreactor.
\item $\mathcal{E}_{00}^{1i}\left(X_1^1=0,X_2^1=0,X_1^2>0,X_2^2>0\right):$ only the two species is in the second bioreactor.
\item $\mathcal{E}_{10}^{10}\left(X_1^1>0,X_2^1=0,X_1^2>0,X_2^2=0\right):$ only the first species is in each bioreactor.
\item $\mathcal{E}_{10}^{1i}\left(X_1^1>0,X_2^1=0,X_1^2>0,X_2^2>0\right):$ exclusion of the second species in the first bioreactor.
\item $\mathcal{E}_{0i}^{01}\left(X_1^1=0,X_2^1>0,X_1^2=0,X_2^2>0\right):$ only the second species is in each bioreactor.
\item $\mathcal{E}_{0i}^{11}\left(X_1^1=0,X_2^1>0,X_1^2>0,X_2^2>0\right):$ exclusion of the first species in the first bioreactor.
\item $\mathcal{E}_{1i}^{11}\left(X_1^1>0,X_2^1>0,X_1^2>0,X_2^2>0\right):$ the two species in each bioreactor are maintained.
\end{itemize}

To describe the steady states of \cref{ModelAM2SR4}, we define in \cref{TabFunc1} the following auxiliary functions depending on the operating parameters $D$, $r$, $S_1^{in}$, and $S_2^{in}$. To simplify the notations, we will denote in the following
$$
\lambda_1^i=\lambda_1^i(D,r) \quad \mbox{and} \quad\lambda_2^{ij}=\lambda_2^{ij}(D,r),\quad i,j=1,2.
$$
\begin{table}[ht]
\caption{Auxiliary functions depending on $D$, $r$, $S_1^{in}$, and $S_2^{in}$ for $i,j=1,2$ where $D_i:=D/r_i$, $D_i^m:=r_i\mu_2\left(S_2^m\right)$, $r_1=r$, $r_2=1-r$, and $X_1^{2*}$ is the unique solution of equation $f_1(x)=g_1(x)$ with the functions $f_1$ and $g_1$ are defined in \cref{TabFunc2}.}  \label{TabFunc1}
\begin{center}
\begin{tabular}{ @{\hspace{1mm}}l@{\hspace{1mm}}  @{\hspace{1mm}}l@{\hspace{1mm}} }	
\hline
                                         &  Definition
\\ \Xhline{3\arrayrulewidth}
$\lambda_1^i(D,r)$                       &
\begin{tabular}{l}
$\lambda_1^i(D,r)$ is the unique solution of equation $\mu_1(S_1) = D_i$. \\
It is defined for $0 < D < r_im_1$. \\
If $D \geq r_im_1$, by convention we let $\lambda_1^i(D,r) = +\infty$.
\end{tabular}
\\ \hline
$\lambda_2^{ij}(D,r)$                     &
\begin{tabular}{l}
$\lambda_2^{i1}(D,r)<\lambda_2^{i2}(D,r)$ are the two solutions of equation $\mu_2\left(S_2\right) = D_i$. \\
They are defined for $0 < D \leq D_i^m$. \\
If $D > D_i^m$, by convention we let $\lambda_2^{ij}(D,r) = +\infty$.
\end{tabular}
\\ \hline
$F_{1j}\left(D,r,S_2^{in}\right)$         &
\begin{tabular}{l}
$F_{1j}\left(D,r,S_2^{in}\right) =\lambda_1^1(D,r)+\frac{k_1}{k_2}\left(\lambda_2^{1j}(D,r)-S_2^{in}\right)$. \\
It is defined for $0 < D < \min\left(r_1m_1, D_1^m\right)$.
\end{tabular}
\\ \hline
$F_{2j}\left(D,r,S_2^{in}\right)$         &
\begin{tabular}{l}
$F_{2j}\left(D,r,S_2^{in}\right) =\lambda_1^2(D,r)+\frac{k_1}{k_2}\left(\lambda_2^{2j}(D,r)-S_2^{in}\right)$. \\
It is defined for $0 < D < \min\left(r_2m_1, D_2^m\right)$.
\end{tabular}
\\ \hline
$\phi_j\left(D,r,S_1^{in},S_2^{in}\right)$         &
\begin{tabular}{l} $\phi_j\left(D,r,S_1^{in},S_2^{in}\right)=S_2^{in}+k_2X_1^{2*}\left(D,r,S_1^{in}\right)-\lambda_2^{2j}(D,r)$.\\
It is defined for $0 < D \leq D_2^m$ and $S_1^{in}>\lambda_1^1$.
\end{tabular}
\\ \hline
\end{tabular}
\end{center}
\end{table}

On the other hand, to establish the existence and multiplicity of steady states of \cref{ModelAM2SR4}, we also need to define some auxiliary functions listed in \cref{TabFunc2}.
\begin{table}[!ht]
\begin{center}
\caption{Auxiliary functions and their domains of definition. Note that $X_1^{1*}$, $X_2^{1*}$, $X_1^{2*}$ are the components of different steady state in \cref{TabComSS}, and $k_i,i=1,2,3$ are the pseudo-stoichiometric coefficients.}\label{TabFunc2}
\begin{tabular}{llll}
\hline
  Auxiliary functions                                    &  Domains of definition \\ \Xhline{3\arrayrulewidth}
$g_1(x):=D_2\left(\frac{x-X_1^{1*}}{x}\right)$           & $x\in \left(0, +\infty\right)$ \\
$g_2(x):=D_2\left(\frac{x-X_2^{1*}}{x}\right)$           & $x\in \left(0, +\infty\right)$\\
$f_1(x):=\mu_1\left(S_1^{in}-k_1x\right)$                & $x\in \left[0, S_1^{in}/k_1\right]$  \\
$f_2(x):=\mu_2\left(S_2^{in}-k_3x\right)$                & $x\in \left[0, S_2^{in}/k_3\right]$\\
$f_3(x):=\mu_2\left(S_2^{in}+k_2X_1^{2*}-k_3x\right)$    & $x\in \left[0, \left(S_2^{in}+k_2X_1^{2*}\right)/k_3\right]$
\\ \hline
\end{tabular}
\end{center}
\end{table}
In the following proposition, we provide the necessary and sufficient conditions for the existence of all steady states of model \cref{ModelAM2SR4}.
\begin{proposition}                                     \label{propES}
Assume that \cref{Hypoth1,Hypoth2} hold. Then, the nine types of steady states of \cref{ModelAM2SR4} are given in \cref{TabComSS}. Their existence conditions are given in \cref{TableCondExis}.
\end{proposition}
\begin{table}[!ht]
\begin{center}
\caption{Components of all steady states of \cref{ModelAM2SR4}. The functions $\lambda_1^i$, $\lambda_2^{1i}$, $\lambda_2^{2i}$, $F_{ij}$ and $\phi_i$, $i,j=1,2$ are defined in \cref{TabFunc1} while the functions $f_1$, $f_2$, $f_3$, $g_1$ and $g_2$ are defined in \cref{TabFunc2}.}\label{TabComSS}
\begin{tabular}{@{\hspace{1mm}}l@{\hspace{1mm}} @{\hspace{1mm}}l@{\hspace{1mm}} @{\hspace{0mm}}l@{\hspace{0mm}}}	
\hline
                          & $\left(X_1^{1*}, X_2^{1*}\right)$        & $\left(X_1^{2*}, X_2^{2*}\right)$   \\ \Xhline{3\arrayrulewidth}
 $\mathcal{E}_{00}^{00}$  & (0, 0)       &  (0, 0) \\ \hline
$\mathcal{E}_{00}^{0i}$   & (0, 0)       &  $\left(0, \left(S_2^{in}-\lambda_2^{2i}\right)/k_3\right)$  \\ \hline
$\mathcal{E}_{00}^{10}$   & (0, 0)       &  $\left(\frac{1}{k_1}\left(S_1^{in}-\lambda_1^2\right), 0\right)$ \\ \hline
$\mathcal{E}_{00}^{1i}$   & (0, 0)       &  $\left(\left(S_1^{in}-\lambda_1^2\right)/k_1,  k_2\left(S_1^{in}-F_{2i}\right)/(k_1k_3)\right)$ \\ \hline
$\mathcal{E}_{10}^{10}$   & $\left(\left(S_1^{in}-\lambda_1^1\right)/k_1, 0\right)$    &
\begin{tabular}{l}
$\left(X_1^{2*}, 0\right)$  \\
$X_1^{2*}$ is the unique solution of $f_1(x)=g_1(x)$
\end{tabular}    \\ \hline
$\mathcal{E}_{10}^{1i}$   & $\left(\left(S_1^{in}-\lambda_1^1\right)/k_1, 0\right)$    &
\begin{tabular}{l}
$\left(X_1^{2*}, \phi_i/k_3\right)$  \\
$X_1^{2*}$ is the unique solution of $f_1(x)=g_1(x)$
\end{tabular}     \\ \hline
$\mathcal{E}_{0i}^{01}$   & $\left(0, \left(S_2^{in}-\lambda_2^{1i}\right) /k_3 \right)$    &
\begin{tabular}{l}
$\left(0, X_2^{2*} \right)$  \\
$X_2^{2*}$ is a solution of $f_2\left(X_2^2\right)=g_2\left(X_2^2\right)$
\end{tabular}
\\ \hline
$\mathcal{E}_{0i}^{11}$   & $\left(0, \left(S_2^{in}-\lambda_2^{1i}\right)/k_3 \right)$    &
\begin{tabular}{l}
$\left(\left(S_1^{in}-\lambda_1^2\right)/k_1, X_2^{2*} \right)$  \\
$X_2^{2*}$ is a solution of $f_3(x)=g_2(x)$
\end{tabular}
\\ \hline
$\mathcal{E}_{1i}^{11}$   & $\left(\frac{1}{k_1}\left(S_1^{in}-\lambda_1^1\right),  \frac{k_2}{k_1k_3}\left(S_1^{in}-F_{1i}\right)\right)$    &
\begin{tabular}{l}
$X_1^2$ is a unique solution of $f_1(x)=g_1(x)$ \\
$X_2^2$ is a solution of $f_3(x)=g_2(x)$
\end{tabular}
\\ \hline
\end{tabular}
\end{center}
\end{table}
\begin{table}[!ht]
\begin{center}
\caption{Existence condition of steady states of system \cref{ModelAM2SR4} where $X_1^{2*}$ is a unique solution of equation $f_1(x)=g_1(x)$. The functions $\lambda_1^i$, $\lambda_2^{1i}$, $\lambda_2^{2i}$, $F_{ij}$ and $\phi_i$, $i,j=1,2$ are defined in \cref{TabFunc1} and the functions $f_1$ and $g_1$ are defined in \cref{TabFunc2}.}\label{TableCondExis}
\begin{tabular}{ll}
\hline
                          & Existence conditions  \\ \Xhline{3\arrayrulewidth}
$\mathcal{E}_{00}^{00}$  & Always exists \\
$\mathcal{E}_{00}^{0i}$  & $S_2^{in}>\lambda_2^{2i}$  \\
$\mathcal{E}_{00}^{10}$  & $S_1^{in}>\lambda_1^2$  \\
$\mathcal{E}_{00}^{1i}$  & $S_1^{in}>\max\left(\lambda_1^2,F_{2i}\right)$\\
$\mathcal{E}_{10}^{10}$  & $S_1^{in}>\lambda_1^1$ \\
$\mathcal{E}_{10}^{1i}$  & $S_1^{in}>\lambda_1^1$ and $\phi_i>0$\\
$\mathcal{E}_{0i}^{01}$  & $S_2^{in}>\lambda_2^{1i}$ \\
$\mathcal{E}_{0i}^{11}$  & $S_1^{in}>\lambda_1^2$ and $S_2^{in}>\lambda_2^{1i}$ \\
$\mathcal{E}_{1i}^{11}$  & $S_1^{in}>\max\left(\lambda_1^1,F_{1i}\right)$
\\ \hline
\end{tabular}
\end{center}
\end{table}
The following proposition determines the existence and multiplicity of steady states of \cref{ModelAM2SR4}.
\begin{proposition}                                 \label{propmultip}
Assume that \cref{Hypoth1,Hypoth2}, and the existence conditions in \cref{TableCondExis} hold for the steady states
$\mathcal{E}_{00}^{kl}$, $\mathcal{E}_{10}^{1l}$, $\mathcal{E}_{0i}^{k1}$ and $\mathcal{E}_{1i}^{11}$, $i=1,2$, $k=0,1$, and $l=0,1,2$.
\begin{itemize}[leftmargin=*]
\item The steady states $\mathcal{E}_{00}^{kl}$, $\mathcal{E}_{10}^{1l}$ and $\mathcal{E}_{01}^{01}$ are unique.
\item There exists at least one steady state of the form $\mathcal{E}_{02}^{01}$. Generically, the system has an odd number of steady states of the form $\mathcal{E}_{02}^{01}$.
\item If $X_2^{1*}> x_2^m$, the steady state $\mathcal{E}_{0i}^{11}$ and $\mathcal{E}_{1i}^{11}$ are unique. If $X_2^{1*}< x_2^m$, then there exists at least one steady state of the form $\mathcal{E}_{0i}^{11}$ and $\mathcal{E}_{1i}^{11}$. Generically, the system has an odd number of steady states of the form $\mathcal{E}_{0i}^{11}$ and $\mathcal{E}_{1i}^{11}$, $i=1,2$.
\end{itemize}
\end{proposition}
\subsection{Local stability of steady states}        \label{SecStabSS}
The following result provides the local stability condition of each steady state of system \cref{ModelAM2SR4}. For convenience, we shall use the abbreviation LES for Locally Exponentially Stable.
\begin{proposition}                                   \label{propstab}
Assume that \cref{Hypoth1,Hypoth2} hold. The steady states $\mathcal{E}_{00}^{02}$, $\mathcal{E}_{00}^{12}$, $\mathcal{E}_{10}^{12}$, $\mathcal{E}_{02}^{01}$, $\mathcal{E}_{02}^{11}$ and $\mathcal{E}_{12}^{11}$ of system \cref{ModelAM2SR4} are unstable when they exist. The conditions stability of all other steady states are given in \cref{TableCondStab}.
\begin{table}[!ht]
\begin{center}
\caption{The stability condition of steady states of system \cref{ModelAM2SR4} except for $\mathcal{E}_{00}^{02}$, $\mathcal{E}_{00}^{12}$, $\mathcal{E}_{10}^{12}$, $\mathcal{E}_{02}^{01}$, $\mathcal{E}_{02}^{11}$ and $\mathcal{E}_{12}^{11}$ which are unstable. The functions $f_1$, $f_3$, $g_1$ and $g_2$ are defined in \cref{TabFunc2} while $X_1^{2*}$ is a unique solution of equation $f_1(x)=g_1(x)$ and $X_2^{2*}$ is a solution of $f_3(x)=g_2(x)$.} \label{TableCondStab}
\begin{tabular}{@{\hspace{2mm}}l@{\hspace{2mm}} @{\hspace{2mm}}l@{\hspace{2mm}} }		
\hline
                         &   Stability conditions        \\ \Xhline{3\arrayrulewidth}
$\mathcal{E}_{00}^{00}$  & $S_1^{in}<\min\left(\lambda_1^1, \lambda_1^2\right)$ and $S_2^{in}\notin\left[\min\left(\lambda_2^{11}, \lambda_2^{21}\right),\max\left(\lambda_2^{12}, \lambda_2^{22}\right)\right]$    \\
$\mathcal{E}_{00}^{01}$  & $S_1^{in}<\min\left(\lambda_1^1, \lambda_1^2\right)$ and $S_2^{in}\notin\left[\lambda_2^{11}, \lambda_2^{12}\right]$    \\
$\mathcal{E}_{00}^{10}$  & $S_1^{in}<\lambda_1^1$, $S_2^{in}\notin\left[\lambda_2^{11}, \lambda_2^{12}\right]$ and $S_2^{in}+\frac{k_2}{k_1}\left(S_1^{in}-\lambda_1^2\right) \notin\left[\lambda_2^{21}, \lambda_2^{22}\right]$     \\
$\mathcal{E}_{00}^{11}$  & $S_1^{in}<\lambda_1^1$ and $S_2^{in}\notin\left[\lambda_2^{11}, \lambda_2^{12}\right]$     \\
$\mathcal{E}_{10}^{10}$  & $S_2^{in}+\frac{k_2}{k_1}\left(S_1^{in}-\lambda_1^1\right) \notin\left[\lambda_2^{11}, \lambda_2^{12}\right]$ and $\phi_1<0$ or $\phi_2>0$\\
$\mathcal{E}_{10}^{11}$  & $S_2^{in}+\frac{k_2}{k_1}\left(S_1^{in}-\lambda_1^1\right) \notin\left[\lambda_2^{11}, \lambda_2^{12}\right]$  \\
$\mathcal{E}_{01}^{01}$  & $S_1^{in}<\min\left(\lambda_1^1, \lambda_1^2\right)$ \\
$\mathcal{E}_{01}^{11}$  & $S_1^{in}<\lambda_1^1$ and $g_2'(X_2^{2*})>f_3'(X_2^{2*})$ \\
$\mathcal{E}_{11}^{11}$  &  $g_2'(X_2^{2*})>f_3'(X_2^{2*})$ \\ \hline
\end{tabular}
\end{center}
\end{table}
\end{proposition}
\section{Operating diagrams}                            \label{Sec_DO}
In this section, we study theoretically the operating diagrams that describe the asymptotic behavior of the process according to four control parameters $D$, $S_1^{in}$, $S_2^{in}$, and $r$ which are the easiest parameters to manipulate in a chemostat.
It is the best tool presented for the biologist to understand the asymptotic behavior of the process according to the operating parameters.
However, since it is difficult to visualize all the regions of the operating diagrams in the four-dimensional space, we fix two parameters and release two others to illustrate these diagrams in the plan.
In \cref{SecDO1s2inFix}, it is analyzed in the two-dimensional plane $\left(D, S_1^{in}\right)$ by fixing the parameter $S_2^{in}$ and $r$.
In \cref{sec2DO}, we fix $D$ and $r$, and we determine the operating diagrams in the plane $\left(S_2^{in},S_1^{in}\right)$. The other cases can be treated similarly.
To determine the different regions in the operating diagram, we must define the surfaces $\gamma_i$, $i=1,\ldots,15$, listed in \cref{TabSurf}.
They are obtained from the existence and stability conditions given in \cref{TableCondExis,TableCondStab}.
They correspond to the boundaries of the change in the number of steady states in the positive octant and/or their local asymptotic behavior.
\begin{table}[!ht]
\begin{center}
\caption{Definitions of the surfaces $\gamma_i$, $i=1,\ldots,15$. The functions $F_{ij}$ and $\phi_j$, $i,j=1,2$ are defined in \cref{TabFunc1} and $X_1^{2*}$ is the unique solution of equation $f_1(x)=g_1(x)$.} \label{TabSurf}
\begin{tabular}{l}
 \hline
 Curve $\gamma_i$   \\ \Xhline{3\arrayrulewidth}
$\gamma_0=\left\{\left(D,S_1^{in},S_2^{in},r\right): S_1^{in}=\lambda_1^1(D,r),~ D\in\left[0,r_1m_1\right)\right\}$\\
$\gamma_1=\left\{\left(D,S_1^{in},S_2^{in},r\right): S_1^{in}=\lambda_1^2(D,r),~ D\in\left[0,r_2m_1\right)\right\}$\\
$\gamma_2=\left\{\left(D,S_1^{in},S_2^{in},r\right): S_1^{in}=F_{21}\left(D,r,S_2^{in}\right),~ D\in\left(0,r_2\min\left(m_1,\mu_2\left(S_2^m\right)\right)\right)\right\}$\\
$\gamma_3=\left\{\left(D,S_1^{in},S_2^{in},r\right): S_1^{in}=F_{22}\left(D,r,S_2^{in}\right),~ D\in\left(0,r_2\min\left(m_1,\mu_2\left(S_2^m\right)\right)\right)\right\}$\\
$\gamma_4=\left\{\left(D,S_1^{in},S_2^{in},r\right): S_1^{in}=F_{11}\left(D,r,S_2^{in}\right),~ D\in\left(0,r_1\min\left(m_1,\mu_2\left(S_2^m\right)\right)\right)\right\}$\\
$\gamma_5=\left\{\left(D,S_1^{in},S_2^{in},r\right): S_1^{in}=F_{12}\left(D,r,S_2^{in}\right),~ D\in\left(0,r_1\min\left(m_1,\mu_2\left(S_2^m\right)\right)\right)\right\}$\\
$\gamma_6=\left\{\left(D,S_1^{in},S_2^{in},r\right): D=D_1^{*}:=r_1\mu_2\left(S_2^{in}\right)\right\}$ \\
$\gamma_7=\left\{\left(D,S_1^{in},S_2^{in},r\right): D=D_2^{*}:=r_2\mu_2\left(S_2^{in}\right)\right\}$\\
$\gamma_8=\left\{\left(D,S_1^{in},S_2^{in},r\right): D=D_1^m:=r_1\mu_2\left(S_2^m\right)\right\}$ \\
$\gamma_9=\left\{\left(D,S_1^{in},S_2^{in},r\right): D=D_2^m:=r_2\mu_2\left(S_2^m\right)\right\}$\\
\hline
$\gamma_{10}=\left\{\left(D,S_1^{in},S_2^{in},r\right): S_2^{in}=S_{21}^{in*}=\lambda_2^{11}(D,r),~ D<D_1^m\right\}$   \\
$\gamma_{11}=\left\{\left(D,S_1^{in},S_2^{in},r\right): S_2^{in}={S_{22}^{in*}}=\lambda_2^{21}(D,r),~ D<D_2^m\right\}$ \\
$\gamma_{12}=\left\{\left(D,S_1^{in},S_2^{in},r\right): S_2^{in}={S_{23}^{in*}}=\lambda_2^{12}(D,r),~ D<D_1^m\right\}$   \\
$\gamma_{13}=\left\{\left(D,S_1^{in},S_2^{in},r\right): S_2^{in}={S_{24}^{in*}}=\lambda_2^{22}(D,r),~ D<D_2^m\right\}$\\
$\gamma_{14}=\left\{\left(D,S_1^{in},S_2^{in},r\right):\phi_1\left(D,r,S_1^{in},S_2^{in}\right)=0,~D \leq D_2^m~\mbox{and}~S_1^{in}>\lambda_1^1\right\}$\\
$\gamma_{15}=\left\{\left(D,S_1^{in},S_2^{in},r\right):\phi_2\left(D,r,S_1^{in},S_2^{in}\right)=0,~D \leq D_2^m~\mbox{and}~S_1^{in}>\lambda_1^1\right\}$
\\ \hline
\end{tabular}
\end{center}
\end{table}
\begin{remark}
\begin{enumerate}
    \item Note that for all $D<D_1^m$, (or equivalently $D<r_1\mu_2\left(S_2^m\right)$), the equation $\mu_2\left(S_2^{in}\right)=D_1:=D/r_1$ is equivalent to $S_2^{in}=\lambda_2^{11}$ or $S_2^{in}=\lambda_2^{12}$. From \cref{TabSurf}, we see that $\gamma_6=\gamma_{10} \cup \gamma_{12}$.
Similarly, for all $D<D_2^m$, (or equivalently $D<r_2\mu_2\left(S_2^m\right)$), the equation $\mu_2\left(S_2^{in}\right)=D_2:=D/r_2$ is equivalent to $S_2^{in}=\lambda_2^{21}$ or $S_2^{in}=\lambda_2^{22}$. From \cref{TabSurf}, we see that $\gamma_7=\gamma_{11} \cup \gamma_{13}$.
\item  Let $r$ be fixed.
When $S_2^{in}$ is fixed, the curves $\gamma_6$ and $\gamma_7$ are vertical lines in the plan $\left(D,S_1^{in}\right)$ with equation $D=D_1^{*}:=r_1\mu_2\left(S_2^{in}\right)$ and $D=D_2^{*}:=r_2\mu_2\left(S_2^{in}\right)$, respectively.
When $D$ is fixed, the curve $\gamma_i$, $i=10,\ldots,13$ is a vertical line in the plan $\left(S_2^{in},S_1^{in}\right)$ with equation $S_2^{in}=S_{2j}^{in*}$, $j=1,\ldots,4$.
\end{enumerate}
\end{remark}
Using the definitions of the break-even concentrations $\lambda_1^i$ and $\lambda_2^{ij}$, $i,j=1,2$, the following proposition determines the relative positions of these critical values according to the distribution of the total volume $r\in(0,1)$.
\begin{lemma}                             \label{LemmaPositionLambdai}
Assume that \cref{Hypoth1,Hypoth2} hold.
\begin{itemize}
    \item If $r\in (0,1/2)$, then $\lambda_1^2(D,r)<\lambda_1^1(D,r)$, $\lambda_2^{21}(D,r)<\lambda_2^{11}(D,r)<\lambda_2^{12}(D,r)<\lambda_2^{22}(D,r)$.
    \item If $r\in (1/2,1)$, then $\lambda_1^1(D,r)<\lambda_1^2(D,r)$, $\lambda_2^{11}(D,r)<\lambda_2^{21}(D,r)<\lambda_2^{22}(D,r)<\lambda_2^{12}(D,r)$.
    \item If $r=1/2$,        then $\lambda_1^1(D,r)=\lambda_1^2(D,r)$, $\lambda_2^{11}(D,r)=\lambda_2^{21}(D,r)$, $\lambda_2^{12}(D,r)=\lambda_2^{22}(D,r)$.
\end{itemize}
\end{lemma}
\subsection\protect{Operating diagram in the plane $\left(D, S_1^{in}\right)$ when $S_2^{in}$ and $r$ are fixed.}  \label{SecDO1s2inFix}
Fix $r$ and $S_2^{in}$. In this case, the surfaces $\gamma_i$ correspond to curves for $i=0,\dots,5$ and vertical lines for $i=6,\dots,9$. Depending on the critical values $m_1$, $\mu_2\left(S_2^{in}\right)$, and $\mu_2\left(S_2^m\right)$, the following three cases should be distinguished in the operating diagram study when $r \in (0,1/2)$ or $r \in (1/2,1)$:
\begin{equation}                                      \label{Cases123}
\mbox{case 1: } m_1 >\mu_2\left(S_2^m\right), \quad \mbox{case 2: } \mu_2\left(S_2^{in}\right) > m_1, \quad \mbox{case 3: } \mu_2\left(S_2^m\right)>m_1 >\mu_2\left(S_2^{in}\right).
\end{equation}
Note that for each case in \cref{Cases123}, $S_2^{in}$ may be less than or greater than $S_2^m$.
In this section, we present cases 1 and 2 for $r \in (0,1/2)$ and case 1 for $r \in (1/2,1)$. However, the study is similar in the other cases.

From the definitions of the auxiliary functions in \cref{TabFunc1}, the following proposition establishes the relative positions of surfaces $\gamma_i$, $i=0,\dots,5$ according to the critical values $D_1^{*}$ and $D_2^{*}$ defined in \cref{TabSurf}.
\begin{proposition}                           \label{propPositionCurv}
Let $r\in(0,1)$. If $m_1>\mu_2\left(S_2^{in}\right)$, then the curve $\gamma_0$ and the vertical line $\gamma_6$ intersect at the point $P_1\left(S_2^{in},r\right)=\left(D_1^*,\lambda_1^1\left(D_1^*,r\right)\right)$ in the plane $\left(D,S_1^{in}\right)$ (see \cref{FigDO1} and \cref{CurvFigDO1}(a)). If, in addition, $S_2^{in} \geq S_2^m$, then $\gamma_0$, $\gamma_5$, and $\gamma_6$ intersect at the same point $P_1$. Otherwise, when $S_2^{in}<S_2^m$, then $\gamma_0$, $\gamma_4$, and $\gamma_6$ intersect at the same point $P_1$. Moreover:
\begin{itemize}
    \item If $D<D_1^*$, then $F_{11}\left(D,r,S_2^{in}\right) < \lambda_1^1(D,r) < F_{12}\left(D,r,S_2^{in}\right)$.
    \item If $D>D_1^*$ and $S_2^{in} \geq S_2^m$, then $F_{11}\left(D,r,S_2^{in}\right) < F_{12}\left(D,r,S_2^{in}\right) < \lambda_1^1(D,r)$.
    \item If $D>D_1^*$ and $S_2^{in} < S_2^m$, then $\lambda_1^1(D,r) < F_{11}\left(D,r,S_2^{in}\right) < F_{12}\left(D,r,S_2^{in}\right)$.
\end{itemize}
If $m_1>\mu_2\left(S_2^{in}\right)$, the curve $\gamma_1$ and the vertical line $\gamma_7$ intersect at  $P_2\left(S_2^{in},r\right)=\left(D_2^*,\lambda_1^2\left(D_2^*,r\right)\right)$ in the plane $\left(D,S_1^{in}\right)$ (see \cref{FigDO1} and \cref{CurvFigDO1}(a)). If, in addition, $S_2^{in} \geq S_2^m$, then $\gamma_1$, $\gamma_3$, and $\gamma_7$ intersect at the same point $P_2$. Otherwise, when $S_2^{in}<S_2^m$, then $\gamma_1$, $\gamma_2$, and $\gamma_7$ intersect at the same point $P_2$. Moreover:
\begin{itemize}
    \item If $D<D_2^*$, then $F_{21}\left(D,r,S_2^{in}\right) < \lambda_1^2(D,r) < F_{22}\left(D,r,S_2^{in}\right)$.
    \item If $D>D_2^*$ and $S_2^{in} \geq S_2^m$, then $F_{21}\left(D,r,S_2^{in}\right) < F_{22}\left(D,r,S_2^{in}\right) < \lambda_1^2(D,r)$.
    \item If $D>D_2^*$ and $S_2^{in} < S_2^m$, then $\lambda_1^2(D,r) < F_{21}\left(D,r,S_2^{in}\right) < F_{22}\left(D,r,S_2^{in}\right)$.
\end{itemize}
\end{proposition}
The asymptotic behavior of system \cref{ModelAM2S2C} varies according to the value of $r$ in the intervals $(0,1/2)$ and $(1/2,1)$ such that the position of the various curves changes with $r$.
The following proposition establishes the relative positions of the curves $\gamma_2$ and $\gamma_4$ when $S_2^{in}>S_2^m$.
\begin{proposition}                                 \label{PropF21F11}
Let $S_2^{in}>S_2^m$.
\begin{itemize}
    \item If $r\in(0,1/2)$, then $F_{21}\left(D,r,S_2^{in}\right)<F_{11}\left(D,r,S_2^{in}\right)$.
    \item If $r\in(1/2,1)$, then $F_{21}\left(D,r,S_2^{in}\right)>F_{11}\left(D,r,S_2^{in}\right)$.
\end{itemize}
\end{proposition}

Now to construct the operating diagrams, the growth rates $\mu_1$ and $\mu_2$ are chosen of Monod and Haldane type which are written:
\begin{equation}                                    \label{SpeciFunc}
\mu_1(S_1) =\frac{m_1 S_1}{k_{S_1}+S_1} \quad \mbox{and} \quad \mu_2(S_2) =\frac{m_2S_2}{k_{S_2}+S_2+\frac{(S_2)^2}{k_I}},
\end{equation}
where $m_i$ is the maximum growth rate and $k_{S_i}$ is the half-saturation (or Michaelis-Menten) constant associated to $S_i$, for $i = 1,2$. However, the construction can be applied to any growth rate verifying \cref{Hypoth1,Hypoth2}.
The biological parameter values used in all figures are provided in \cref{TabParamVal}.
Note that we have used the same specific growth rates of Monod and Haldane types \cref{SpeciFunc} and parameter values as in \cite{Bernard2001}, with the exception of the maximum growth rate $m_1$ to have a richer process behavior.
\begin{proposition}                                \label{propstabstabDO1}
Assume that \cref{Hypoth1,Hypoth2} hold. Let $r$ and $S_2^{in}$ be fixed. Let $\mu_1$ and $\mu_2$ be the specific growth rates defined in \cref{SpeciFunc} and the set of the biological parameter values in \cref{TabParamVal}.
Let $\mathcal{J}_i$, $i=0,\dots,69$ be the various regions defined in \cref{TabRegDOJi}.
The existence and the local stability properties of steady states of the AM2 model with two interconnected chemostats \cref{ModelAM2S2C} in the operating diagram are defined as follows:
\begin{enumerate}
    \item If $r\in(0,1/2)$, $S_2^{in}>S_2^m$ and Case 1 of \cref{Cases123} holds $\left(m_1>\mu_2\left(S_2^m\right)\right)$, there are twenty-one regions $\mathcal{J}_i$, $i=0-20$ shown in \cref{FigDO1}.
    \item If $r\in(0,1/2)$, $S_2^{in}>S_2^m$ and Case 2 of \cref{Cases123} holds $\left(m_1<\mu_2\left(S_2^{in}\right)\right)$, there are seventeen regions $\mathcal{J}_i$, $i=0,3-5,8,13-24$ shown in \cref{FigDO2}.
    \item If $r\in(0,1/2)$, $S_2^{in}<S_2^m$ and Case 1 of \cref{Cases123} holds, there are twenty-one regions $\mathcal{J}_i$, $i=0,1,4-9,15-20,25-31$ shown in \cref{FigDO3}.
    \item If $r\in(1/2,1)$, $S_2^{in}>S_2^m$ and Case 1 of \cref{Cases123} holds, there are twenty regions $\mathcal{J}_i$, $i=0,15,20-22,32-46$ shown in \cref{FigDO4}.
\end{enumerate}
The conditions on the operating parameters defining the various boundaries of the regions are described in \cref{TabRegionCondDO123} [resp. \cref{TabRegionCondDO4}] for the first three items [last item].
\end{proposition}
\begin{figure}[!ht]
\setlength{\unitlength}{1.0cm}
\begin{center}
\begin{picture}(7,6)(0,0)
\put(-4.5,0){\rotatebox{0}{\includegraphics[width=7.5cm,height=5.5cm]{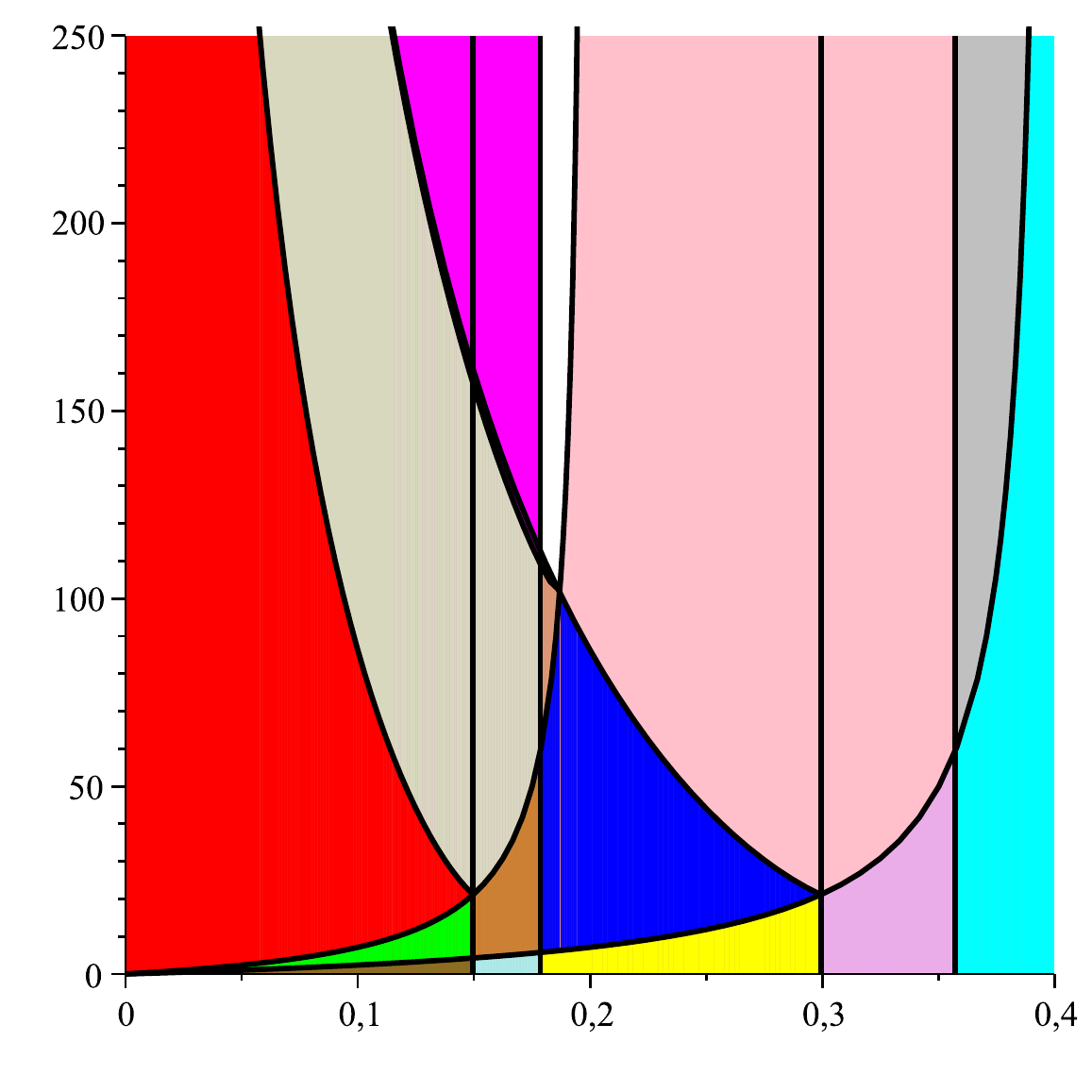}}}
\put(3.5,0){\rotatebox{0}{\includegraphics[width=7.5cm,height=5.5cm]{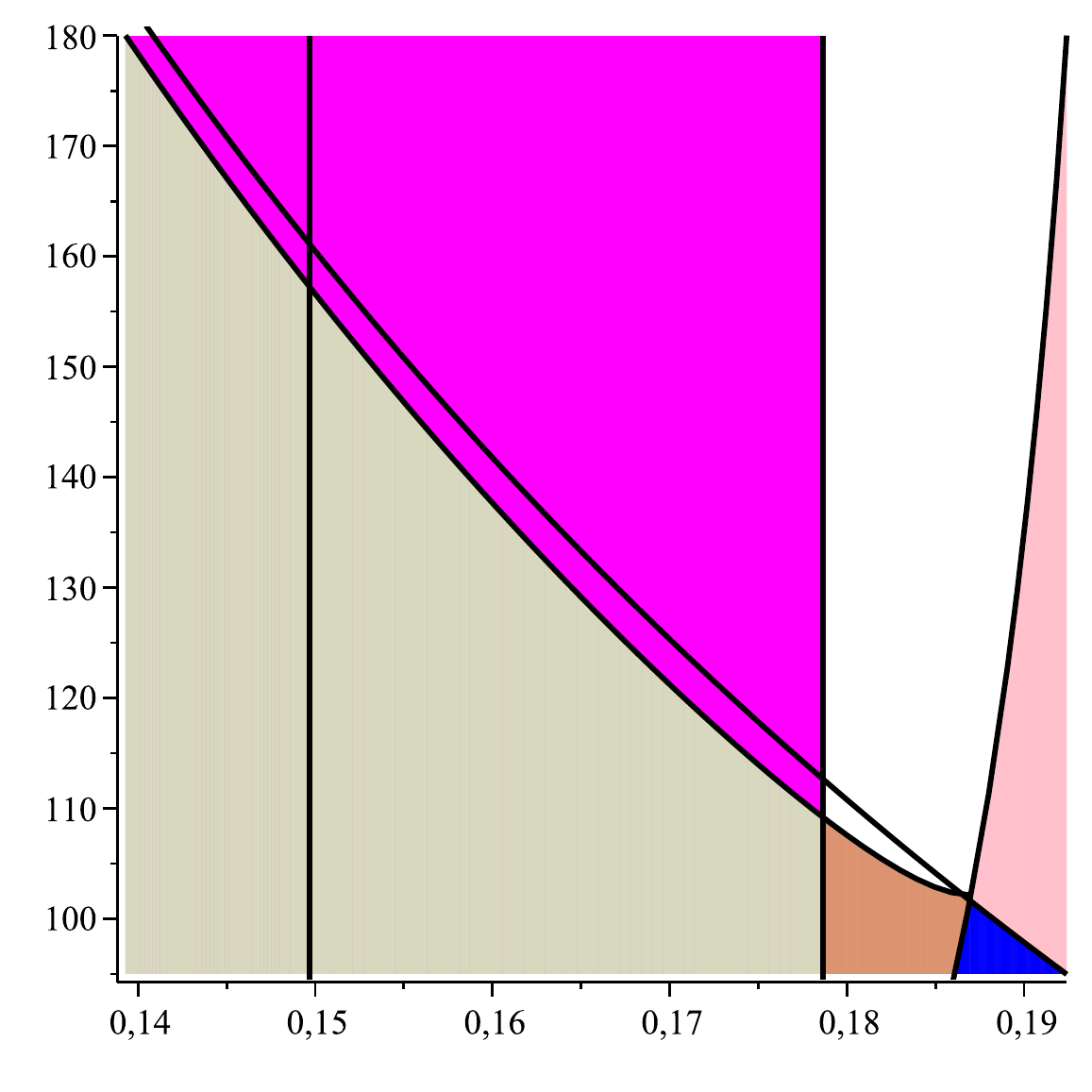}}}
\put(-1,5.8){\sc $(a)$}	
\put(2.2,0.7){\sc { $\mathcal{J}_0$}}
\put(2,3.7){\sc { $\mathcal{J}_1$}}
\put(1.3,3.7){\sc { $\mathcal{J}_2$}}
\put(1.45,0.7){\sc { $\mathcal{J}_3$}}
\put(0.6,0.67){\sc { $\mathcal{J}_4$}}
\put(-0.5,1.1){\sc { \color{white}$\mathcal{J}_5$}}
\put(-1.08,2.35){\sc { $\rightarrow$}}
\put(-1.35,2.3){\sc { $\mathcal{J}_6$}}
\put(0.2,3.7){\sc { $\mathcal{J}_8$}}
\put(-0.8,3.7){\sc { $\leftarrow$}}
\put(-0.5,3.7){\sc { $\mathcal{J}_9$}}
\put(-0.85,2.55){\sc { $\leftarrow$}}
\put(-0.5,2.55){\sc { $\mathcal{J}_7$}}
\put(-1.25,3.7){\sc {\color{white} $\mathcal{J}_{10}$}}
\put(-1.25,1.7){\sc {$\mathcal{J}_{12}$}}
\put(-1.25,0.8){\sc {\color{white}$\mathcal{J}_{13}$}}
\put(-1.1,0.45){\sc $\uparrow$}
\put(-1.25,0.15){\sc { $\mathcal{J}_{14}$}}
\put(-1.8,0.45){\sc $\uparrow$}	
\put(-2,0.15){\sc { $\mathcal{J}_{15}$}}
\put(-1.8,0.8){\sc {\color{white} $\searrow$}}
\put(-2.2,1.1){\sc {\color{white} $\mathcal{J}_{16}$}}
\put(-3.4,3.5){\sc { \color{white}$\mathcal{J}_{17}$}}
\put(-1.75,3.75){\sc { $\rightarrow$}}
\put(-2.1,3.7){\sc { $\mathcal{J}_{19}$}}
\put(-1.5,3.15){\sc { $\rightarrow$}}
\put(-1.9,3.1){\sc { $\mathcal{J}_{11}$}}
\put(-2.5,4.2){\sc { $\mathcal{J}_{18}$}}
\put(-1.85,5){\sc { \color{white}$\mathcal{J}_{20}$}}
\put(-3.8,5.45){\sc $S_1^{in}$}
\put(2.75,0.55){\sc  $D$}
\put(-0.75,5.45){\sc { $\gamma_0$}}
\put(0.9,5.45){\sc { $\gamma_7$}}
\put(-1.5,5.45){\sc { $\gamma_6$}} 	
\put(-1.9,5.45){\sc { $\gamma_3$}}
\put(-2.3,5.45){\sc { $\gamma_{15}$}}
\put(-1.05,5.45){\sc { $\gamma_8$}}
\put(1.8,5.45){\sc { $\gamma_9$}}
\put(2.35,5.45){\sc { $\gamma_1$}}
\put(-3,5.45){\sc { $\gamma_5$}}
\put(6.8,5.8){{\sc $(b)$}}
\put(9.4,0.65){\sc { $\mathcal{J}_6$}}
\put(9.6,3.7){\sc { $\mathcal{J}_9$}}
\put(10.3,1.5){\sc { $\mathcal{J}_8$}}
\put(10,0.65){\sc { \color{white}$\mathcal{J}_5$}}
\put(9.2,1.4){\sc { $\swarrow$}}
\put(9.4,1.7){\sc { $\mathcal{J}_7$}}
\put(8,3.7){\sc { \color{white}$\mathcal{J}_{10}$}}
\put(6.8,1.5){\sc {$\mathcal{J}_{12}$}}
\put(4.8,2.2){\sc {$\mathcal{J}_{18}$}}
\put(4.95,4.35){\sc { $\rightarrow$}}
\put(4.5,4.3){\sc { $\mathcal{J}_{19}$}}
\put(6.6,2.9){\sc { $\rightarrow$}}
\put(6.1,2.9){\sc { $\mathcal{J}_{11}$}}
\put(5,5){\color{white}\sc { $\mathcal{J}_{20}$}}
\put(4,5.5){\sc $S_1^{in}$}
\put(10.9,0.5){\sc  $D$}
\put(10.6,5.45){\sc { $\gamma_0$}}
\put(8.8,5.45){\sc { $\gamma_8$}}
\put(5.3,5.45){\sc { $\gamma_6$}} 	
\put(4.4,5.45){\sc { $\gamma_3$}}
\put(4.1,5){\sc { $\nearrow$}}
\put(3.7,4.9){\sc { $\gamma_{15}$}}
\end{picture}
\vspace{-0.5cm}
\caption{Operating diagram of \cref{ModelAM2S2C} with twenty-one regions $\mathcal{J}_i$, $i=0-20$ when $r=1/3\in(0,1/2)$ and $S_2^{in}=150>S_2^m\approx48.740$. Case 1 of \cref{Cases123}: $m_1=0.6>\mu_2\left(S_2^m\right)\approx0.535$. }\label{FigDO1}
\end{center}
\end{figure}
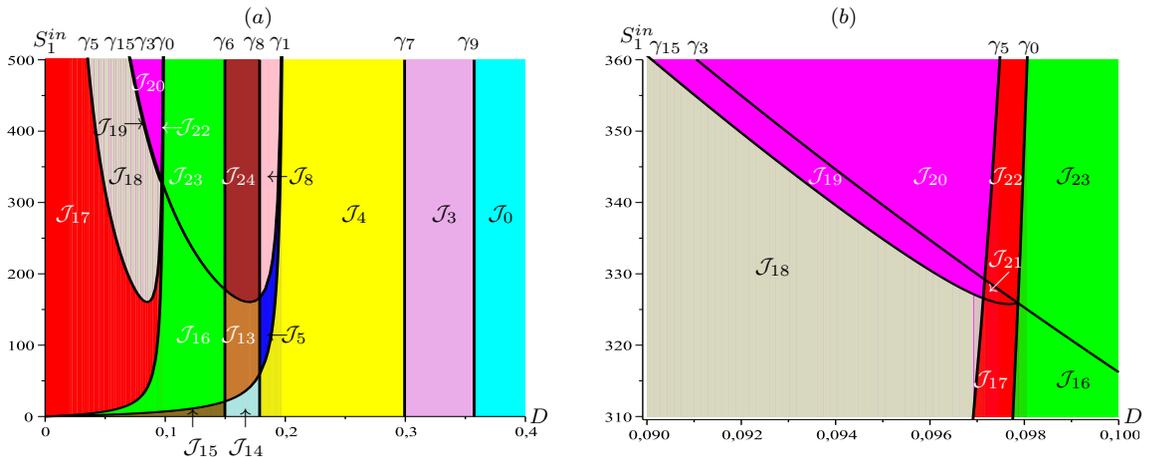
\begin{figure}[!ht]
\setlength{\unitlength}{1.0cm}
\begin{center}
\begin{picture}(7,5.6)(0,0)
\put(-4.5,0){\rotatebox{0}{\includegraphics[width=7.5cm,height=5.5cm]{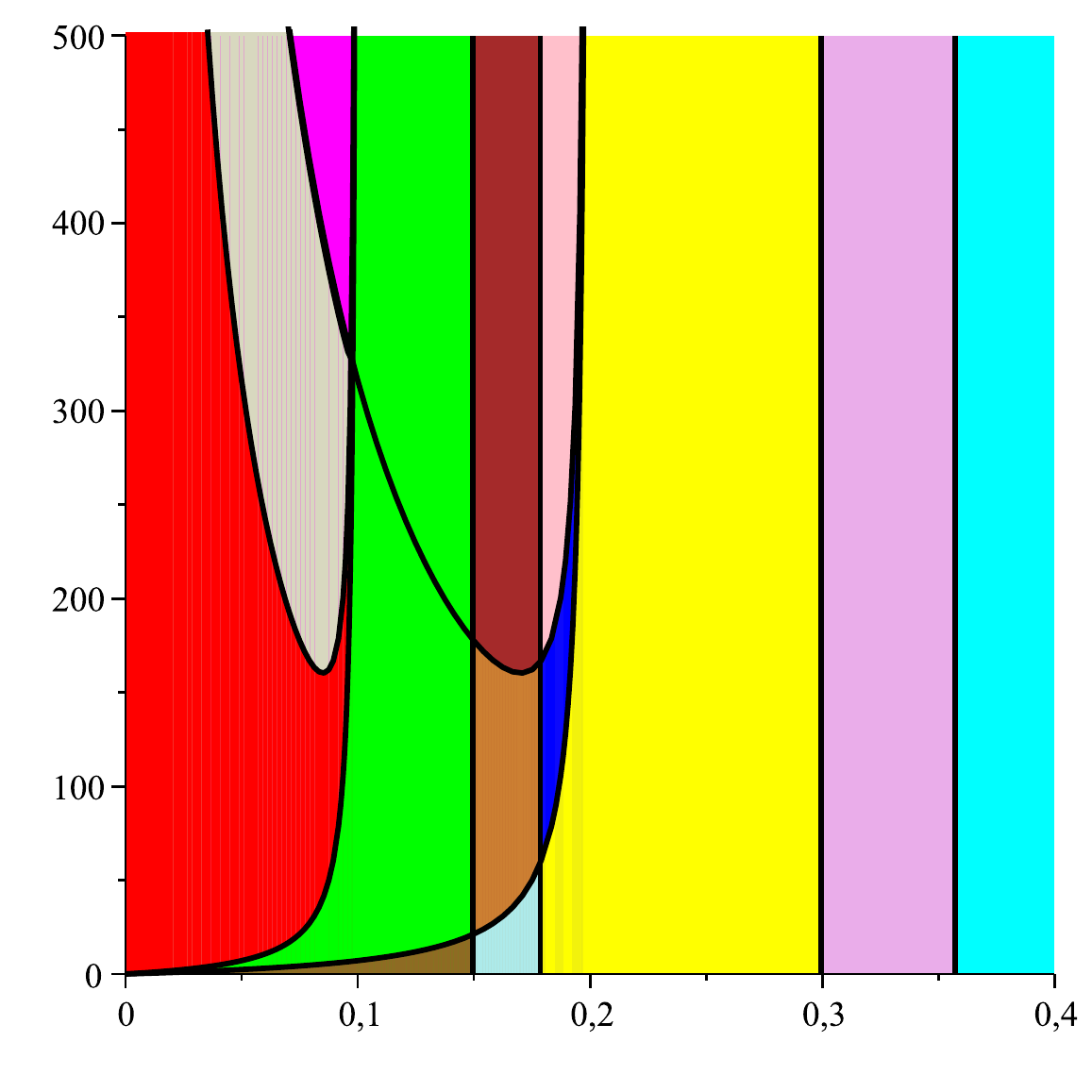}}}
\put(3.5,0){\rotatebox{0}{\includegraphics[width=7.5cm,height=5.5cm]{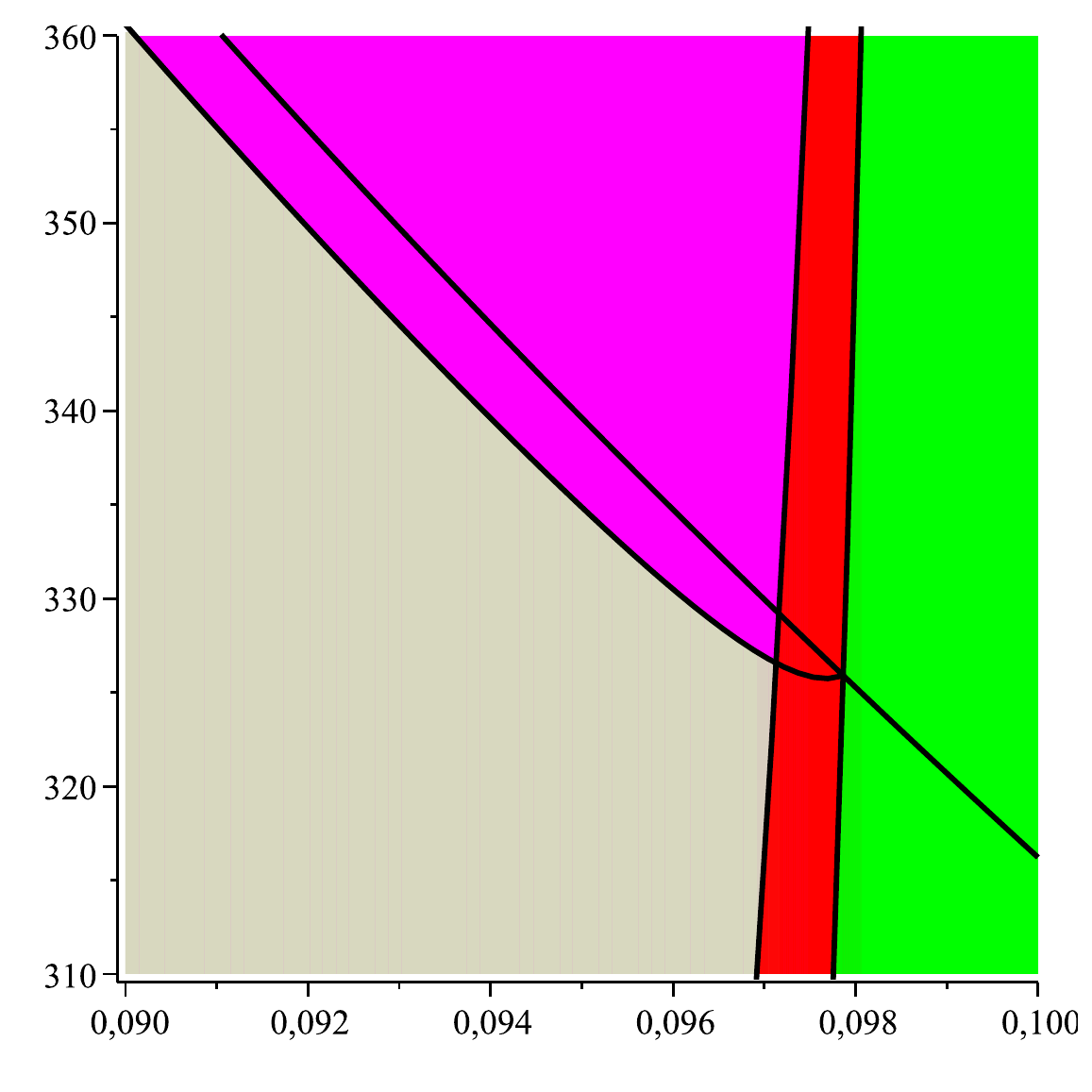}}}
\put(-1,5.8){\sc $(a)$}	
\put(2.15,3.2){\sc { $\mathcal{J}_0$}}
\put(1.4,3.2){\sc { $\mathcal{J}_3$}}
\put(0.2,3.2){\sc { $\mathcal{J}_4$}}
\put(-0.8,1.6){\sc { $\leftarrow$}}
\put(-0.6,1.6){\sc { $\mathcal{J}_5$}}
\put(-0.8,3.7){\sc { $\leftarrow$}}
\put(-0.5,3.7){\sc { $\mathcal{J}_8$}}
\put(-1.3,3.7){\sc {\color{white} $\mathcal{J}_{24}$}}
\put(-1.4,1.6){\sc { $\color{white}\mathcal{J}_{13}$}}
\put(-1.05,0.5){\sc $\uparrow$}
\put(-1.3,0.1){\sc { $\mathcal{J}_{14}$}}
\put(-1.75,0.5){\sc $\uparrow$}	
\put(-1.9,0.1){\sc { $\mathcal{J}_{15}$}}
\put(-2,1.6){\sc { $\color{white}\mathcal{J}_{16}$}}
\put(-2,3.7){\sc {$\color{white}\mathcal{J}_{23}$}}
\put(-2.6,4.95){\sc { \color{white}$\mathcal{J}_{20}$}}
\put(-2.2,4.35){\sc { $\color{white}\leftarrow$}}
\put(-2,4.35){\sc { $\color{white}\mathcal{J}_{22}$}}
\put(-2.7,4.4){\sc { $\rightarrow$}}
\put(-3,4.35){\sc {$\mathcal{J}_{19}$}}
\put(-2.8,3.7){\sc {$\mathcal{J}_{18}$}}
\put(-3.5,3.2){\sc {\color{white} $\mathcal{J}_{17}$}}
\put(-3.8,5.5){\sc $S_1^{in}$}
\put(2.8,0.5){\sc  $D$}
\put(-2.3,5.5){\sc { $\gamma_0$}}
\put(-1.1,5.5){\sc { $\gamma_8$}}
\put(-1.5,5.5){\sc { $\gamma_6$}} 	
\put(-2.55,5.5){\sc { $\gamma_3$}}
\put(-2.95,5.5){\sc { $\gamma_{15}$}}
\put(0.9,5.5){\sc { $\gamma_7$}}
\put(1.75,5.5){\sc { $\gamma_9$}}
\put(-3.3,5.5){\sc { $\gamma_5$}}
\put(-0.75,5.5){\sc { $\gamma_1$}}
\put(6.8,5.8){{\sc $(b)$}}
\put(9.7,1){\sc { $\mathcal{J}_{16}$}}
\put(9.7,3.7){\sc { $\mathcal{J}_{23}$}}
\put(8.6,1){\sc { \color{white}$\mathcal{J}_{17}$}}
\put(8.9,2.3){\sc {\color{white} $\swarrow$}}
\put(8.9,2.6){\sc {\color{white} $\mathcal{J}_{21}$}}
\put(7.8,3.7){\sc { \color{white}$\mathcal{J}_{20}$}}
\put(6.4,3.7){\sc { \color{white}$\mathcal{J}_{19}$}}
\put(8.8,3.7){\sc { \color{white}$\mathcal{J}_{22}$}}
\put(5.8,2.5){\sc {$\mathcal{J}_{18}$}}
\put(4,5.55){\sc $S_1^{in}$}
\put(10.7,0.5){\sc  $D$}
\put(8.8,5.45){\sc { $\gamma_5$}}
\put(9.2,5.45){\sc { $\gamma_0$}} 	
\put(4.8,5.45){\sc { $\gamma_3$}}
\put(4.3,5.45){\sc { $\gamma_{15}$}}
\end{picture}
\vspace{-0.5cm}
\caption{Operating diagram of \cref{ModelAM2S2C} with seventeen regions $\mathcal{J}_i$, $i=0,3-5,8,13-24$ when $r=1/3\in(0,1/2)$ and $S_2^{in}=150>S_2^m\approx48.740$. Case 2 of \cref{Cases123}: $m_1=0.3<\mu_2\left(S_2^{in}\right)\approx0.449$.}\label{FigDO2}
\end{center}
\end{figure}
\begin{table}[!ht]
{\scriptsize
\begin{center}
\caption{Existence and stability of steady states in the regions regions $\mathcal{J}_i$, $i=0,\dots,69$ of the operating diagrams in \cref{FigDO1,FigDO2,FigDO3,FigDO4,FigDO5Dfix}.} \label{TabRegDOJi}
\vspace{-0.5cm}
\begin{tabular}{ @{\hspace{1mm}}l@{\hspace{1mm}} @{\hspace{1mm}}c@{\hspace{1mm}} @{\hspace{1mm}}l@{\hspace{1mm}} @{\hspace{1mm}}l@{\hspace{1mm}} @{\hspace{1mm}}l@{\hspace{1mm}}
                 @{\hspace{1mm}}l@{\hspace{1mm}} @{\hspace{1mm}}l@{\hspace{1mm}} @{\hspace{1mm}}l@{\hspace{1mm}} @{\hspace{1mm}}l@{\hspace{1mm}} @{\hspace{1mm}}l@{\hspace{1mm}}
                 @{\hspace{1mm}}l@{\hspace{1mm}} @{\hspace{1mm}}l@{\hspace{1mm}} @{\hspace{1mm}}l@{\hspace{1mm}} @{\hspace{1mm}}l@{\hspace{1mm}} @{\hspace{1mm}}l@{\hspace{1mm}}
                 @{\hspace{1mm}}l@{\hspace{1mm}} @{\hspace{1mm}}l@{\hspace{1mm}} @{\hspace{1mm}}l@{\hspace{1mm}} @{\hspace{1mm}}l@{\hspace{1mm}} }
 & $\mathcal{E}_{00}^{00}$ & $\mathcal{E}_{00}^{01}$ & $\mathcal{E}_{00}^{02}$ & $\mathcal{E}_{00}^{10}$ & $\mathcal{E}_{00}^{11}$ & $\mathcal{E}_{00}^{12}$ & $\mathcal{E}_{10}^{10}$ & $\mathcal{E}_{10}^{11}$ & $\mathcal{E}_{10}^{12}$ & $\mathcal{E}_{01}^{01}$ & $\mathcal{E}_{02}^{01}$ & $\mathcal{E}_{01}^{11}$ & $\mathcal{E}_{02}^{11}$ & $\mathcal{E}_{11}^{11}$ & $\mathcal{E}_{12}^{11}$ & Color \\ \hline
$\mathcal{J}_0$   & S &   &   &   &   &   &   &   &   &   &   &   &   &   &    &   Cyan     \\ \hline
$\mathcal{J}_1$   & U &   &   & S &   &   &   &   &   &   &   &   &   &   &    &   Grey     \\ \hline
$\mathcal{J}_2$   & U &  U & U & S & S & U &   &   &   &   &   &   &   &   &    &   Pink     \\ \hline
$\mathcal{J}_3$   & S & S & U &   &   &   &   &   &   &   &   &   &   &   &    &  Plum   \\ \hline
$\mathcal{J}_4$   & U & S &   &   &   &   &   &   &   &   &   &   &   &   &    &   Yellow     \\ \hline
$\mathcal{J}_5$   & U & U &   & U & S &   &   &   &   &   &   &   &   &   &    &   Blue   \\ \hline
$\mathcal{J}_6$   & U & U &   & U & U &   & U & S &  &   &   &   &   &   &    &   Tan      \\ \hline
$ \mathcal{J}_7$   & U & U &   & U & U &  & S & S & U &   &   &   &   &   &    &  White      \\ \hline
$\mathcal{J}_8$   & U & U &   & S & S & U &   &   &   &   &   &   &   &   &    &  Pink    \\ \hline
$\mathcal{J}_9$   & U & U &   & U & U & U & S & S & U &   &   &   &   &   &    &  White    \\ \hline
$\mathcal{J}_{10}$   & U & U &   & U & U & U & S & S & U  &  U & U &  U & U & S &  U &   Magenta   \\ \hline
$\mathcal{J}_{11}$   & U &  U &   & U & U &   & S & S & U &  U  & U &  U  & U & S &  U &   Magenta   \\ \hline
$\mathcal{J}_{12}$& U & U &   & U & U &   & U & S &   &  U & U &  U & U & S &  U &   Wheat  \\ \hline
$\mathcal{J}_{13}$& U & U &   & U & S &   &   &   &   &  U & U & S & U &   &    &   Gold  \\ \hline
$\mathcal{J}_{14}$& U & S &   &   &   &   &   &   &   &   S & U &   &   &   &    &   Turquoise  \\ \hline
$\mathcal{J}_{15}$& U & U &   &   &   &   &   &   &   & S &   &   &   &   &    &   Sienna    \\ \hline
$\mathcal{J}_{16}$& U & U &   & U & U &   &   &   &   & U &   & S &   &   &    &   Green  \\ \hline
$\mathcal{J}_{17}$& U & U &   & U & U &   & U & U &   & U &   & U &   & S &    &   Red  \\ \hline
$\mathcal{J}_{18}$& U & U &   & U & U &   & U & S &   & U &   & U &   & S &  U &   Wheat    \\ \hline
$\mathcal{J}_{19}$& U & U &   & U & U &   & S & S & U & U &   & U &   & S &  U &   Magenta   \\ \hline
$\mathcal{J}_{20}$& U & U &   & U & U & U & S & S & U & U &   & U &   & S &  U &   Magenta   \\ \Xhline{3\arrayrulewidth}
$\mathcal{J}_{21}$& U & U &   & U & U &   & U & U & U  & U &   & U &   & S &    &   Red  \\ \hline
$\mathcal{J}_{22}$& U & U &   & U & U & U & U & U & U  & U &   & U &   & S &    &   Red  \\ \hline
$\mathcal{J}_{23}$& U & U &   & U & U & U &   &   &   & U &   & S &   &   &    & Green   \\ \hline
$\mathcal{J}_{24}$& U & U &   & S & S & U &   &   &   &   U  & U &  S  & U &   &    &   Brown    \\
\hline
$\mathcal{J}_{25}$& U &   &   & U & S &   &   &   &   &   &   &   &   &   &    &   Blue    \\ \hline
$\mathcal{J}_{26}$& U &   &   & S & S & U &   &   &   &   &   &   &   &   &    &    Pink   \\ \hline
$\mathcal{J}_{27}$& U & U &   & U & U & U & S & S & U &   &   &   &   & S &  U &   Magenta  \\ \hline
$\mathcal{J}_{28}$   & U &  U &   & U & U &   & S & S & U &    & U &   & U & S &  U &   Magenta   \\ \hline
$\mathcal{J}_{29}$& U & U &   & U & U &   & U & S &   &   &   &   &   & S &  U &   Wheat  \\ \hline
$\mathcal{J}_{30}$& U & U &   & U & U &   & U & U &   &   &   &   &   & S &    &   Red      \\ \hline
$\mathcal{J}_{31}$& U & U &   & U & U &   & S & S &   &   &   &   &   &   &    &   White    \\ \Xhline{3\arrayrulewidth}
$\mathcal{J}_{32}$   & U &   &   &   &   &   &  S &   &   &   &   &   &   &   &    &  Violet     \\ \hline
$\mathcal{J}_{33}$   & U &   &   &   &   &   & S  &   &   &  U & U &   &   &  S &  U  &   Navy     \\ \hline
$\mathcal{J}_{34}$   & S &   &   &   &   &   &   &   &   & S &  U &   &   &   &    & Khaki   \\ \hline
$\mathcal{J}_{35}$   & U &  &   &   &   &   &   &   &   & S  &   &   &   &   &    &   Sienna\\ \hline
$\mathcal{J}_{36}$   & U &   &   &   &   &   & U &   &   &  U &   &   &   &  S &    &   Red  \\ \hline
$\mathcal{J}_{37}$   & U &   &   & U & U &  U & U &   &   &  U &   & U &   &  S &    &   Red      \\ \hline
$\mathcal{J}_{38}$   & U &   &   &   &   &   &  S &   &   &  U &   &   &   &  S &  U  &  Navy    \\ \hline
$\mathcal{J}_{39}$   & U &   &   & U & U &  U & S &   &   &  U &   & U &   &  S &  U  &   Navy    \\ \hline
$\mathcal{J}_{40}$   & U &  U  &  U & U & U & U & S & S & U & U &   & U &   & S &  U &   Magenta   \\ \hline
$\mathcal{J}_{41}$ & U & U & U & U & U &  U & U &  U & U & U &   &  U &   &  S &    &   Red      \\ \hline
$\mathcal{J}_{42}$ & U &  U  & U &   &   &    & U &  U & U & U &   &    &   &  S &    &   Red  \\ \hline
$\mathcal{J}_{43}$& U &  U &  U &   &   &   &   &   &   & S &   &   &   &   &    &    Sienna  \\ \hline
$\mathcal{J}_{44}$& U & U &   &   &   &   &  U &  U &   & U &   &   &   &  S &    &   Red  \\ \hline
$ \mathcal{J}_{45}$& U & U &   &   &   &   &  U &  U & U & U &   &   &   &  S &    &   Red  \\ \hline
$\mathcal{J}_{46}$  & U & U &  & U & U &  & U & U &  & U &  & U &  & S &  & Red\\ \Xhline{3\arrayrulewidth}
$\mathcal{J}_{47}$ & S & S & U &  &  &  &   &  &   & S & U &   &   &   &  & Black \\ \hline
$\mathcal{J}_{48}$ & U & U & U & U & S & U &  &  &  & U & U & S & U &  &  & Gold\\ \hline
$\mathcal{J}_{49}$ & U & U & U & U & U & U & S & S & U & U & U & U & U & S & U & Magenta \\ \hline
$\mathcal{J}_{50}$ & U & U &  & U & U & U & S & S & U & U & U & U & U & S & U  & Magenta \\\hline
$\mathcal{J}_{51}$ & U & U &  & U & U & & S & S & U & U & U & U & U & S & U & Magenta\\\hline
$\mathcal{J}_{52}$ & U & U &  & U & U &  & U & S & & U & U & U & U & S & U & Wheat\\ \hline
$\mathcal{J}_{53}$ & U & U &  & S & S & U &  &   &  & U & U & S & U &  & &Coral\\ \hline
$\mathcal{J}_{54}$ & U & U &  & U & S &  &  &  &  & U & U & S & U &  &  & Gold\\ \hline
$\mathcal{J}_{55}$ & U & S &  &  &  &  &  &  &  & S & U &   &  &  &  &Turquoise\\ \hline
$\mathcal{J}_{56}$  & U & U &  & U & U &  &  &  &  & U & & S &  &  &  & Green\\ \hline
$\mathcal{J}_{57}$  & U & U &  & U & U &  & U & S &  & U &  & U &  & S & U & Wheat\\ \hline
$\mathcal{J}_{58}$  & U & U &  & U & U &  & S & S & U & U &  & U &  & S & U & Magenta\\ \hline
$\mathcal{J}_{59}$  & U & U &  & U & U &  U & S & S & U & U &  & U &  & S & U & Magenta\\ \hline
$\mathcal{J}_{60}$ & U & U &  & U & U &  & S & S & U &  &  &  &  & S & U & Magenta\\ \hline
$\mathcal{J}_{61}$  & U & U &  & U & U &  & U & S &  &  &  &  &  & S & U &Wheat\\ \hline
$\mathcal{J}_{62}$  & U & U &  & U & U &  & U & U &  &  &  &  &  & S &  & Red\\ \hline
$\mathcal{J}_{63}$  & U & U &  & U & U &  & U & S &  &  &  &  &  &  &  &Tan\\ \hline
$\mathcal{J}_{64}$  & U &  &  & U & S &  &  &  &  &  &  &  &  &  &  & Blue\\ \hline
$\mathcal{J}_{65}$  & U &  &  & U & U &  & U & S &  &  &  &  &  &  &  & Tan\\ \hline
$\mathcal{J}_{66}$  & U &  &  & U & U &  & U & U &  &  &  &  &  & S & & Red\\ \hline
$\mathcal{J}_{67}$  & U &  &  & U & U &  & U & S &  &  &  &  &  & S & U &Wheat\\ \hline
$\mathcal{J}_{68}$  & U &  &  & U & U &  & S & S & U &  &  &  &  & S & U & Magenta\\ \hline
$\mathcal{J}_{69}$  & U &  &  & U & U & U & S & S & U &  &  &  &  & S & U & Magenta\\ \hline
\end{tabular}
\end{center}}
\end{table}
\begin{figure}[!ht]
\setlength{\unitlength}{1.0cm}
\begin{center}
\begin{picture}(7,5.7)(0,0)
\put(-4.5,0){\rotatebox{0}{\includegraphics[width=7.5cm,height=5.5cm]{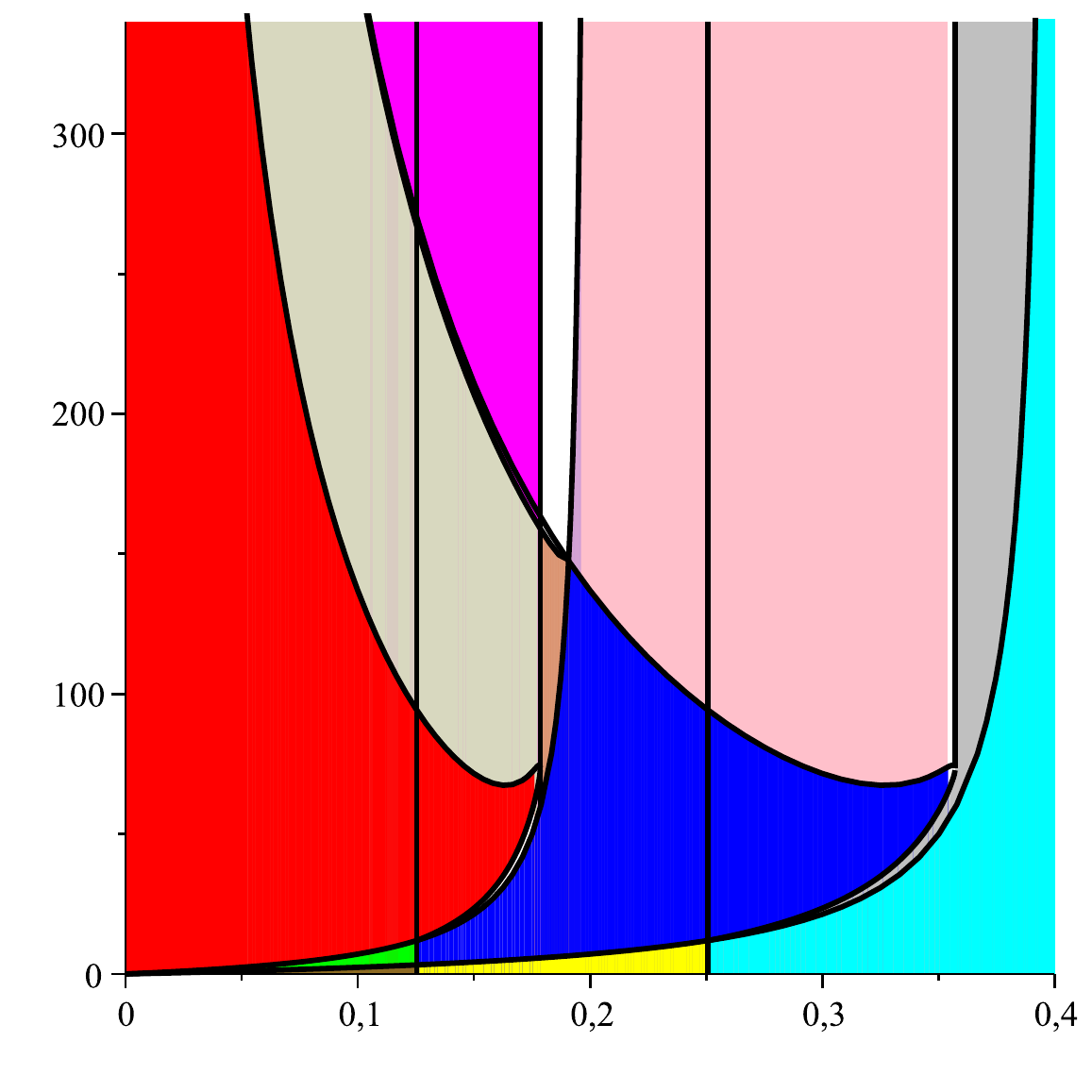}}}
\put(3.5,0){\rotatebox{0}{\includegraphics[width=7.5cm,height=5.5cm]{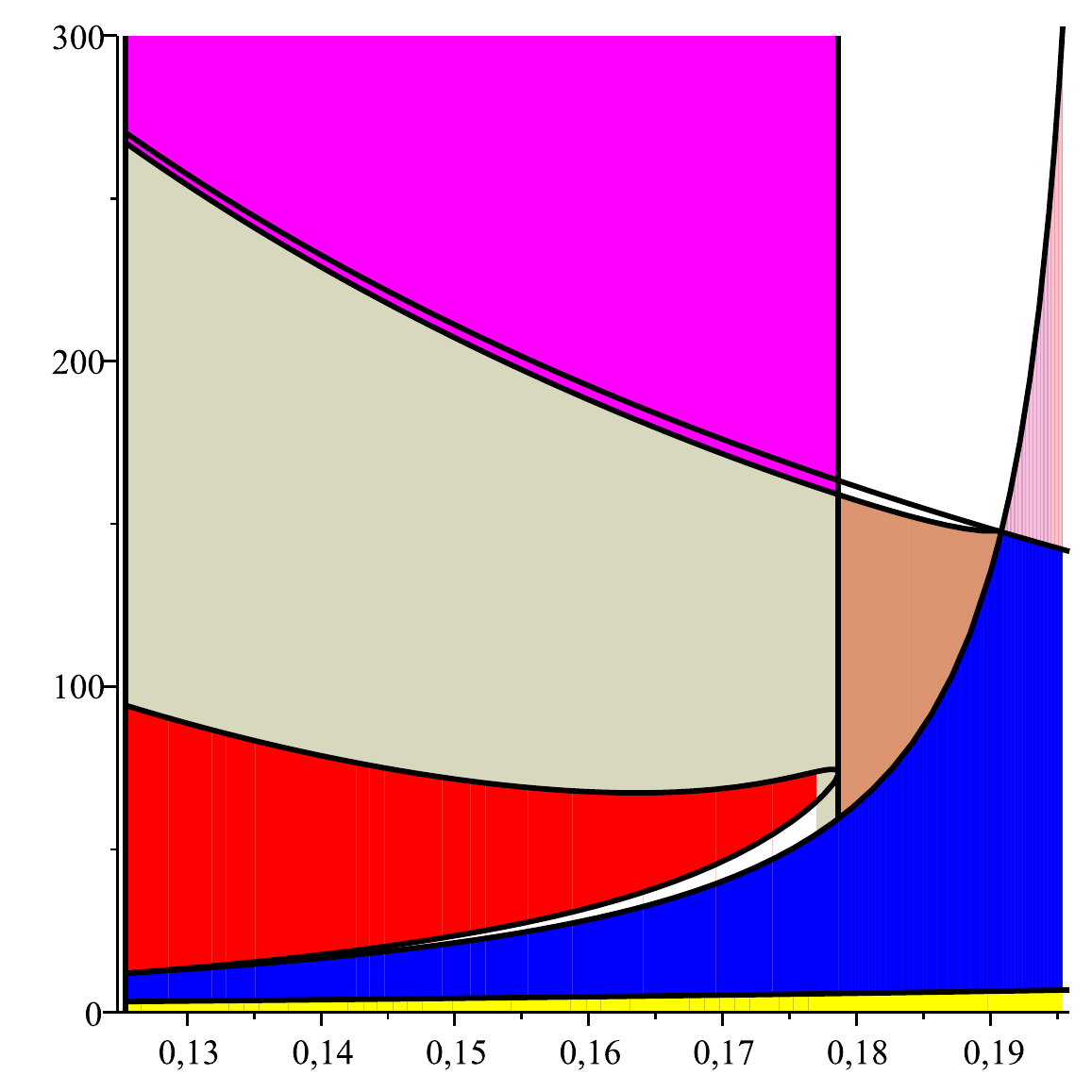}}}
\put(-0.5,5.8){\sc $(a)$}
\put(2,3.5){\sc { $\mathcal{J}_1$}}
\put(2,0.9){\sc { $\mathcal{J}_0$}}
\put(0.8,1.1){\sc { \color{white}$\mathcal{J}_{25}$}}
\put(0.8,3.5){\sc { $\mathcal{J}_{26}$}}
\put(-1.3,4.2){\sc {\color{white} $\mathcal{J}_{27}$}}
\put(-1.6,2.5){\sc { $\mathcal{J}_{29}$}}
\put(-1.85,1.4){\sc { \color{white}$\rightarrow$}}
\put(-2.2,1.4){\sc {\color{white}$\mathcal{J}_{30}$}}
\put(-1.3,1.1){\sc { \color{white}$\rightarrow$}}
\put(-1.6,1.1){\sc {\color{white} $\mathcal{J}_{31}$}}
\put(-0.4,1.1){\sc { \color{white}$\mathcal{J}_5$}}
\put(-1.1,2){\sc { $\rightarrow$}}
\put(-1.4,1.95){\sc { $\mathcal{J}_6$}}
\put(-0.75,3.5){\sc { $\leftarrow$}}
\put(-0.4,3.5){\sc { $\mathcal{J}_9$}}
\put(-0.8,2.7){\sc { $\leftarrow$}}
\put(-0.5,2.7){\sc { $\mathcal{J}_7$}}
\put(-0.3,4.3){\sc { $\mathcal{J}_8$}}
\put(-0.2,0.45){\sc $\uparrow$}
\put(-0.4,0,15){\sc { $\mathcal{J}_4$}}
\put(-1.9,0.42){\sc $\uparrow$}	
\put(-2.1,0.12){\sc { $\mathcal{J}_{15}$}}
\put(-2.1,0.75){\sc {\color{white} $\searrow$}}
\put(-2.3,1){\sc { \color{white}$\mathcal{J}_{16}$}}
\put(-3.3,2.8){\sc { \color{white}$\mathcal{J}_{17}$}}
\put(-2.3,3.3){\sc { $\mathcal{J}_{18}$}}
\put(-1.9,3.95){\sc { $\rightarrow$}}
\put(-2.3,3.9){\sc { $\mathcal{J}_{28}$}}
\put(-2.25,4.85){\sc { $\rightarrow$}}
\put(-2.7,4.8){\sc { $\mathcal{J}_{19}$}}
\put(-1.85,4.8){\sc { \color{white}$\leftarrow$}}
\put(-1.55,4.8){\sc {\color{white} $\mathcal{J}_{20}$}}
\put(-3.9,5.6){\sc $S_1^{in}$}
\put(2.8,0.5){\sc  $D$}
\put(-0.7,5.5){\sc { $\gamma_0$}}
\put(1.8,5.5){\sc { $\gamma_9$}}
\put(-1.9,5.5){\sc { $\gamma_6$}} 	
\put(-2.2,5.5){\sc { $\gamma_3$}}
\put(-2.2,5.2){\sc {$\searrow$}}
\put(-2.6,5.5){\sc { $\gamma_{15}$}}
\put(0.1,5.5){\sc { $\gamma_7$}}
\put(-1.1,5.5){\sc { $\gamma_8$}}
\put(2.4,5.5){\sc { $\gamma_1$}}
\put(-3.1,5.5){\sc { $\gamma_5$}}		
\put(7,5.7){{\sc $(b)$}}
\put(6,2.8){\sc { $\mathcal{J}_{29}$}}
\put(9.6,2.95){\sc { $\searrow$}}
\put(9.5,3.3){\sc { $\mathcal{J}_7$}}
\put(8.2,3.25){\sc { \color{white}$\leftarrow$}}
\put(8.5,3.3){\sc { \color{white}$\mathcal{J}_{28}$}}
\put(9.5,1.1){\sc { \color{white}$\mathcal{J}_5$}}
\put(9.8,4.5){\sc {$\mathcal{J}_9$}}
\put(5.8,4.7){\sc {\color{white}$\mathcal{J}_{27}$}}
\put(10.5,3.3){\sc {$\mathcal{J}_8$}}
\put(8,0.25){\sc $\uparrow$}
\put(7.8,-0.05){\sc { $\mathcal{J}_4$}}
\put(7.65,1.03){\sc { \color{white}$\searrow$}}
\put(7.2,1.25){\sc { \color{white}$\mathcal{J}_{31}$}}
\put(4.8,1.25){\sc { \color{white}$\mathcal{J}_{30}$}}
\put(9.5,2.2){\sc { $\mathcal{J}_6$}}
\put(4,5.6){\sc $S_1^{in}$}
\put(10.9,0.3){\sc  $D$}
\put(3.9,0.6){\sc { $\gamma_4$}}
\put(4.25,5.45){\sc { $\gamma_6$}}
\put(3.75,4.8){\sc { $\gamma_{15}$}}
\put(10.6,5.5){\sc { $\gamma_0$}} 	
\put(10.8,2.7){\sc { $\gamma_3$}}
\put(10.8,0.6){\sc { $\gamma_1$}}
\put(3.88,1.8){\sc { $\gamma_5$}}
\put(9,5.45){\sc { $\gamma_8$}}
\end{picture}
\vspace{-0.3cm}
\caption{Operating diagram of \cref{ModelAM2S2C} with twenty-one regions $\mathcal{J}_i$, $i=0,1,4-9,15-20,25-31$ when $r= 1/3\in(0,1/2)$, and $S_2^{in}=10<S_2^m\approx48.740$. Case 1 of \cref{Cases123}: $m_1=0.6>\mu_2\left(S_2^m\right)\approx0.535$.}\label{FigDO3}
\end{center}
\end{figure}

Note that we have chosen the same color for regions with the same number of LES steady states. For example, the regions $\mathcal{J}_2$ and $\mathcal{J}_8$ don't have the same number of unstable steady states, but they are colored in pink because both regions have the same behavior with the bistability between $\mathcal{E}_{00}^{10}$ and $\mathcal{E}_{00}^{11}$ (see \cref{TabRegDOJi}). In summary, there are three types of behavior of the process: either the stability of only one steady state while all others are unstable if they exist;
or the system exhibits bistability or the tristability of only two or three steady states while all others are unstable if they exist.
\begin{figure}[!ht]
\setlength{\unitlength}{1.0cm}
\begin{center}
\begin{picture}(7,5.9)(0,0)
\put(-4.5,0){\rotatebox{0}{\includegraphics[width=7.5cm,height=5.5cm]{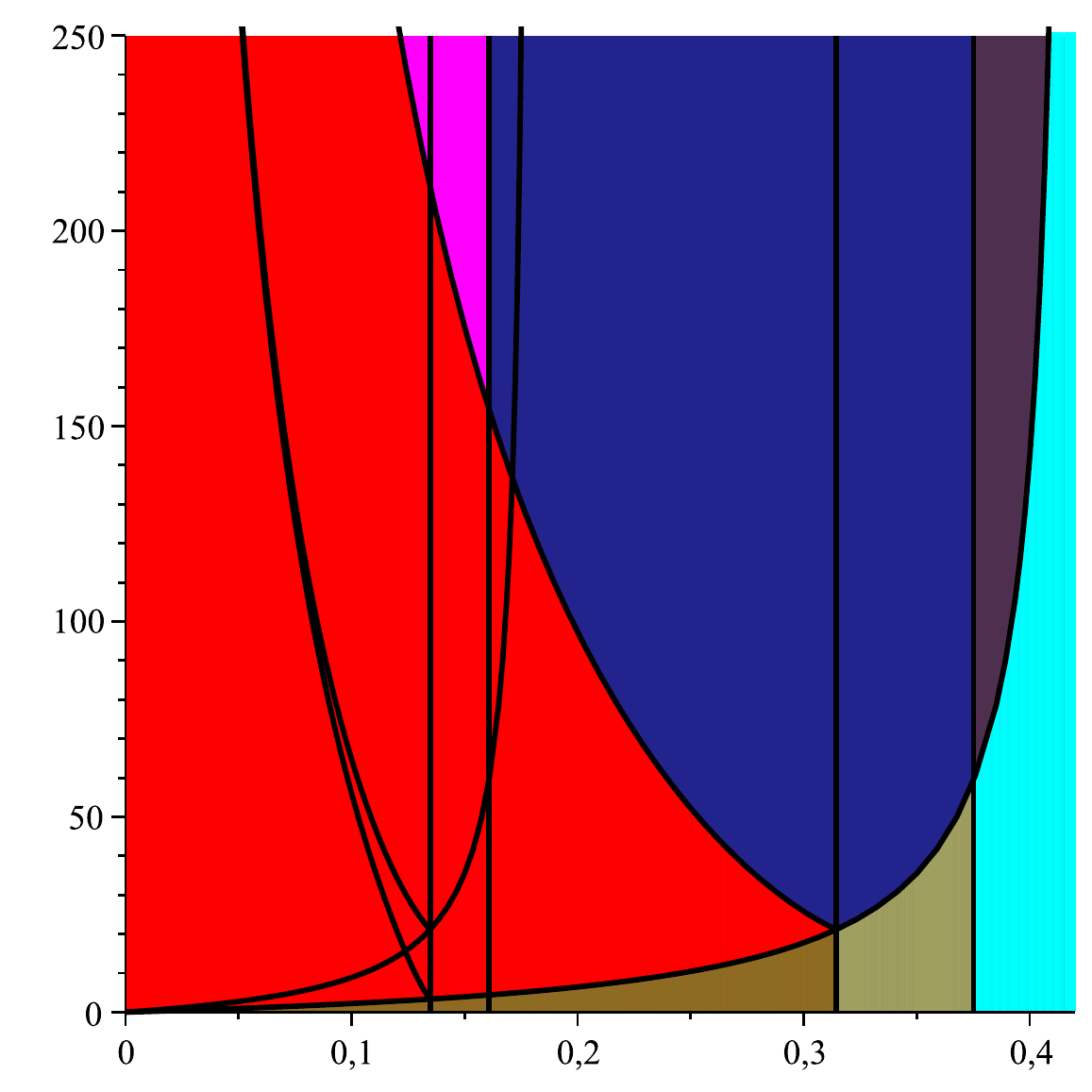}}}
\put(3.5,0){\rotatebox{0}{\includegraphics[width=7.5cm,height=5.5cm]{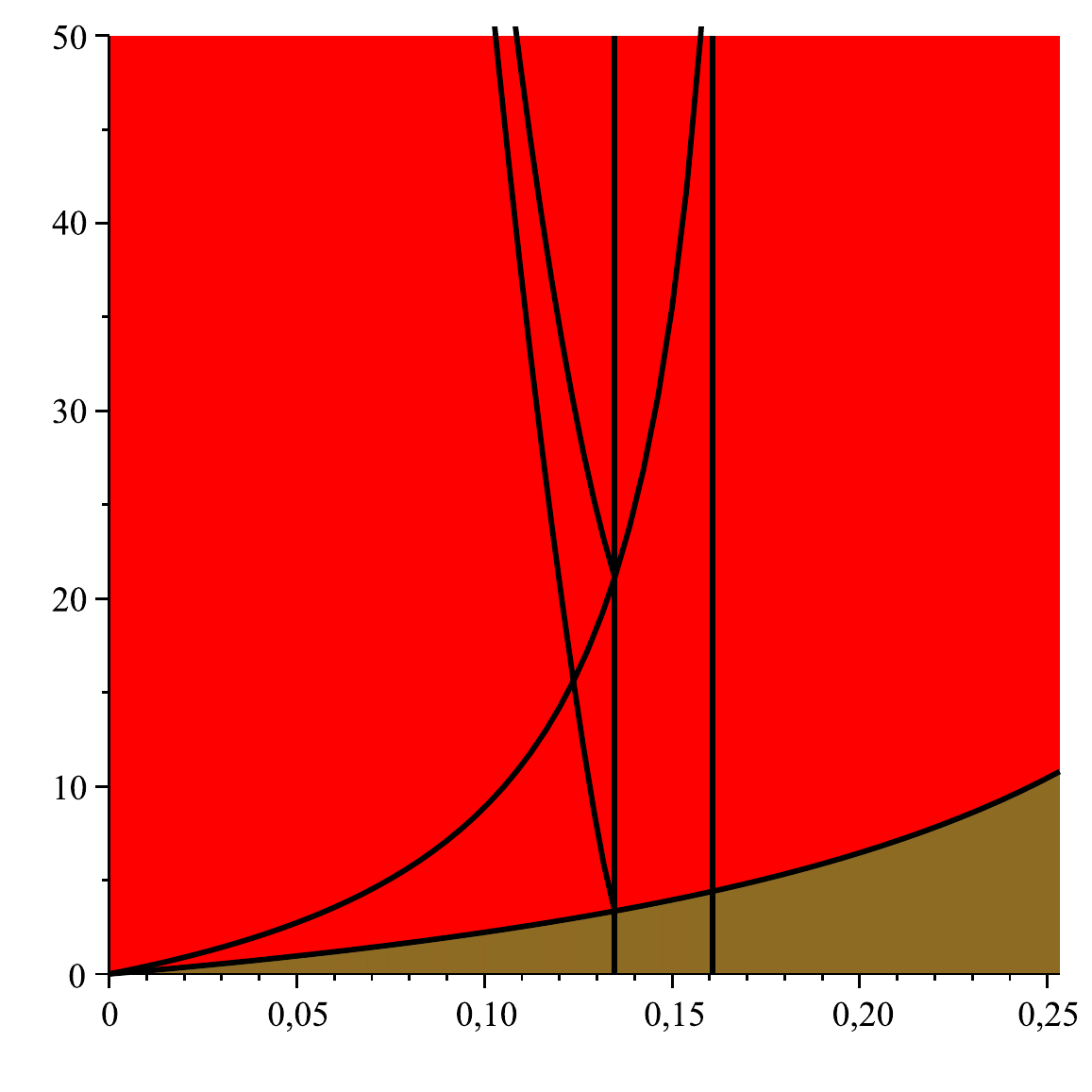}}}
\put(-1,5.8){\sc $(a)$}
\put(2.3,1){\sc { $\mathcal{J}_0$}}
\put(2.15,3.7){\sc {\color{white} $\mathcal{J}_{32}$}}
\put(1.5,3.7){\sc {\color{white} $\mathcal{J}_{33}$}}
\put(1.4,0.65){\sc { \color{white}$\mathcal{J}_{34}$}}
\put(0.75,0.5){\sc {\color{white} $\mathcal{J}_{35}$}}
\put(-0.2,1){\sc { \color{white}$\mathcal{J}_{36}$}}
\put(-1.65,1.9){\sc { \color{white}$\mathcal{J}_{41}$}}
\put(0.1,3.7){\sc { \color{white}$\mathcal{J}_{38}$}}
\put(-1.15,3.7){\sc { \color{white}$\leftarrow$}}
\put(-0.7,3.7){\sc {\color{white} $\mathcal{J}_{39}$}}
\put(-1.55,4.8){\sc {\color{white} $\mathcal{J}_{40}$}}
\put(-1.1,2.2){\sc \color{white}$\nwarrow$}
\put(-0.9,1.9){\sc { \color{white}$\mathcal{J}_{37}$}}
\put(-1.35,1){\sc { \color{white}$\leftarrow$}}
\put(-1,1){\sc {\color{white}$\mathcal{J}_{42}$}}
\put(-1.75,0.55){\sc { \color{white}$\leftarrow$}}
\put(-1.45,0.55){\sc {\color{white}$\mathcal{J}_{45}$}}
\put(-1.35,0.25){\sc $\uparrow$}
\put(-1.55,-0.05){\sc { $\mathcal{J}_{43}$}}
\put(-2.,0.2){\sc $\uparrow$}	
\put(-2.2,-0.1){\sc { $\mathcal{J}_{15}$}}
\put(-2.2,1.2){\sc {\color{white} $\rightarrow$}}
\put(-2.75,1.2){\sc {\color{white} $\mathcal{J}_{21}$}}
\put(-2.4,0.6){\sc { \color{white}$\searrow$}}
\put(-2.75,0.8){\sc {\color{white} $\mathcal{J}_{44}$}}
\put(-3.5,3.7){\sc { \color{white}$\mathcal{J}_{46}$}}
\put(-2.3,3.7){\sc {\color{white} $\mathcal{J}_{22}$}}
\put(-1.9,4.95){\sc {\color{white} $\rightarrow$}}
\put(-2.25,4.8){\sc {\color{white} $\mathcal{J}_{20}$}}
\put(-3.8,5.45){\sc $S_1^{in}$}
\put(2.9,0.4){\sc  $D$}
\put(-1.1,5.45){\sc { $\gamma_1$}}
\put(2,5.45){\sc { $\gamma_8$}}
\put(1,5.45){\sc { $\gamma_6$}} 	
\put(-2.95,5.45){\sc { $\gamma_3$}}
\put(-3.25,5.15){\sc { $\searrow$}}
\put(-3.4,5.45){\sc { $\gamma_{15}$}}
\put(-1.8,5.45){\sc { $\gamma_7$}}
\put(-1.4,5.45){\sc { $\gamma_9$}}
\put(2.45,5.45){\sc { $\gamma_0$}}
\put(-2.1,5.45){\sc { $\gamma_5$}}	
\put(7,5.8){{\sc $(b)$}}
\put(7.1,4.8){\sc { \color{white}$\mathcal{J}_{22}$}}
\put(7.2,2.5){\sc { \color{white}$\mathcal{J}_{21}$}}
\put(7.1,0.65){\sc { \color{white}$\mathcal{J}_{15}$}}
\put(7.8,0.7){\sc { \color{white}$\mathcal{J}_{43}$}}
\put(9.7,0.85){\sc { \color{white}$\mathcal{J}_{35}$}}
\put(7.8,2){\sc { \color{white}$\mathcal{J}_{42}$}}
\put(9.7,2){\sc { \color{white}$\mathcal{J}_{36}$}}
\put(7.7,4.8){\sc { \color{white}$\mathcal{J}_{41}$}}
\put(7.4,2){\sc {\color{white}$\mathcal{J}_{45}$}}
\put(6.5,1){\sc { \color{white}$\mathcal{J}_{44}$}}
\put(5,4){\sc { \color{white}$\mathcal{J}_{46}$}}
\put(4,5.45){\sc $S_1^{in}$}
\put(10.8,0.6){\sc  $D$}
\put(6.4,5.45){\sc { $\gamma_{15}$}}	
\put(7.9,5.45){\sc { $\gamma_1$}}
\put(8.2,5.45){\sc { $\gamma_9$}}
\put(10.7,1.6){\sc { $\gamma_0$}}
\put(6.9,5.45){\sc { $\gamma_3$}}
\put(7.4,5.45){\sc { $\gamma_7$}}
\end{picture}
\vspace{-0.3cm}
\caption{Operating diagram of \cref{ModelAM2S2C} with twenty regions $\mathcal{J}_i$, $i=0,15,20-22,32-46$ when $r= 0.7\in(1/2,1)$ and $S_2^{in}=150>S_2^m\approx48.740$. Case 1 of \cref{Cases123}: $m_1=0.6>\mu_2\left(S_2^m\right)\approx0.535$.}\label{FigDO4}
\end{center}
\end{figure}
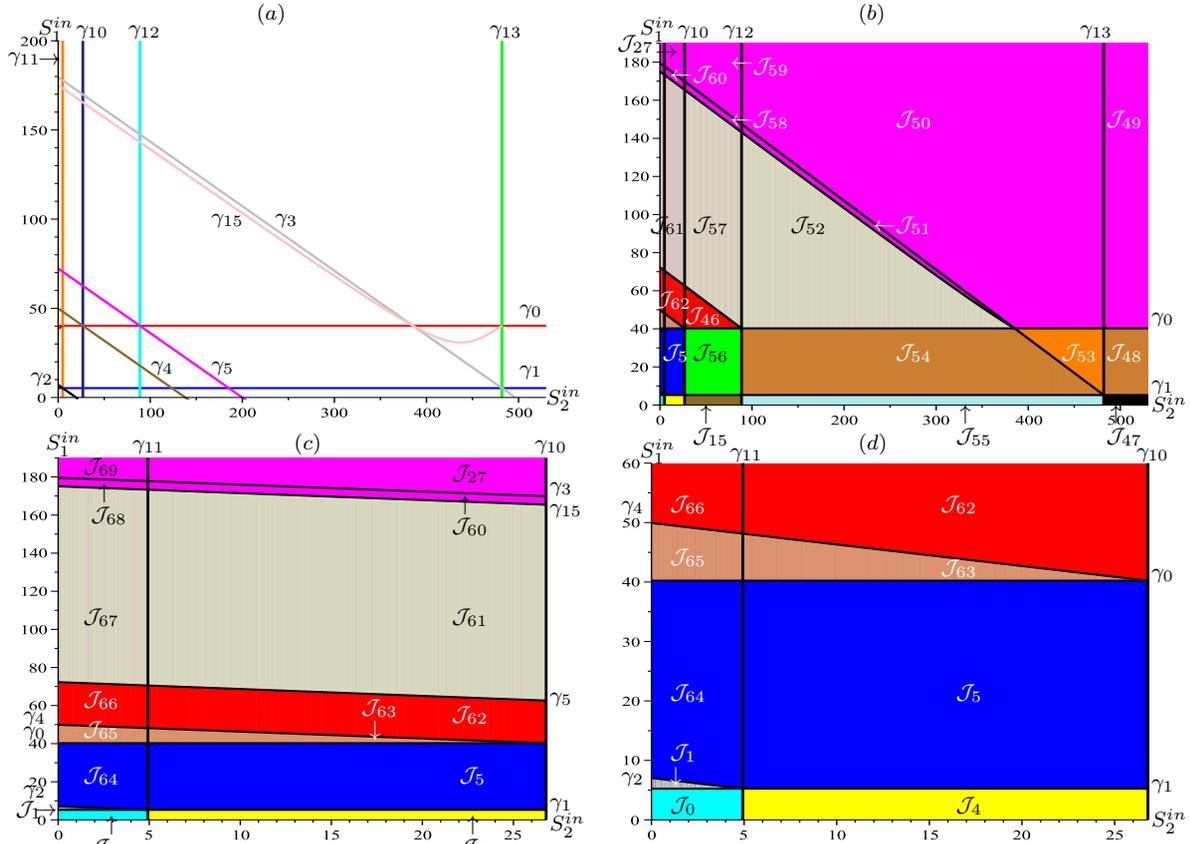
\begin{figure}[!ht]
\setlength{\unitlength}{1.0cm}
\begin{center}
\begin{picture}(7,6)(0,0)
\put(-4.5,0.2){\rotatebox{0}{\includegraphics[width=7.5cm,height=5.5cm]{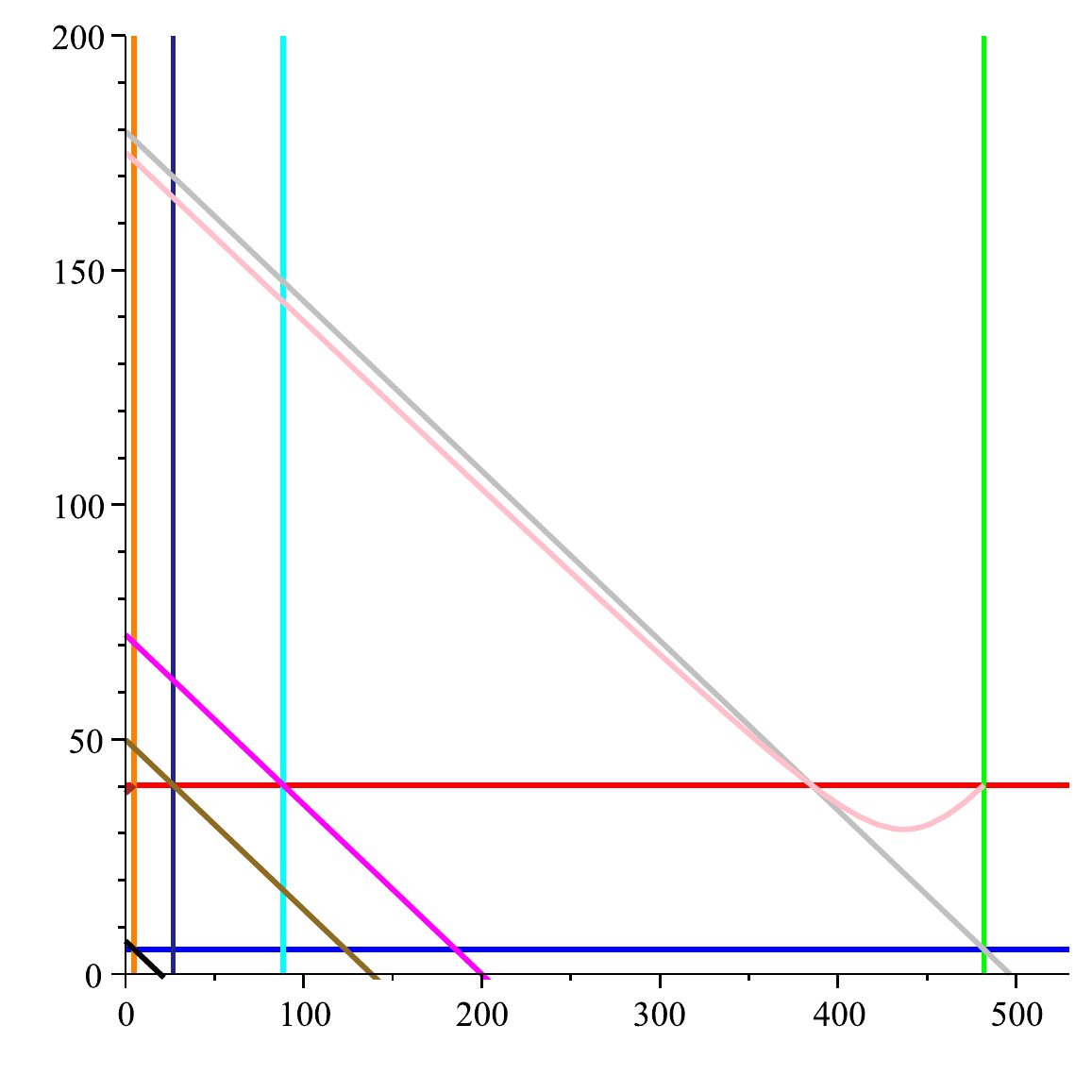}}}
\put(3.5,0.1){\rotatebox{0}{\includegraphics[width=7.5cm,height=5.5cm]{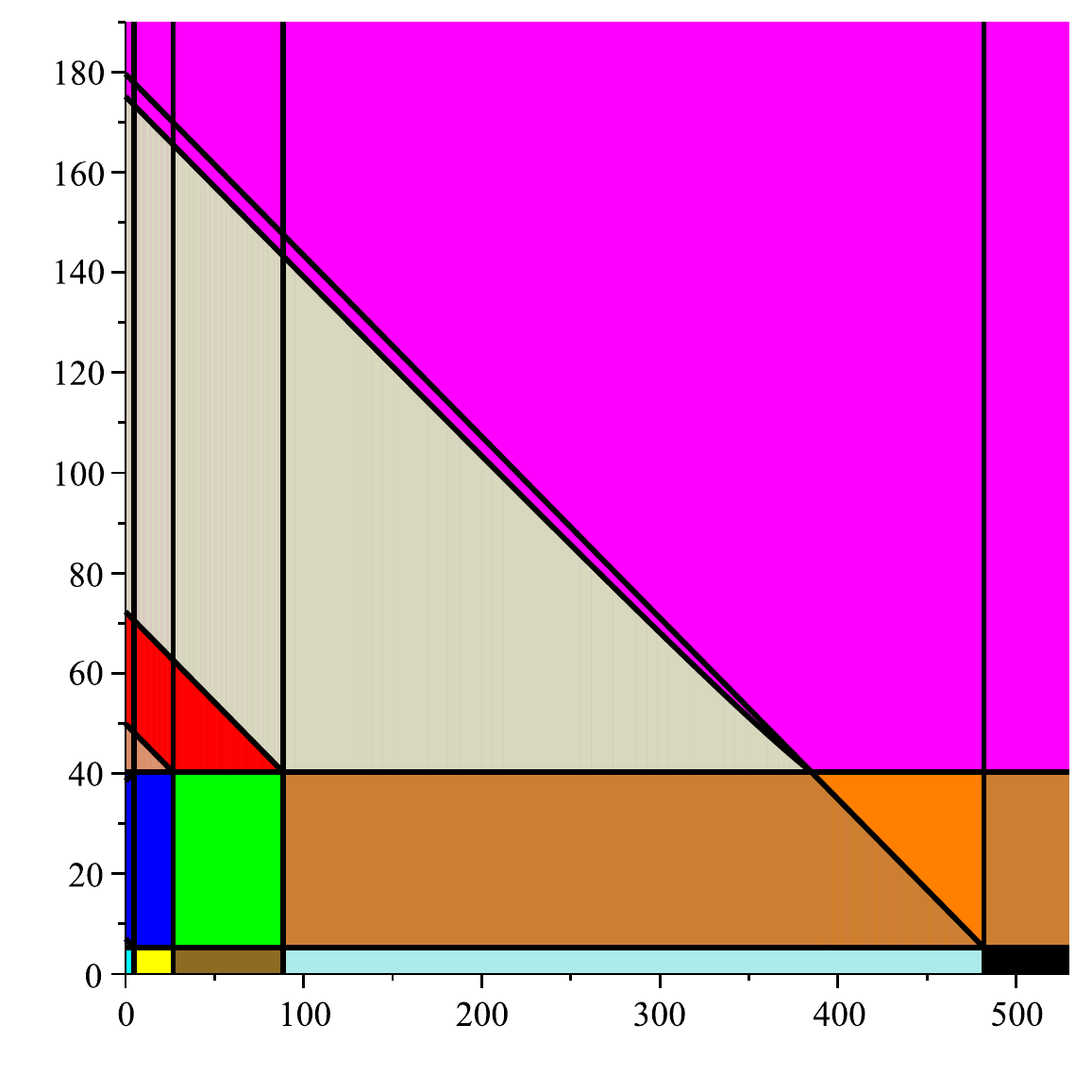}}}
\put(-1,5.8){\sc $(a)$}
\put(-3.9,5.6){\sc $S_1^{in}$}
\put(2.85,0.7){\sc $S_2^{in}$}
\put(2.4,1.9){\sc { $\gamma_0$}}
\put(-4,5.2){\sc { $\rightarrow$}}
\put(-4.4,5.25){\sc { $\gamma_{11}$}}
\put(-1.7,3.1){\sc { $\gamma_{15}$}}
\put(-2.8,5.6){\sc { $\gamma_{12}$}} 
\put(-0.85,3.1){\sc { $\gamma_3$}}
\put(2,5.6){\sc { $\gamma_{13}$}}
\put(2.4,1.1){\sc { $\gamma_1$}}
\put(-3.5,5.6){\sc { $\gamma_{10}$}}
\put(-1.7,1.15){\sc { $\gamma_5 $}}
\put(-4.1,1){\sc { $ \gamma_2 $}}
\put(-2.5,1.15){\sc { $ \gamma_4 $}}			
\put(7,5.8){{\sc $(b)$}}
\put(10.3,4.4){\sc {\color{white} $\mathcal{J}_{49}$}}
\put(9.7,1.3){\sc {\color{white} $\mathcal{J}_{53}$}}
\put(10.3,1.3){\sc {\color{white} $\mathcal{J}_{48}$}}
\put(7.5,4.4){\sc {\color{white} $\mathcal{J}_{50}$}}
\put(7.1,3){\sc { \color{white}$\leftarrow$}}
\put(7.5,3){\sc {\color{white} $\mathcal{J}_{51}$}}
\put(7.5,1.3){\sc {\color{white} $\mathcal{J}_{54}$}}
\put(5.3,5.15){\sc {\color{white} $\leftarrow$}}
\put(5.6,5.1){\sc {\color{white} $\mathcal{J}_{59}$}}
\put(4.7,1.8){\sc {\color{white} $\mathcal{J}_{46}$}}
\put(4.7,3){\sc { $\mathcal{J}_{57}$}}
\put(6,3){\sc { $\mathcal{J}_{52}$}}
\put(8.35,0.5){\sc $\uparrow$}
\put(8.2,0.2){\sc { $\mathcal{J}_{55}$}}
\put(10.35,0.5){\sc $\uparrow$}
\put(10.2,0.2){\sc { $\mathcal{J}_{47}$}}
\put(4.9,0.5){\sc $\uparrow$}	
\put(4.7,0.2){\sc { $\mathcal{J}_{15}$}}
\put(4.2,2){\sc { \color{white}$\mathcal{J}_{62}$}}
\put(4.3,1.3){\sc { \color{white}$\mathcal{J}_5$}}
\put(4.7,1.3){\sc { $\mathcal{J}_{56}$}}
\put(5.3,4.4){\sc {\color{white} $\leftarrow$}}
\put(5.5,4.4){\sc { \color{white}$\mathcal{J}_{58}$}}
\put(4.5,5){\sc {\color{white} $\leftarrow$}}
\put(4.7,5){\sc { \color{white}$\mathcal{J}_{60}$}}
\put(4.3,5.3){\sc {$\rightarrow$}}
\put(3.8,5.4){\sc {$\mathcal{J}_{27}$}}
\put(4.23,3){\sc {$\mathcal{J}_{61}$}}
\put(4.1,5.6){\sc $S_1^{in}$}
\put(10.9,0.6){\sc  $S_2^{in}$}
\put(10.8,1.8){\sc { $\gamma_0$}}
\put(4.5,5.6){\sc { $\gamma_{10}$}}
\put(5.1,5.6){\sc { $\gamma_{12}$}}	
\put(9.85,5.6){\sc { $\gamma_{13}$}}
\put(10.8,0.9){\sc { $\gamma_1$}}
\put(-4.5,-5.4){\rotatebox{0}{\includegraphics[width=7.5cm,height=5.5cm]{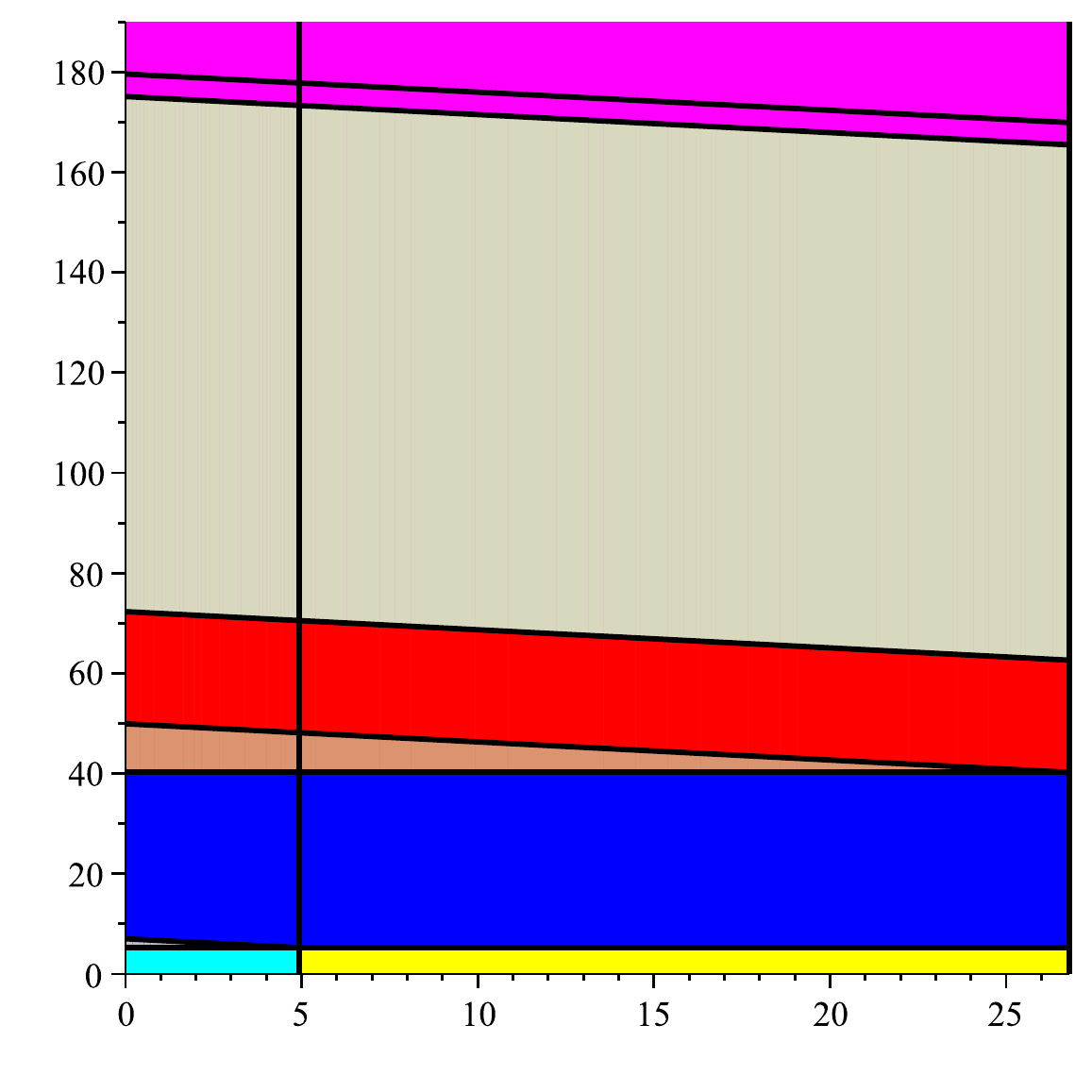}}}
\put(3.5,-5.4){\rotatebox{0}{\includegraphics[width=7.5cm,height=5.5cm]{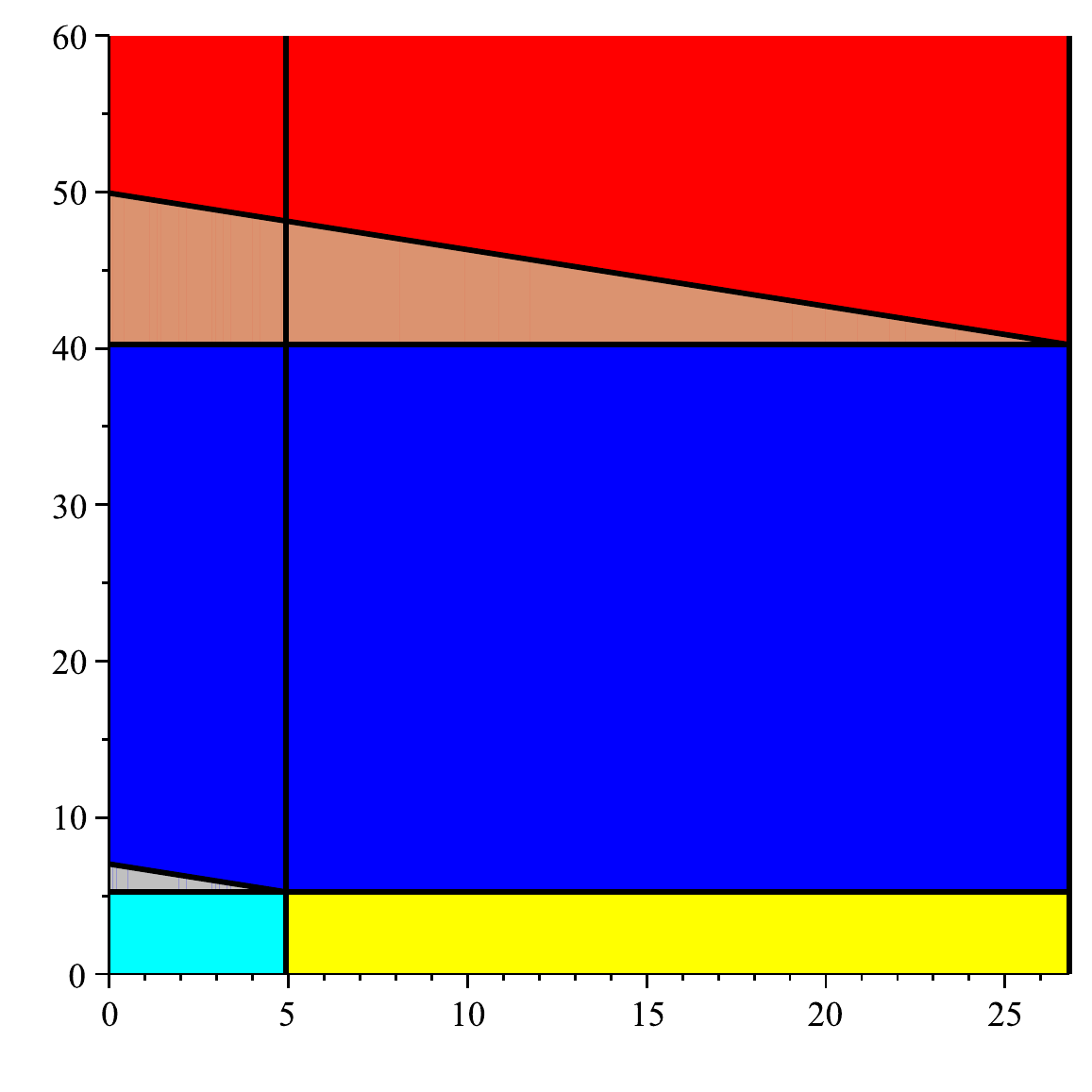}}}
\put(-0.5,0.1){\sc $(c)$}
\put(1.5,-2.2){\sc { $\mathcal{J}_{61}$}}
\put(1.6,-4.3){\sc { \color{white}$\mathcal{J}_5$}}
\put(1.5,-3.5){\sc { \color{white}$\mathcal{J}_{62}$}}
\put(-3.3,-3.7){\sc {\color{white} $\mathcal{J}_{65}$}}
\put(-3.3,-4.3){\sc {\color{white}$\mathcal{J}_{64}$}}
\put(-3.95,-4.75){\sc {$\rightarrow$}}
\put(-4.3,-4.75){\sc { $\mathcal{J}_1$}}
\put(0.5,-3.7){\sc\color{white} {$\downarrow$}}
\put(0.3,-3.4){\sc { \color{white}$\mathcal{J}_{63}$}}
\put(-3.3,-3.3){\sc {\color{white} $\mathcal{J}_{66}$}}
\put(1.5,-0.3){\sc { $\mathcal{J}_{27}$}}
\put(1.7,-0.7){\sc $\uparrow$}
\put(1.5,-1){\sc { $\mathcal{J}_{60}$}}
\put(1.8,-4.95){\sc $\uparrow$}
\put(1.6,-5.25){\sc { $\mathcal{J}_4$}}
\put(-3,-4.95){\sc $\uparrow$}	
\put(-3.3,-5.25){\sc { $\mathcal{J}_0$}}
\put(-3.1,-0.55){\sc $\uparrow$}
\put(-3.3,-0.85){\sc { $\mathcal{J}_{68}$}}
\put(-3.3,-0.2){\sc {$\mathcal{J}_{69}$}}
\put(-3.3,-2.2){\sc {$\mathcal{J}_{67}$}}
\put(-3.8,0.1){\sc $S_1^{in}$}
\put(2.9,-4.95){\sc  $S_2^{in}$}
\put(2.8,-4.65){\sc { $\gamma_1$}}
\put(2.8,-0.75){\sc { $\gamma_{15}$}}
\put(2.8,-0.45){\sc { $\gamma_3$}}
\put(2.6,0.1){\sc { $\gamma_{10}$}}	
\put(-4.2,-3.5){\sc { $\gamma_4$}}
\put(-2.75,0.1){\sc { $\gamma_{11}$}}
\put(-4.2,-4.48){\sc { $\gamma_2$}}
\put(-4.2,-3.7){\sc { $\gamma_0$}}
\put(2.8,-3.25){\sc { $\gamma_5$}}	
\put(7,0.1){{\sc $(d)$}}
\put(8,-0.7){\sc { \color{white}$\mathcal{J}_{62}$}}
\put(8,-1.55){\sc { \color{white}$\mathcal{J}_{63}$}}
\put(8.2,-4.7){\sc { $\mathcal{J}_4$}}
\put(8.2,-3.2){\sc { \color{white}$\mathcal{J}_5$}}
\put(4.5,-4.3){\sc\color{white} {$\downarrow$}}
\put(4.5,-4){\sc {\color{white} $\mathcal{J}_1$}}
\put(4.5,-4.7){\sc {$\mathcal{J}_0$}}
\put(4.4,-0.7){\sc { \color{white}$\mathcal{J}_{66}$}}
\put(4.4,-1.4){\sc { \color{white}$\mathcal{J}_{65}$}}
\put(4.4,-3.2){\sc { \color{white}$\mathcal{J}_{64}$}}
\put(4.1,0){\sc $S_1^{in}$}
\put(10.9,-4.9){\sc $S_2^{in}$}
\put(3.75,-0.7){\sc { $\gamma_4$}}
\put(10.8,-4.4){\sc { $\gamma_1$}} 	
\put(3.75,-4.3){\sc { $\gamma_2$}}
\put(10.8,-1.6){\sc { $\gamma_0$}}
\put(10.6,0){\sc { $\gamma_{10}$}}
\put(5.2,0){\sc { $\gamma_{11}$}}
\end{picture}
\vspace{4.8cm}
\caption{Operating diagram of \cref{ModelAM2S2C} in the plane $(S_2^{in},S_1^{in})$ with thirty regions $\mathcal{J}_i$, $i=0,1,4,5,15,27,46-69$ when $D=0.17$ and $r=0.3\in(0,1/2)$.}\label{FigDO5Dfix}
\end{center}
\end{figure}

More precisely, the regions of existence of only one steady state LES (while all others are unstable if they exist) with the color are the following:
$\mathcal{E}_{00}^{00}$ in cyan ($\mathcal{J}_0$),
$\mathcal{E}_{00}^{10}$ in grey ($\mathcal{J}_1$),
$\mathcal{E}_{00}^{01}$ in yellow ($\mathcal{J}_4$),
$\mathcal{E}_{00}^{11}$ in blue ($\mathcal{J}_5$, $\mathcal{J}_{25}$, $\mathcal{J}_{64}$),
$\mathcal{E}_{10}^{11}$ in tan ($\mathcal{J}_6$, $\mathcal{J}_{63}$, $\mathcal{J}_{65}$),
$\mathcal{E}_{01}^{01}$ in sienna ($\mathcal{J}_{15}$, $\mathcal{J}_{35}$, $\mathcal{J}_{43}$),
$\mathcal{E}_{01}^{11}$ in green ($\mathcal{J}_{16}$, $\mathcal{J}_{23}$, $\mathcal{J}_{56}$),
$\mathcal{E}_{11}^{11}$ in red ($\mathcal{J}_i$, $i=17,21,22,30,36,37,41,42,44-46,62,66$),
$\mathcal{E}_{10}^{10}$ in violet ($\mathcal{J}_{32}$).

The bistability regions with the corresponding two steady states LES (while all others are unstable if they exist) and color are the following:
$\mathcal{E}_{00}^{10}$ and $\mathcal{E}_{00}^{11}$ in pink ($\mathcal{J}_2$, $\mathcal{J}_8$, $\mathcal{J}_{26}$),
$\mathcal{E}_{00}^{00}$ and $\mathcal{E}_{00}^{01}$ in plum ($\mathcal{J}_3$),
$\mathcal{E}_{10}^{10}$ and $\mathcal{E}_{10}^{11}$ in white ($\mathcal{J}_7$, $\mathcal{J}_9$, $\mathcal{J}_{31}$),
$\mathcal{E}_{10}^{11}$ and $\mathcal{E}_{11}^{11}$ in wheat ($\mathcal{J}_i$, $i=12,18,29,52,57,61,67$),
$\mathcal{E}_{00}^{11}$ and $\mathcal{E}_{01}^{11}$ in gold ($\mathcal{J}_{13}$, $\mathcal{J}_{48}$, $\mathcal{J}_{54}$),
$\mathcal{E}_{00}^{01}$ and $\mathcal{E}_{01}^{01}$ in turquoise ($\mathcal{J}_{14}$, $\mathcal{J}_{55}$),
$\mathcal{E}_{10}^{10}$ and $\mathcal{E}_{11}^{11}$ in navy ($\mathcal{J}_{33}$, $\mathcal{J}_{38}$, $\mathcal{J}_{39}$),
$\mathcal{E}_{00}^{00}$ and $\mathcal{E}_{01}^{01}$ in khaki ($\mathcal{J}_{34}$).

The tristability regions with the corresponding three steady states LES and color are the following:
$\mathcal{E}_{10}^{10}$, $\mathcal{E}_{10}^{11}$ and $\mathcal{E}_{11}^{11}$ in magenta ($\mathcal{J}_i$, $i=10,11,19,20,27,28,40,49-51,58-60,68,69$),
$\mathcal{E}_{00}^{10}$, $\mathcal{E}_{00}^{11}$ and $\mathcal{E}_{01}^{11}$ in brown ($\mathcal{J}_{24}$),
$\mathcal{E}_{00}^{00}$, $\mathcal{E}_{00}^{01}$ and $\mathcal{E}_{01}^{01}$ in black ($\mathcal{J}_{47}$),
$\mathcal{E}_{00}^{10}$, $\mathcal{E}_{00}^{11}$ and $\mathcal{E}_{01}^{11}$ in coral ($\mathcal{J}_{53}$).
\subsection\protect{Operating diagrams in the plane $\left(S_2^{in},S_1^{in}\right)$ when $D$ and $r$ are fixed.}  \label{sec2DO}
\begin{proposition}                                     \label{PropPosiCurvDfix}
For $D$ fixed in $\left(0,r\mu_2\left(S_2^m\right)\right) $ and $r\in(0,1)$, the lines $\gamma_0$, $\gamma_4$ and $\gamma_{10}$ intersect at the same point $\left(S_{21}^{in*},\lambda_1^1(D,r)\right)$. In addition, the lines $\gamma_0$, $\gamma_5$ and $\gamma_{12}$ intersect at the same point $\left({S_{23}^{in*}},\lambda_1^1(D,r)\right)$ in the plane $(S_2^{in},S_1^{in})$, where $S_{21}^{in*}=\lambda_2^{11}(D,r)$ and ${S_{23}^{in*}}=\lambda_2^{12}(D,r)$ (see \cref{FigDO5Dfix}). Moreover:
\begin{itemize}
    \item If $S_2^{in}<S_{21}^{in*}$, then $\lambda_1^1(D,r)<F_{11}\left(D,r,S_2^{in}\right)<F_{12}\left(D,r,S_2^{in}\right)$.
    \item If $S_{21}^{in*}<S_2^{in}<{S_{23}^{in*}}$, then $F_{11}\left(D,r,S_2^{in}\right)<\lambda_1^1(D,r)<F_{12}\left(D,r,S_2^{in}\right)$.
    \item If $S_2^{in}>{S_{23}^{in*}}$, then $F_{11}\left(D,r,S_2^{in}\right)<F_{12}\left(D,r,S_2^{in}\right)<\lambda_1^1(D,r)$.
\end{itemize}
For $D$ fixed in $(0,r_2\mu_2(S_2^m)) ~~\mbox{and}~~r\in(0,1)$, the lines $\gamma_1$, $\gamma_2$, $\gamma_{11}$ intersect at the same point $\left({S_{22}^{in*}},\lambda_1^2(D,r)\right)$ and the lines $\gamma_1$, $\gamma_3$, $\gamma_{13}$ intersect at the same point $\left({S_{24}^{in*}},\lambda_1^2(D,r)\right)$ in the plane $(S_2^{in},S_1^{in})$, where $S_{22}^{in*}=\lambda_2^{21}(D,r)$ and $S_{24}^{in*}=\lambda_2^{22}(D,r)$ (see \cref{FigDO5Dfix}). Moreover:
\begin{itemize}
    \item If $S_2^{in}<S_{22}^{in*}$, then $\lambda_1^2(D,r)<F_{21}\left(D,r,S_2^{in}\right)<F_{22}\left(D,r,S_2^{in}\right)$.
    \item If $S_{24}^{in*}<S_2^{in}<S_{24}^{in*}$, then $F_{21}\left(D,r,S_2^{in}\right)<\lambda_1^2(D,r)<F_{22}\left(D,r,S_2^{in}\right)$.
    \item If $S_2^{in}>S_{24}^{in*}$, then $F_{21}\left(D,r,S_2^{in}\right)<F_{22}\left(D,r,S_2^{in}\right)<\lambda_1^2(D,r)$.
\end{itemize}
\end{proposition}
\begin{proposition}                                \label{PropDO-Dfix}
Assume that \cref{Hypoth1,Hypoth2} hold. Let $r$ and $D$ be fixed so that $0<r<1/2$ and $D\in\left(0,\mu_2\left(S_2^m\right)\right)$. For the specific growth rates $\mu_1$ and $\mu_2$ defined in \cref{SpeciFunc} and the set of the biological parameter values in \cref{TabParamVal}, the existence and the local stability properties of steady states of system \cref{ModelAM2S2C} in the thirty regions $\mathcal{J}_i$, $i=0,1,4,5,15,27,46-69$ of the operating diagram shown in \cref{FigDO5Dfix} are defined in \cref{TabRegDOJi}, where the conditions on the operating parameters defining the various boundaries of the regions are described in \cref{TabRegionCondDO5}.
\end{proposition}

\section{Discussion and conclusion}\label{Sec-Conc}

In this paper, we have introduced the original AM2 model \cref{ModelAM2S2C} with a serial configuration of two chemostats, thanks to the mass conservation principle.
The mathematical model considered is a differential system in eight dimensions.
The analysis of this model is original and has not been studied in the literature.
Using the same dilution rate with a large class of monotonic and nonmonotonic growth functions for the first and the second species, we have presented an in-depth mathematical study of system \cref{ModelAM2S2C}.
Since this model has a cascade structure, we show that it can be reduced to a four-dimensional system \cref{ModelAM2SR4}.
The existence and the stability of each steady state of the complete model \cref{ModelAM2S2C} can be deduced from the reduced model \cref{ModelAM2SR4}.
The necessary and sufficient conditions of existence and local stability of all steady states of \cref{ModelAM2SR4} are established according to the four operating parameters $D$, $r$, $S_1^{in}$ and $S_2^{in}$.
Generically, there can be nine types of steady states, six of which can be doubled for non-monotonic growth rates, giving a total of fifteen steady states.

To describe the behavior of the process when these four control parameters are varied, we have determined theoretically the operating diagram from the existence and local stability conditions.
Fixing the biological parameter set provided in \cref{TabParamVal}, there can be seventy regions with three different behaviors: either the stability of a single LES steady state, or bistability or tristability, with the existence of two or three stable steady states, while all other steady states are unstable when they exist.


It is interesting to note that this general study includes in fact a number of already known results. In other words, the actual results comprise results that were already established for specific functioning conditions. In particular, it is the case when the functioning conditions are such that both biomasses are washed in the first reactor (i.e. when $X_1^1=0$ and $X_2^1=0$): the results obtained are then equivalent to those obtained for the analysis of the AM2 model considered in a single chemostat \cite{BenyahiaJPC2012,SariNonLinDyn2021}.
It is logically the case when the first reactor is smaller than the first one, in which case, for a given too high dilution rate, the washout of the biomasses arises in the first reactor.
The other singular case is that one in which the first biomass is washed out in both reactors: in such a case the analysis is equivalent to the study of a single bioreaction i) having as bioreactional scheme $S_2\rightarrow X_2 + biogas$ and ii) taking place in a system of two chemostats in series \cite{DaliYoucefIFAC2022,DaliYoucefBMB2022}.
Following the same idea, we also find more recent results established in \cite{HarmandJPC2020}. These authors had shown that under certain circumstances, well characterized as a function of operating conditions and system parameters by Sari and Benyahia (cf. \cite{SariNonLinDyn2021}), a multi-step system (in the sense that the substrate feed to a reaction $i$ is the reaction output of step $i-1$, which is the case of the AM2 reaction scheme), in a chemostat could exhibit highly original behavior.
In particular, it was shown that in these particular circumstances, the system can be stabilized by increasing the feed rate, rather than decreasing it as is traditionally the case.
It's interesting to note that this is the case here for two chemostats in series, and it may happen when the value of $r$ is smaller than $1/2$.  
In fact it appears that the operating diagrams shown in \cref{FigDO2,FigDO3}, were drawn up with values that allow these results to be found. Consider \cref{FigDO2} where $S_2^{in}>S_2^m$ and Case 1 of \cref{Cases123} ($m_1>\mu_2\left(S_2^m\right)$) holds: the red zone rises slightly to the right of the lowest point of the wheat zone. Let's start from a point in the wheat zone, very close to its low point. In this zone (denoted $\mathcal{J}_{18}$ in \cref{FigDO2}), there are two stable steady states.
Suppose that the initial conditions are such that the system converges to the steady state $\mathcal{E}_{10}^{11}$ characterized by $X_2^1=0$.
We can imagine that on an operational level, it is not desirable to have one of the biomasses of the first reactor which is zero.
By increasing $D$ very slightly, we will find ourselves in the part of the red zone ($\mathcal{J}_{17}$) which goes up to the right of the low point of the wheat zone.
In other words, we find ourselves in a zone where the only stable steady state is now the positive steady state $\mathcal{E}_{11}^{11}$, having increased rather than decreased the feed rate of the system!

Now, when $r<1/2$, $S_2^{in}<S_2^m$ and Case 2 of \cref{Cases123} $\left(m_1<\mu_2\left(S_2^{in}\right)\right)$ holds, one has exactly the same phenomenon that can be highlighted from the operating diagram plotted in \cref{FigDO3}.

To enhance and extend the results of the mathematical study of this anaerobic digestion process with a series configuration of interconnected chemostats, a forthcoming paper will analyse performance criteria by comparing with the single chemostat.
\appendix
\section{Proofs}                                      \label{AppProof}
First, we need the following Lemmas to establish the existence and multiplicity of all steady states of \cref{ModelAM2SR4}.
\begin{lemma}                                            \label{lemH1}
Let $f_1$ and $g_1$ be the functions defined in \cref{TabFunc2}. Let $X_1^{1*}<S_1^{in}/k_1$. Assume that \cref{Hypoth1} holds. The equation $f_1(x)=g_1(x)$ has a unique solution in $\left(X_1^{1*},S_1^{in}/k_1\right)$ (see \cref{Figsoleq12}(a)).
\end{lemma}
\begin{figure}[!ht]
\setlength{\unitlength}{1.0cm}
\begin{center}
\begin{picture}(7.8,4.5)(0,0)	
\put(-4.3,0){{\includegraphics[width=5.5cm,height=4.5cm]{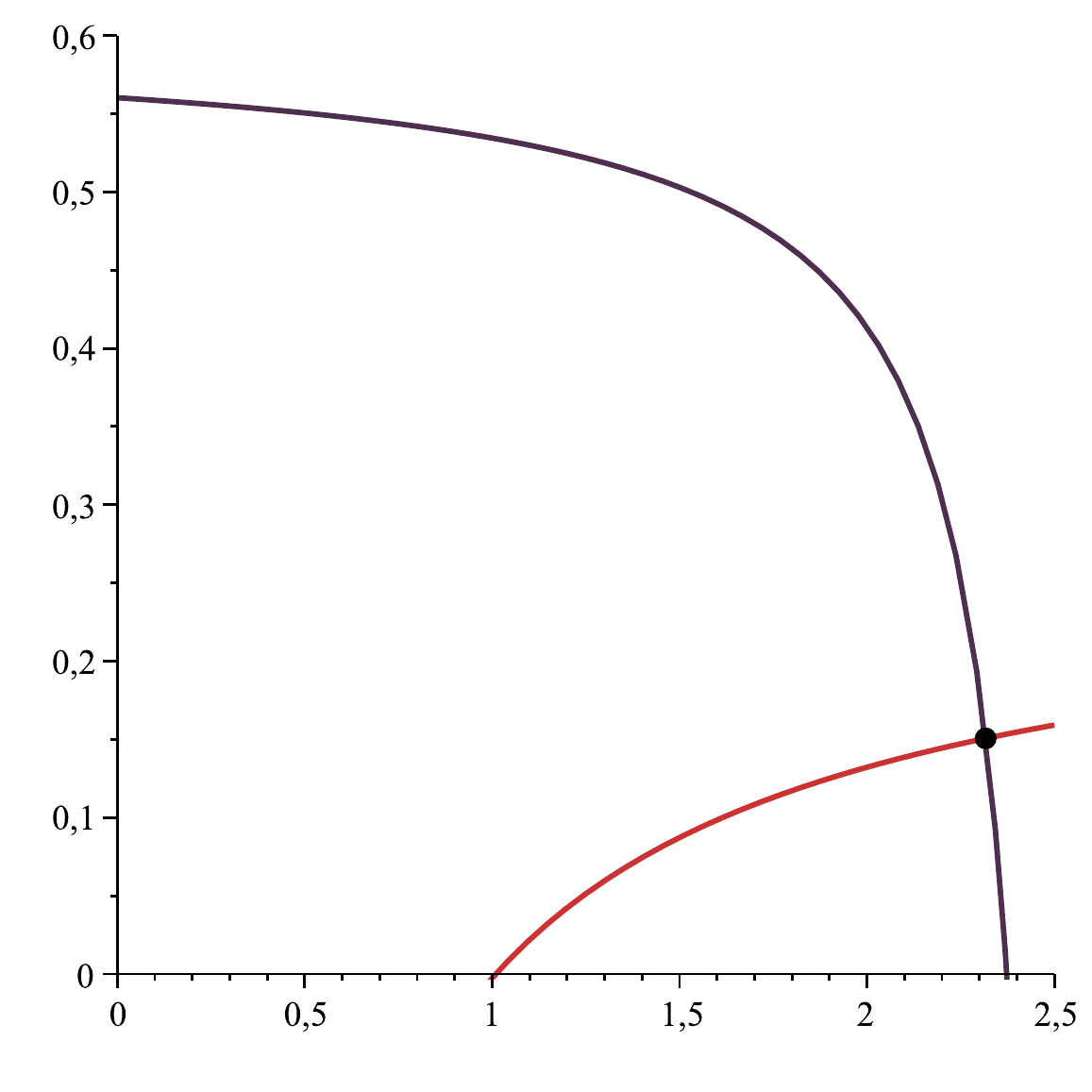}}}	
\put(0.9,0){{\includegraphics[width=5.5cm,height=4.5cm]{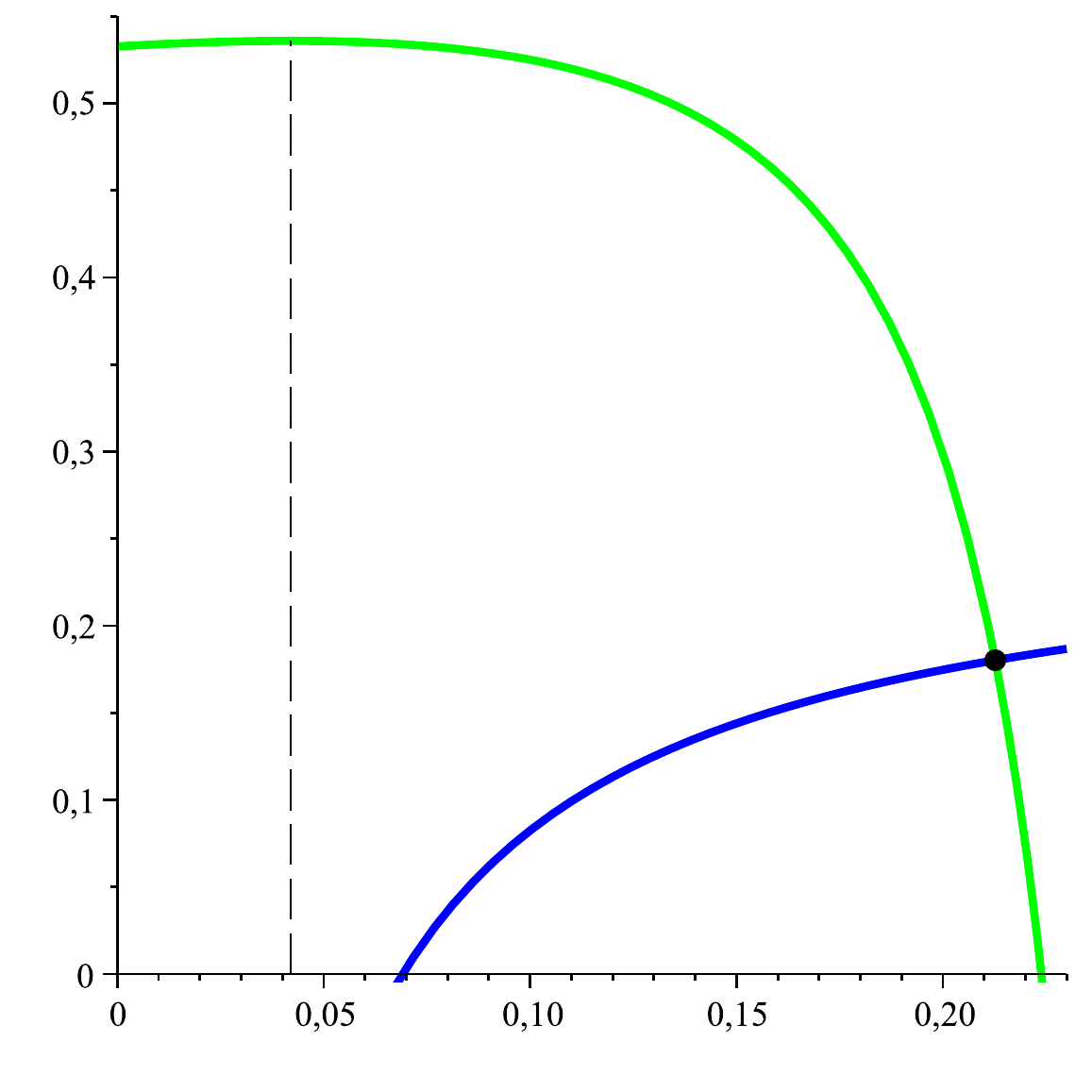}}}
\put(5,-6.7){{\includegraphics[width=8.5cm,height=12.5cm]{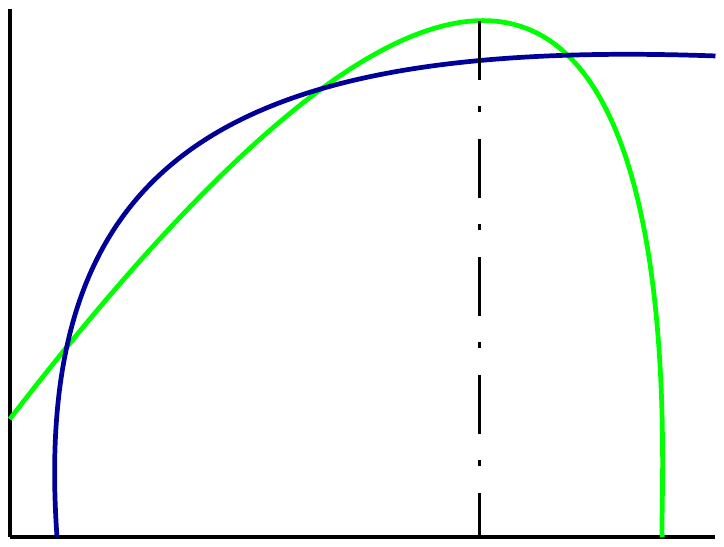}}}
\put(-2,4.4){\sc (a)}	
\put(0,1.45){\sc {\color{orange}$g_1$}}
\put(0,3.35){\sc {\color{violet}$f_1$}}
\put(-2,0){\sc $X_1^{1*}$}
\put(0.35,0,1){\sc $S_1^{in}/k_1$}
\put(1,0.45){\sc $x$}
\put(3.4,4.4){\sc (b)}	
\put(4.4,1.6){\sc {\color{blue}$g_2$}}
\put(4.4,4.1){\sc {\color{green}$f_2$}}
\put(2.1,0){\sc $x_1^m$}
\put(2.7,0.05){\sc $X_2^{1*}$}
\put(5.7,0.08){\sc $S_2^{in}/k_3$}
\put(6.25,0.45){\sc $x$}
\put(9,4.4){{\sc (c)}}
\put(11.3,3.6){\sc {\color{blue}$g_2$}}
\put(10.4,4.2){\sc {\color{green}$f_2$}}
\put(9.85,0.1){\sc $x_1^m$}
\put(7,0.05){\sc $X_2^{1*}$}
\put(10.8,0.1){\sc $S_2^{in}/k_3$}
\put(11.7,0.35){\sc $x$}
\put(7.12,1.7){\sc $\bullet$}
\put(8.9,3.56){\sc $\bullet$}
\put(10.54,3.78){\sc $\bullet$}
\end{picture}
\end{center}
\vspace{-0.3cm}
\caption{Number of positive solutions of equations : (a) $f_1(x)=g_1(x)$ (unique solution); (b) $f_2(x)=g_2(x)$ (unique solution when $X_2^{1*}\geq x_1^m$); (c) $f_2(x)=g_2(x)$ (an odd number of intersections when $X_2^{1*}< x_1^m$).} \label{Figsoleq12}
\end{figure}
\begin{proof}[\bf{Proof of \cref{lemH1}}]
Under \cref{Hypoth1}, the function $x \mapsto f_1(x)$ is nonnegative, decreasing and continuous on $\left[0, S_1^{in}/k_1\right]$. The function $x\mapsto g_1(x)$ is nonnegative, increasing hyperbola and continuous on $\left[X_1^{1*}, +\infty\right)$. Indeed,
$$
g_1'(x)=D_2\frac{X_1^{1*}}{x^2}>0.
$$
Let $X_1^{1*}<S_1^{in}/k_1$ and $F_1$ be the function defined by $F_1(x):=f_1(x)-g_1(x)$. Therefore, $F_1$ is decreasing continuous on $\left[X_1^{1*},S_1^{in}/k_1\right]$. In addition, we have
$$
F_1\left(X_1^{1*}\right)=\mu_1\left(S_1^{in}-k_1X_1^{1*}\right)>0 \quad\mbox{and}\quad F_1\left(\frac{S_1^{in}}{k_1}\right)=-g_1\left(\frac{S_1^{in}}{k_1}\right)<0.
$$
Using the Intermediate Value Theorem, it follows that the equation $F_1(x)=0$ has a unique solution in $\left(X_1^{1*}, S_1^{in}/k_1\right)$ (see \cref{Figsoleq12}(a)).
\end{proof}
\begin{lemma}\label{lemH2}
Let $f_2$ and $g_2$ be the functions defined in \cref{TabFunc2}. Let $X_2^{1*}$ be a positive constant so that $X_2^{1*}<S_2^{in}/k_3$ and $x_1^m:=\left(S_2^{in}-S_2^m\right)/k_3$ be the maximum value of $f_2$ on $\left[0, S_2^{in}/k_3\right]$. Assume that \cref{Hypoth2} holds.
\begin{enumerate}[leftmargin=*]
\item When $X_2^{1*}\geq x_1^m$, then the equation $f_2(x)=g_2(x)$ has a unique solution in $\left(X_2^{1*}, S_2^{in}/k_3\right)$, (see \cref{Figsoleq12}(b)).
\item When $X_2^{1*}<x_1^m$ then the equation $f_2(x)=g_2(x)$ has at least one solution in $\left(X_2^{1*},S_2^{in}/k_3\right)$. Generically, there exist an odd number of solutions in $\left(X_2^{1*},S_2^{in}/k_3\right)$, (see \cref{Figsoleq12}(c)).
\end{enumerate}
\end{lemma}
\begin{proof}[\bf{Proof of \cref{lemH2}}]
Under \cref{Hypoth2}, $f_2$ is nonnegative and continuous on the interval $\left[0, S_2^{in}/k_3\right]$. In addition, $f_2$ is a increasing on $\left[0,x_1^m\right]$ and is decreasing on $\left[x_1^m,S_2^{in}/k_3\right]$. The function $g_2$ is nonnegative and increasing hyperbola continuous on $\left[X_2^{1*},+\infty\right)$. Indeed,
\begin{equation}                                       \label{g2prime}
g_2'(x)=D_2\frac{X_2^{1*}}{x^2}>0.
\end{equation}
Let $X_2^{1*}<S_2^{in}/k_3$ and let $F_2$ be the function defined by $F_2(x):=f_2(x)-g_2(x)$.
Therefore, $F_2$ is continuous on $\left[X_2^{1*},S_2^{in}/k_3\right]$. In addition, we have
$$
F_2\left(X_2^{1*}\right)=\mu_2\left(S_2^{in}-k_3X_2^{1*}\right)>0 \quad\mbox{and}\quad F_2\left(\frac{S_2^{in}}{k_3}\right)=-g_2\left(\frac{S_2^{in}}{k_3}\right)<0.
$$
If $X_2^{1*}>x_1^m$ then $F_2$ is decreasing on $\left[X_2^{1*}, S_2^{in}/k_3\right]$, because $f_2$ is decreasing and $g_2$ is increasing on the same interval. Consequently, the equation $F_2(x)=0$  has a unique solution on $\left(X_2^{1*}, S_2^{in}/k_3\right)$, (see \cref{Figsoleq12}(b)). If $X_2^{1*}<x_1^m$, then $F_2$ can be nonmonotonic. We complete the proof by using the Intermediate Value Theorem.
\end{proof}
\begin{lemma}                                            \label{lemH3}
Let $f_3$ and $g_2$ be the functions defined in \cref{TabFunc2}. Let $d:=\left(S_2^{in}+k_2X_1^{2*}\right)/k_3$ be the root of $f_3(x)=0$ and $x_2^m:=\left(S_2^{in}+k_2X_1^{2*}-S_2^m\right)/k_3$ be the maximum value of $f_3$.
Assume that $X_2^{1*}<d$ and \cref{Hypoth2} holds.
\begin{enumerate}[leftmargin=*]
\item When $X_2^{1*}\geq x_2^m$ then the equation $f_3(x)=g_2(x)$ has a unique solution in $\left(X_2^{1*},d\right)$ (see \cref{Figsoleq3}(a)).
\item When $X_2^{1*}<x_2^m$ then the equation $f_3(x)=g_2(x)$ has at least one solution in $\left(X_2^{1*},d\right)$. Generically, there exist an odd number of solutions in $\left(X_2^{1*},d\right)$, (see \cref{Figsoleq3}(b-c)).
\end{enumerate}
\end{lemma}
\begin{figure}[!ht]
\setlength{\unitlength}{1.0cm}
\begin{center}
\begin{picture}(8.8,4.7)(0,0)
\put(-5.3,-6.7){{\includegraphics[width=8.5cm,height=12.5cm]{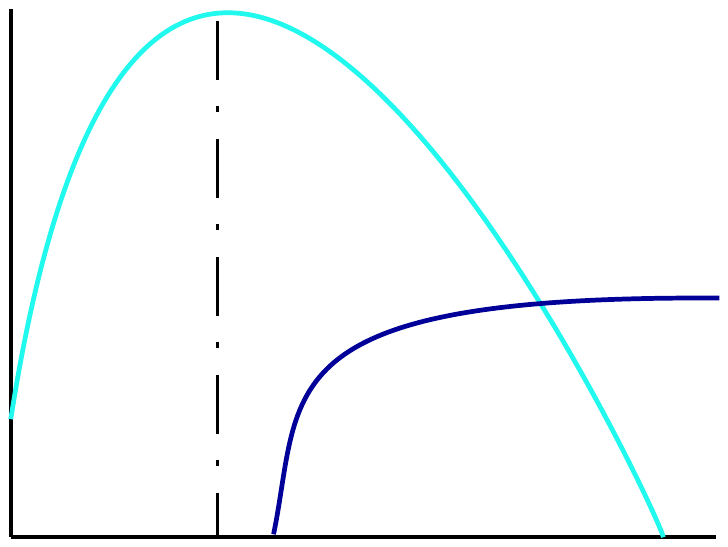}}}	
\put(1.5,0){{\includegraphics[width=5.5cm,height=4.5cm]{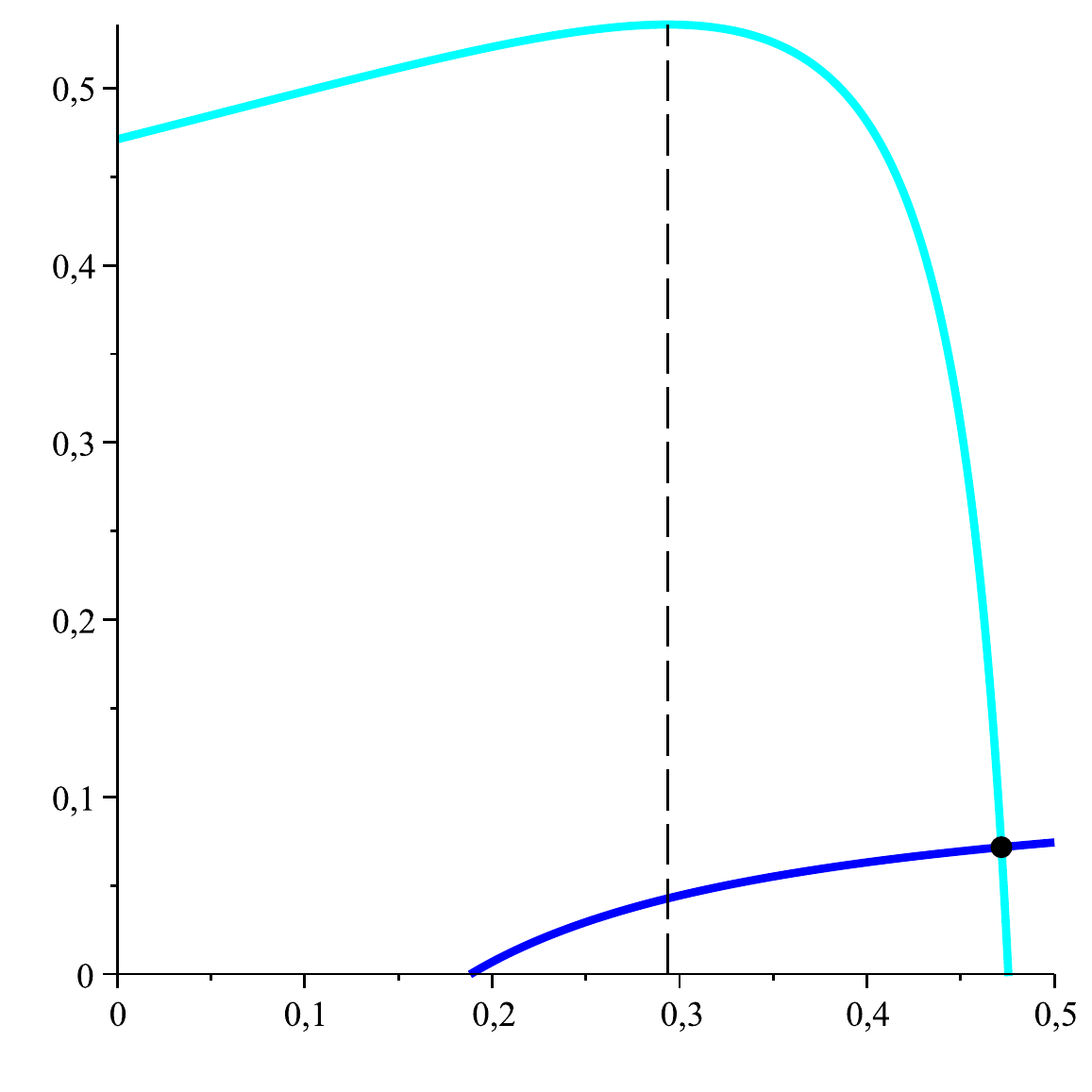}}}
\put(6.7,0){{\includegraphics[width=5.5cm,height=4.5cm]{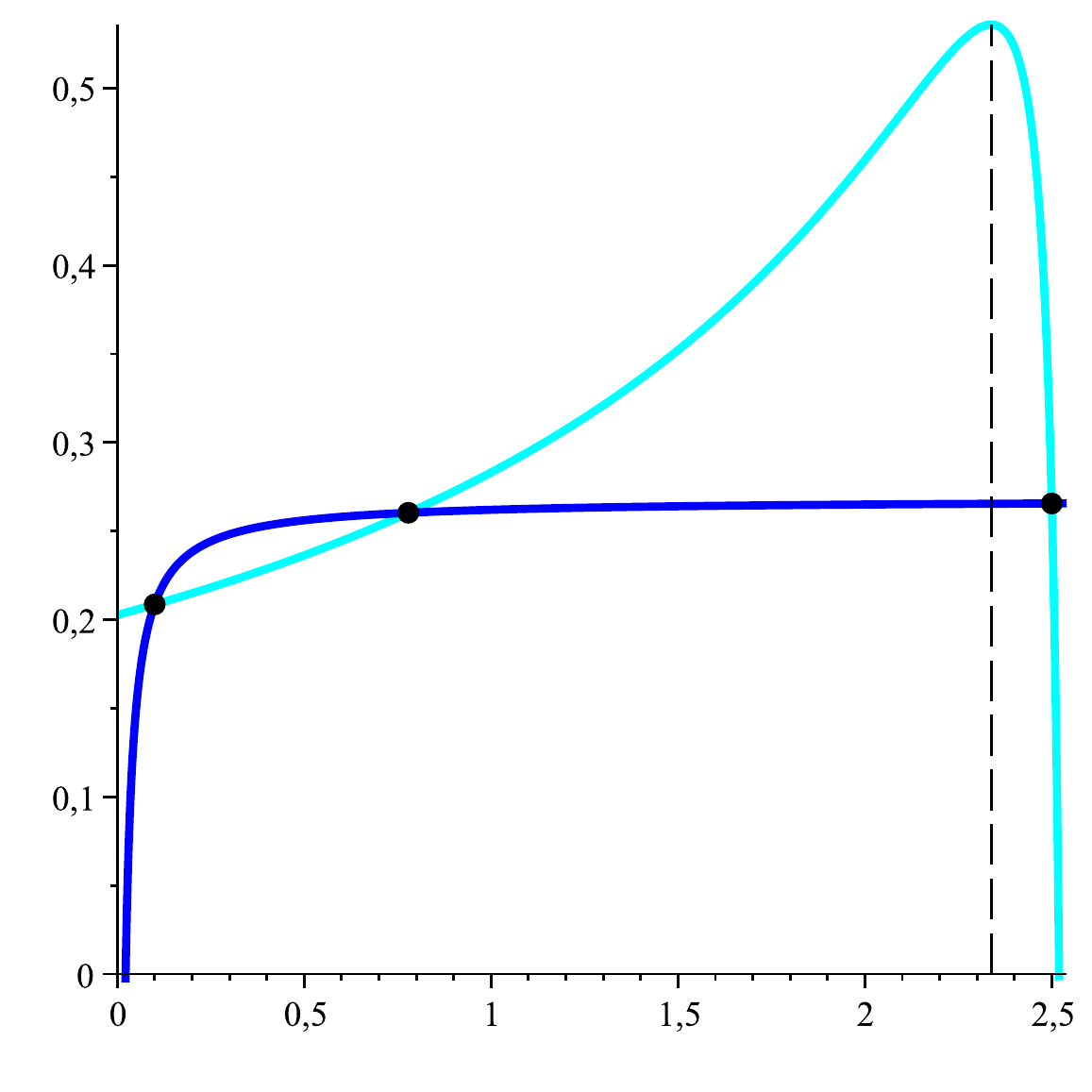}}}
\put(-2.3,4.5){\sc (a)}	
\put(0.5,2.2){\sc {\color{blue}$g_2$}}
\put(-0.6,3.3){\sc {\color{cyan}$f_3$}}
\put(-2.35,0.05){\sc $x_2^m$}
\put(-1.8,0.05){\sc $X_2^{1*}$}
\put(0.85,0.05){\sc $d$}
\put(1.4,0.4){\sc $x$}
\put(0.06,2.02){\sc $\bullet$}
\put(4.2,4.5){\sc (b)}	
\put(5.7,1.1){\sc {\color{blue}$g_2$}}
\put(5.7,4.3){\sc {\color{cyan}$f_3$}}
\put(4.7,0){\sc $x_2^m$}
\put(3.7,0){\sc $X_2^{1*}$}
\put(6.5,0.05){\sc $d$}
\put(6.8,0.45){\sc $x$}
\put(9.7,4.5){{\sc (c)}}
\put(10.7,2.55){\sc {\color{blue}$g_2$}}
\put(10.7,3.9){\sc {\color{cyan}$f_3$}}
\put(11.4,0.0){\sc $x_2^m$}
\put(7.2,0){\sc $X_2^{1*}$}
\put(12,0){\sc $d$}
\put(12.2,0.45){\sc $x$}
\end{picture}
\end{center}
\vspace{-0.3cm}
\caption{Number of positive solutions of equation $f_3(x)=g_2(x)$ on $(0,d)$ where $d$ is defined in \cref{lemH3}: (a) unique solution when $X_2^{1*}\geq x_2^m$; (b-c) an odd number of intersections when $X_2^{1*}<x_2^m$.} \label{Figsoleq3}
\end{figure}
\begin{proof}[\bf{Proof of \cref{lemH3}}]
Under \cref{Hypoth2}, $f_3$ is nonnegative and continuous on $[0,d]$ with $d:= \left(S_2^{in}+k_2X_1^{2*}\right)/k_3$. In addition, $f_3$ is increasing on $\left[0,x_2^m\right]$ and is decreasing on $\left[x_2^m,d\right]$. Let $X_2^{1*}<d$ and $F_3$ be the function defined by $F_3(x):=f_3(x)-g_2(x)$. Therefore, $F_3$ is continuous on $\left[X_2^{1*},d\right]$. In addition, we have
$$
F_3\left(X_2^{1*}\right)=\mu_2\left(S_2^{in}+k_2X_1^{2*}-k_3X_2^{1*}\right)>0,\quad F_3\left(\frac{S_2^{in}+k_2X_1^{2*}}{k_3}\right)=-g_2\left(\frac{S_2^{in}+k_2X_1^{2*}}{k_3}\right)<0.
$$
Recall from \cref{g2prime} that $g_2$ is increasing on $\left[X_2^{1*},+\infty\right)$. If $X_2^{1*}\geq x_2^m$ then $F_3$ is decreasing on $\left[X_2^{1*},d\right]$, because $f_3$ is decreasing on the same interval. Consequently, the equation $F_3(x)=0$ has a unique solution in $\left(X_2^{1*},d\right)$. If $X_2^{1*}<x_2^m$  then $F_3$ can be nonmonotonic. Using the Intermediate Value Theorem, the proof is completed.
\end{proof}
\begin{proof}[\bf{Proof of \cref{propES}}]
Similarly to system \cref{ModelAM2}, subsystem \cref{ModelAM2SR4} has a cascade structure.
Thus, for each steady state of the upper two-dimensional system of \cref{ModelAM2SR4} denoted by
$$
E_{ij}\left(X_1^1,X_2^1\right), \quad \{i,j\}\in \{0,1,2\}
$$
corresponds a steady state of the lower two-dimensional system of \cref{ModelAM2SR4} denoted by
$$
E_{ij}^{kl}\left(X_1^2,X_2^2\right).
$$
This steady state is given by putting the left-hand side of the lower two-dimensional subsystem of \cref{ModelAM2SR4} equal to zero when the two components $X_1^1$ and $X_2^1$ are at the steady state $E_{ij}$.
Then, we can deduce the corresponding steady state of the four-dimensional system \cref{ModelAM2SR4} denoted by $\mathcal{E}_{ij}^{kl}\left(X_1^1, X_2^1, X_1^2, X_2^2\right)$ as well as their existence condition by checking those of $E_{ij}$ and $E_{ij}^{kl}$ simultaneously.
Indeed, the upper two-dimensional system of \cref{ModelAM2SR4} corresponds to the classic AM2 model where the steady states are given by the solutions of the following equations:
 \begin{align}
\left[\mu_1\left(S_1^{in}-k_1X_1^1\right) - D_1\right]X_1^1          &=0,  \label{eq1Am2} \\
\left[\mu_2\left(S_2^{in}+k_2X_1^1-k_3X_2^1\right) - D_1\right]X_2^1 &=0.  \label{eq2Am2}
\end{align}
\begin{table}[!ht]
\begin{center}
\caption{Two components $\left(X_1^1, X_2^1\right)$ of all steady states and their existence condition for the upper two-dimensional system of \cref{ModelAM2SR4}. The functions $\lambda_1^1$, $\lambda_2^{1i}$, and $F_{1i}$, $i=1,2$ are defined in \cref{TabFunc1}.} \label{TableAM2}
\begin{tabular}{@{\hspace{1mm}}l@{\hspace{2mm}} @{\hspace{2mm}}l@{\hspace{2mm}}  @{\hspace{2mm}}l@{\hspace{2mm}} @{\hspace{2mm}}l@{\hspace{1mm}} }	
\hline
             & $X_1^1$    & $X_2^1$                               & Existence conditions    \\ \Xhline{3\arrayrulewidth}

${E}_{00}$  &      0      &   0                                    & Always exists  \\
$E_{10}$    & $\frac{1}{k_1}\left(S_1^{in}-\lambda_1^1\right)$ & 0 & $S_1^{in}>\lambda_1^1$ \\
$E_{0i}$    &    0        &  $\frac{1}{k_3}\left(S_2^{in}-\lambda_2^{1i}\right)$  & $S_2^{in}>\lambda_2^{1i}$  \\
$E_{1i}$    & $\frac{1}{k_1}\left(S_1^{in}-\lambda_1^1\right)$  & $\frac{k_2}{k_1k_3}\left(S_1^{in}-F_{1i}\right)$ & $S_1^{in}>\lambda_1^1$ and $S_1^{in}>F_{1i}$  \\ \hline
\end{tabular}
\end{center}
\end{table}
\begin{itemize}[leftmargin=*]
\item For $E_{00}$, $X_1^1=X_2^1=0$. It always exists.
\item For $E_{10}$, $X_1^1>0$ and $X_2^1=0$. Thus, \cref{eq1Am2} results in $\mu_1\left(S_1^{in}-k_1X_1^1\right) = D/r$. Using \cref{Hypoth1}, we obtain $X_1^1$ and the existence condition in
\cref{TableAM2}.
\item For $E_{0i}$, $i=1,2$, $X_1^1=0$ and $X_2^1>0$. Hence, \cref{eq2Am2} results in $\mu_2\left(S_2^{in}-k_3X_2^1\right) = D/r$. Using \cref{Hypoth2}, we obtain $X_2^1$ and the existence condition in  \cref{TableAM2}.
\item For $E_{1i}$, $i=1,2$, $X_1^1>0$ and $X_2^1>0$. From \cref{eq1Am2} and \cref{eq2Am2}, it follows that
 $$
 \mu_1\left(S_1^{in}-k_1X_1^1\right) = D_1 \quad\mbox{and}\quad \mu_2\left(S_2^{in}+k_2X_1^1-k_3X_2^1\right) = D_1.
 $$
Using the expressions of $\lambda_1^1$, $\lambda_2^{1i}$, and $F_{1i}$, $i=1,2$ in \cref{TabFunc1}, we obtain the two components of $E_{1i}$ and the two existence conditions in  \cref{TableAM2}.
\end{itemize}

The steady states of the lower two-dimensional system of \cref{ModelAM2SR4} are given by the solutions of the following equations:
\begin{align}
\mu_1\left(S_1^{in}-k_1X_1^2\right)X_1^2 + D_2\left(X_1^1- X_1^2\right)          &=0,     \label{eq3Am2s1}\\
\mu_2\left(S_2^{in}+k_2X_1^2-k_3X_2^2\right)X_2^2 + D_2\left(X_2^1- X_2^2\right) &=0.     \label{eq4Am2s2}
\end{align}
\begin{table}[!ht]
\begin{center}
\caption{Two components $\left(X_1^2, X_2^2\right)$ and existence condition of all steady states of the lower two-dimensional system of \cref{ModelAM2SR4}. The functions $\lambda_1^2$, $\lambda_2^{2i}$, and $F_{2i}$, $i=1,2$ are defined in \cref{TabFunc1}. The functions $f_1$, $f_2$, $f_3$, $g_1$ and $g_2$ are defined in \cref{TabFunc2} while $X_1^{2*}$ is the unique solution of equation $f_1(x)=g_1(x)$. } \label{TableAM2S}
\begin{tabular}{@{\hspace{1mm}}l@{\hspace{2mm}} @{\hspace{2mm}}l@{\hspace{2mm}} @{\hspace{2mm}}l@{\hspace{2mm}} @{\hspace{2mm}}l@{\hspace{1mm}}}		
\hline
                &   $X_1^2$   &   $X_2^2$             & Existence conditions         \\ \Xhline{3\arrayrulewidth}
$E_{00}^{00}$   &      0      &      0                & Always exists  \\
$E_{00}^{10}$   &  $\frac{1}{k_1}\left(S_1^{in}-\lambda_1^2\right)$      &      0    & $S_1^{in}>\lambda_1^2$  \\
$E_{00}^{0i}$   &      0      & $\frac{1}{k_3}\left(S_2^{in}-\lambda_2^{2i}\right)$  & $S_2^{in}>\lambda_2^{2i}$    \\
$E_{00}^{1i}$   &  $\frac{1}{k_1}\left(S_1^{in}-\lambda_1^2\right)$ & $\frac{k_2}{k_1k_3}\left(S_1^{in}-F_{2i}\right)$ & $S_1^{in}>\lambda_1^2$~~\mbox{and}~~$S_1^{in}>F_{2i}$  \\
$E_{10}^{10}$   &  $X_1^{2*}$ &      0                & Always exists \\
$E_{10}^{1i}$   &  $X_1^{2*}$ &  $\frac{1}{k_3}\left(S_2^{in}+k_2X_1^{2*}-\lambda_2^{2i}\right)$  & $S_2^{in}+k_2 X_1^{2*}>\lambda_2^{2i}$  \\
$E_{0i}^{01}$   &      0      &  a solution of $f_2\left(X_2^2\right)=g_2\left(X_2^2\right)$      & Always exists \\
$E_{0i}^{11}$   & $\frac{1}{k_1}\left(S_1^{in}-\lambda_1^2\right)$  &  a solution of $f_3(x)=g_2(x)$ & $S_1^{in}>\lambda_1^2$  \\
$E_{1i}^{11}$   &  $X_1^{2*}$ &  a solution of $f_3(x)=g_2(x)$      & Always exists  \\ \hline
\end{tabular}
\end{center}
\end{table}
\begin{itemize}[leftmargin=*]
\item For $E_{00}$, we have $X_1^1=X_2^1=0$. Hence, \cref{eq3Am2s1} and \cref{eq4Am2s2} result in
 \begin{align}
\left[\mu_1\left(S_1^{in}-k_1X_1^2\right) - D_2\right]X_1^2         &=0,      \label{eq3Am2s} \\
\left[\mu_2\left(S_2^{in}+k_2X_1^2-k_3X_2^2\right) -D_2\right]X_2^2 &=0.      \label{eq4Am2s}
\end{align}
\begin{itemize}[leftmargin=*,label=\raisebox{0.25ex}{\tiny$\bullet$}]
\item For $E_{00}^{00}$, $X_1^2=X_2^2=0$. It always exists.
\item For $E_{00}^{10}$, $X_1^2>0$ and $X_2^2=0$. Thus, \cref{eq3Am2s} results in $\mu_1\left(S_1^{in}-k_1X_1^2\right) =D/(1-r)$.  Using \cref{Hypoth1}, we obtain $X_1^2$ and the existence condition in  \cref{TableAM2S}.
\item For $E_{00}^{0i}$, $i=1,2$, $X_1^2=0$ and $X_2^2>0$. Hence, \cref{eq4Am2s} results in $\mu_2\left(S_2^{in}-k_3X_2^2\right) =  D/(1-r)$. Using \cref{Hypoth2}, we obtain $X_2^2$ and the existence condition in  \cref{TableAM2S}.
\item For $E_{00}^{1i}$, $i=1,2$, $X_1^2>0$ and $X_2^2>0$. Thus, \cref{eq3Am2s} and \cref{eq4Am2s} result in
$$
 \mu_1\left(S_1^{in}-k_1X_1^2\right) = D_2\quad\mbox{and}\quad \mu_2\left(S_2^{in}+k_2X_1^2-k_3X_2^2\right) = \frac{ D}{1-r}.
$$
Using \cref{Hypoth1,Hypoth2} and the expressions of $\lambda_1^2$, $\lambda_2^{2i}$, and $F_{2i}$ in \cref{TabFunc1},
we obtain two components $\left(X_1^2, X_2^2\right)$ and the two existence conditions in  \cref{TableAM2S}.
\end{itemize}
\item For $E_{10}$, we have $X_1^1>0$ and $X_2^1=0$. Hence, \cref{eq3Am2s1} and \cref{eq4Am2s2} result in
 \begin{align}
\mu_1\left(S_1^{in}-k_1X_1^2\right)X_1^2 + D_2\left(X_1^1- X_1^2\right) &=0,      \label{eq5Am2s}\\
\left[\mu_2\left(S_2^{in}+k_2X_1^2-k_3X_2^2\right) -D_2\right]X_2^2     &=0.      \label{eq6Am2s}
\end{align}
\begin{itemize}[leftmargin=*,label=\raisebox{0.25ex}{\tiny$\bullet$}]
\item For $E_{10}^{00}$, $X_1^2=X_2^2=0$. Then \cref{eq5Am2s} implies that $X_1^1=0$, which is in contradiction with the existence condition of $E_{10}$ where $X_1^1>0$. Hence, $E_{10}^{00}$ does not exist.
\item For $E_{10}^{10}$, $X_1^2>0$ and $X_2^2=0$. Thus, \cref{eq5Am2s} results in
\begin{equation}                                        \label{eqf1g1}
\mu_1\left(S_1^{in}-k_1X_1^2\right)= D_2\left(\frac{X_1^2-X_1^1}{X_1^2}\right),
\end{equation}
which is equivalent to $f_1(x)=g_1(x)$ (see \cref{TabFunc2}). From the component $X_1^1$ of $E_{10}$ in \cref{TableAM2}, it follows that $X_1^1<S_1^{in}/k_1$.
From \cref{lemH1}, it follows that the equation $f_1(x)=g_1(x)$ has a unique solution $X_1^{2*}\in\left(X_1^1,S_1^{in}/k_1\right)$ (see \cref{Figsoleq12}(a)). Thus, we obtain the components of steady state $E_{10}^{10}$ in \cref{TableAM2S}, which always exists.
\item For $E_{10}^{0i}$, $i=1,2$, $X_1^2=0$ and $X_2^2>0$. Then \cref{eq5Am2s} implies that $X_1^1=0$, which is in contradiction with the existence condition of $E_{10}$ where $X_1^1>0$. Hence, $E_{10}^{0i}$ does not exist.
\item For $E_{10}^{1i}$, $i=1,2$, $X_1^2>0$ and $X_2^2>0$. Thus, \cref{eq5Am2s} is equivalent to \cref{eqf1g1}. Similarly to the steady state $E_{10}^{10}$, equation \cref{eqf1g1} has a unique solution $X_1^{2*}\in \left(X_1^{1*},~S_1^{in}/k_1\right)$. In addition, \cref{eq6Am2s} implies that
$$
\mu_2\left(S_2^{in}+k_2X_1^{2*}-k_3X_2^2\right)= D_2.
$$
From \cref{Hypoth2} and the expression of $\lambda_2^{2i}$ in \cref{TabFunc1}, we obtain the component $X_2^2$ of the steady state $E_{10}^{1i}$ in \cref{TableAM2S}.
Moreover, the existence condition in \cref{TableAM2S} follows.
\end{itemize}
\item For $E_{0i}$, we have $X_1^1=0$ and $X_2^{1i} > 0$, $i=1,2$. Hence, \cref{eq3Am2s1} and \cref{eq4Am2s2} result in
 \begin{align}
\left[\mu_1\left(S_1^{in}-k_1X_1^2\right)- D_2 \right]X_1^2                         &=0,      \label{eq7Am2s1}\\
\mu_2\left(S_2^{in}+k_2X_1^2-k_3X_2^2\right)X_2^2 + D_2\left(X_2^{1i}- X_2^2\right) &=0.      \label{eq8Am2s2}
\end{align}
\begin{itemize}[leftmargin=*,label=\raisebox{0.25ex}{\tiny$\bullet$}]
\item For $E_{0i}^{00}$ [resp. $E_{0i}^{10}$], $X_1^2=X_2^2=0$ [resp. $X_1^2>0$ and $X_2^2=0$]. Then \cref{eq8Am2s2} implies that $X_2^{1i}=0$, which is in contradiction with $X_2^{1i}>0$. Hence, $E_{0i}^{00}$ and $E_{0i}^{10}$ does not exist.
\item For $E_{0i}^{01}$, we have $X_1^2=0$ and $X_2^2>0$. Then, \cref{eq8Am2s2} results in $X_2^2$ must be a solution of equation
$$
\mu_2\left(S_2^{in}-k_3X_2^2\right)= D_2\left(\frac{X_2^2-X_2^{1i}}{X_2^2}\right),
$$
which is equivalent to $f_2(X_2^2)=g_2(X_2^2)$ (see \cref{TabFunc2}). From the component $X_2^{1i}$ of $E_{0i}$ in \cref{TableAM2}, we have $X_2^{1i}<S_2^{in}/k_3$.
From \cref{lemH2}, it follows that the equation $f_2(X_2^2)=g_2(X_2^2)$ has at least one solution $X_2^{2*}\in \left(X_2^{1i}, S_2^{in}/k_3\right)$. Hence, we obtain the components of steady state $E_{0i}^{01}$ in \cref{TableAM2S}, where there exists at least one.
\item For $E_{0i}^{11}$, we have $X_1^2>0$ and $X_2^2>0$. Then, the component $X_1^2$ in \cref{TableAM2S} follows from equation \cref{eq7Am2s1} and \cref{Hypoth1}. In addition, \cref{eq8Am2s2} implies that $X_2^2$ must be a solution of equation
\begin{equation}                                        \label{Eqf3g2}
\mu_2\left(S_2^{in}+k_2X_1^2-k_3{X_2^2}\right)=D_2\left(\frac{X_2^2-X_2^{1i}}{X_2^2}\right),
\end{equation}
which is equivalent to $f_3(X_2^2)=g_2(X_2^2)$ (see \cref{TabFunc2}). Using the expression of $d$ in \cref{lemH3} and the component of $X_2^{1i}$ in \cref{TableAM2}, we see that $X_2^{1i}<d$. From \cref{lemH3}, it follows that the equation $f_3(X_2^2)=g_2(X_2^2)$ has at least one solution $X_2^{2*}\in \left(X_2^{1i}, d\right)$ (see \cref{Figsoleq3}). Consequently, we obtain the two components and the existence condition of the steady state $E_{0i}^{11}$ in \cref{TableAM2S}.
\end{itemize}
\item For $E_{1i}$, we have $X_1^1>0$ and $X_2^{1i} > 0$, $i=1,2$.
\begin{itemize}[leftmargin=*,label=\raisebox{0.25ex}{\tiny$\bullet$}]
\item For $E_{1i}^{00}$ [resp. $E_{1i}^{10}$, and $E_{1i}^{01}$,], $X_1^2=X_2^2=0$ [resp. ($X_1^2>0$ and $X_2^2=0$) and ($X_1^2=0$ and $X_2^2>0$)]. Then, \cref{eq3Am2s1} or \cref{eq4Am2s2} implies that $X_1^1=0$ or $X_2^{1i}=0$, which is a contradiction. Hence, these three types of steady states do not exist.
\item For $E_{1i}^{11}$, we have $X_1^2>0$ and $X_2^2>0$. Then, \cref{eq3Am2s1} and \cref{eq4Am2s2} are equivalent to \cref{eqf1g1} and \cref{Eqf3g2}, or, equivalently, $f_1(x)=g_1(x)$ and $f_3(x)=g_2(x)$, respectively.
Similarly to the steady state $E_{10}^{10}$, we obtain the component $X_1^2$ which is the unique solution of equation $f_1(x)=g_1(x)$.
Using the expression of $X_1^{1*}$ and $X_2^{1i*}$ in \cref{TableAM2} and the expressions of $F_{1i}$ in \cref{TabFunc1}, it follows that
$$
X_2^{1i*}=\frac{1}{k_3}\left(k_2 X_1^{1*}-\lambda_2^{1i}+S_2^{in}\right).
$$
From the expressions of $d$ in \cref{lemH3}, we have
$$
X_2^{1i*}-d=\frac{1}{k_3}\left(k_2\left(X_1^{1*}-X_1^{2*}\right)-\lambda_2^{1i}\right),
$$
which is negative since $X_1^{1*}<X_1^{2*}$ (see \cref{lemH1}).
From \cref{lemH3}, it follows that the equation $f_3(x)=g_2(x)$ has at least one solution in $\left(X_2^{1i*}, d\right)$.
Therefore, we obtain the two components of the steady state $E_{1i}^{11}$ in \cref{TableAM2S}, which always exists.
\end{itemize}
\end{itemize}
\end{proof}
\begin{proof}[\bf{Proof of \cref{propmultip}}]
For $k=0,1$ and $l=0,1,2$, the uniqueness of the six first steady states $\mathcal{E}_{00}^{kl}$ and $\mathcal{E}_{10}^{1l}$, follows from \cref{TabComSS}.
For $\mathcal{E}_{01}^{01}$, we have $\lambda_2^{11}<S_2^m$, that is, $X_2^{1*}>\left(S_2^{in}-S_2^m\right)/k_3:=x_1^m$ where $X_2^{1*}$ is defined in \cref{TabComSS} and $x_1^m$ in \cref{lemH2}, implying that the equation $f_2\left(X_2^2\right)=g_2\left(X_2^2\right)$ has a unique solution in $\left(X_2^{1*},S_2^{in}/k_3\right)$ (see \cref{Figsoleq12}(b)). That is, $\mathcal{E}_{01}^{01}$ is unique when it exists.
Inversely, for $\mathcal{E}_{02}^{01}$, we have $\lambda_2^{12}>S_2^m$, that is, $X_2^{1*}<x_1^m$. From \cref{lemH2}, the equation $f_2\left(X_2^2\right)=g_2\left(X_2^2\right)$ has at least one solution on $\left(X_2^{1*}, S_2^{in}/k_3\right)$ (see \cref{Figsoleq12}(b-c)).
For $\mathcal{E}_{0i}^{11}$, $i=1,2$, the $X_2^2$-component is a solution of $f_3(x)=g_2(x)$. Using the expressions of $X_2^{1*}$ and $X_1^{2*}$ in \cref{TabComSS} and the expressions of $x_2^m$ in \cref{lemH3}, a straightforward calculation shows that
$$
X_2^{1*}-x_2^m=\frac{1}{k_3}\left(S_2^m-\lambda_2^{11}-\frac{k_2}{k_1}\left(S_1^{in}-\lambda_1^2\right)\right).
$$
As $S_2^m>\lambda_2^{11}$ and $S_1^{in}>\lambda_1^2$, the two case $X_2^{1*}\geq x_2^m$ and $X_2^{1*}<x_2^m$ can be hold. From \cref{lemH3}, the equation $f_3(x)=g_2(x)$ has at least one solution in $\left(X_2^{1*}, d\right)$ (see \cref{Figsoleq3}).
The multiplicity of the other steady states $\mathcal{E}_{11}^{11}$ and $\mathcal{E}_{12}^{12}$ follows similar arguments.
\end{proof}
\begin{proof}[\bf{Proof of \cref{propstab}}]
Let $J$ be the Jacobian matrix of \cref{ModelAM2SR4} at $\left(X_1^1, X_2^1, X_1^2, X_2^2\right)$, which is given by the following lower block triangular matrix of the form
\begin{equation*}
J=
\begin{bmatrix}
J_1 & 0   \\
J_2 & J_3
\end{bmatrix};\quad
J_1=
\begin{bmatrix}
a_{11} & 0     \\
a_{21} & a_{22}
\end{bmatrix}, \quad
J_2=
\begin{bmatrix}
a_{31} & 0     \\
0      & a_{42}
\end{bmatrix},\quad
J_3=
\begin{bmatrix}
a_{33} & 0     \\
a_{43} & a_{44}
\end{bmatrix}
\end{equation*}
where
\begin{equation*}
\begin{array}{ll}
a_{11}=\mu_1\left(S_1^{in}-k_1X_1^1\right) - D_1-k_1\mu_1'\left(S_1^{in}-k_1X_1^1\right)X_1^1, \quad  a_{21}=k_2\mu_2'\left(S_2^{in}+k_2X_1^1-k_3X_2^1\right)X_2^1,  \\
a_{22}= \mu_2\left(S_2^{in}+k_2X_1^1-k_3X_2^1\right)-D_1-k_3\mu_2'\left(S_2^{in}+k_2X_1^1-k_3X_2^1\right)X_2^1, \quad  a_{31}=D_2, \\
a_{33}=\mu_1\left(S_1^{in}-k_1X_1^2\right) - D_2-k_1\mu_1'\left(S_1^{in}-k_1X_1^2\right)X_1^2,                 \quad  a_{43}=k_2\mu_2'\left(S_2^{in}+k_2X_1^2-k_3X_2^2\right)X_2^2,   \\
a_{44}=\mu_2\left(S_2^{in}+k_2X_1^2-k_3X_2^2\right)-D_2-k_3\mu_2'\left(S_2^{in}+k_2X_1^2-k_3X_2^2\right)X_2^2, \quad a_{42}=D_2.
\end{array}
\end{equation*}

Note that, $J_1$ is the Jacobian matrix of the upper two-dimensional subsystem of \cref{ModelAM2SR4} evaluated at $\left(X_1^1, X_2^1\right)$, and $J_3$ is the Jacobian matrix of the lower two-dimensional subsystem of \cref{ModelAM2SR4} evaluated at $\left(X_1^2, X_2^2\right)$. Thus, the stability of the steady
state is determined by the sign of the real parts of the eigenvalues of $J_1$ and $J_3$. A steady state is LES if the real parts of the eigenvalues of $J_1$ and $J_3$ at this steady state are negative.
Since they are lower triangular matrices, a steady state is LES if the corresponding eigenvalues $a_{ii}$, $i=1,\ldots,4$.

To simplify the presentation of the proof, we first determine the stability conditions of the steady states of the upper two-dimensional subsystem of \cref{ModelAM2SR4} which corresponds to the AM2 model.
Then, we will establish the stability conditions of the steady states of the lower two-dimensional subsystem of \cref{ModelAM2SR4}. Finally, the stability of the complete system \cref{ModelAM2SR4} follows by combining all conditions.
\begin{table}[h!]
\begin{center}
\caption{The stability conditions of all steady states of the upper two-dimensional subsystem of \cref{ModelAM2SR4}, where the existence condition are defined in \cref{TableAM2} and the functions $\lambda_1^1$ and $\lambda_2^{1i}$ are defined in \cref{TabFunc1}.}\label{TableStabModelAM2S1S}
\begin{tabular}{@{\hspace{2mm}}l@{\hspace{2mm}} @{\hspace{2mm}}l@{\hspace{2mm}} }		
\hline
            &   Stability conditions        \\ \Xhline{3\arrayrulewidth}
$E_{00}$    &      $S_1^{in}<\lambda_1^1$ and $S_2^{in}\notin\left[\lambda_2^{11}, \lambda_2^{12}\right]$   \\
$E_{10}$    &      $S_2^{in}+\frac{k_2}{k_1}\left(S_1^{in}-\lambda_1^1\right) \notin\left[\lambda_2^{11}, \lambda_2^{12}\right]$   \\
$E_{01}$    &     $S_1^{in}<\lambda_1^1$   \\
$E_{02}$    &     Always unstable   \\
$E_{11}$    &     LES when it exists   \\
$E_{12}$    &     Always unstable     \\ \hline
\end{tabular}
\end{center}
\end{table}
\begin{itemize}[leftmargin=*]
\item For $E_{00}=(0,0)$, the two eigenvalues of $J_1$ are
$$
a_{11}=\mu_1\left(S_1^{in}\right)-D_1 \quad\mbox{and}\quad a_{22}=\mu_2\left(S_2^{in}\right)-D_1.
$$
From \cref{Hypoth1,Hypoth2} and the expressions of $\lambda_1^1$, $\lambda_2^{11}$ and $\lambda_2^{12}$ in \cref{TabFunc1}, ${E}_{00}$ is LES if and only if $a_{11}$ and $a_{22}$ are negative, that is, the stability conditions of ${E}_{00}$ in
\cref{TableStabModelAM2S1S} hold.
\item For $E_{10}=\left(X_1^{1*}, 0\right)$, where $X_1^{1*}=(S_1^{in}-\lambda_1^1)/k_1$, the two eigenvalues of $J_1$ are
$$
a_{11}=-k_1\mu_1'\left(\lambda_1^1\right)X_1^{1*}\quad\mbox{and}\quad a_{22}= \mu_2\left(S_2^{in}+\frac{k_2}{k_1}\left(S_1^{in}-\lambda_1^1\right)\right)-D_1.
$$
From \cref{Hypoth1}, $a_{11}$ is negative. From \cref{Hypoth2}, $a_{22}$ is negative if and only if the stability condition of ${E}_{10}$ in \cref{TableStabModelAM2S1S} holds.
\item For $E_{0i}=\left(0,~X_2^{1i*}\right)$, $i=1,2$, where $X_2^{1i*}=\left(S_2^{in}-\lambda_2^{1i}\right)/k_3$, the two eigenvalues of $J_1$ are
$$
a_{11}=\mu_1\left(S_1^{in}\right)-D_1 \quad\mbox{and}\quad a_{22}=-k_3\mu_2'\left(\lambda_2^{1i}\right)X_2^{1i*}.
$$
Using \cref{Hypoth2}, $E_{02}$ is unstable when it exists since $a_{22}$ is positive. However, $E_{01}$ is LES if and only if $S_1^{in}<\lambda_1^1$ since $a_{22}$ is negative.
\item For $E_{1i}=\left(X_1^{1*}, X_2^{1i*}\right)$, $i=1,2$, where $X_1^{1*}=(S_1^{in}-\lambda_1^1)/k_1$ and $X_2^{1i*}=\left(S_2^{in}-\lambda_2^{1i}\right)/k_3$, the two eigenvalues of $J_1$ are
$$
a_{11}=-k_1\mu_1'\left(\lambda_1^1\right)X_1^{1*}\quad\mbox{and}\quad a_{22}=-k_3\mu_2'\left(\lambda_2^{1i}\right)X_2^{1i*}.
$$
From \cref{Hypoth2}, $E_{12}$ is unstable when it exists since $a_{22}$ is positive. However, $E_{11}$ is LES when it exists since $a_{11}$ and $a_{22}$ are negative.
\end{itemize}

Let us now consider the lower two-dimensional system of \cref{ModelAM2SR4}. From \cref{TableStabModelAM2S1S}, it follows that the steady states $E_{02}^{01}$, $E_{02}^{11}$ and $E_{12}^{11}$ are unstable.
\begin{table}[!ht]
\begin{center}
\caption{The stability conditions of all steady states of the lower two-dimensional system of \cref{ModelAM2SR4}. The functions $f_1$, $f_3$, $g_1$ and $g_2$ are defined in \cref{TabFunc2} while $X_1^{2*}$ is a unique solution of equation $f_1(x)=g_1(x)$ and $X_2^{2*}$ is a solution of equation $f_3(x)=g_2(x)$.} \label{TableStabModelAM2S2S}
\begin{tabular}{@{\hspace{2mm}}l@{\hspace{2mm}} @{\hspace{2mm}}l@{\hspace{2mm}} }		
\hline
                 &   Stability conditions           \\ \Xhline{3\arrayrulewidth}
$E_{00}^{00}$    &      $S_1^{in}<\lambda_1^2$ and $S_2^{in}\notin\left[\lambda_2^{21}, \lambda_2^{22}\right]$    \\
$E_{00}^{01}$    &      $S_1^{in}<\lambda_1^2$      \\
$E_{00}^{02}$    &      Always unstable             \\
$E_{00}^{10}$    &      $S_2^{in}+\frac{k_2}{k_1}\left(S_1^{in}-\lambda_1^2\right) \notin\left[\lambda_2^{21}, \lambda_2^{22}\right]$            \\
$E_{00}^{11}$    &      LES when it exists          \\
$E_{00}^{12}$    &      Always unstable             \\
$E_{10}^{10}$    &     $S_2^{in}+k_2X_1^{2*}\notin\left[\lambda_2^{21},\lambda_2^{22}\right]$            \\
$E_{10}^{11}$    &     LES when it exists           \\
$E_{10}^{12}$    &     Always unstable             \\
$E_{01}^{01}$    &     $S_1^{in}<\lambda_1^2$      \\
$E_{01}^{11}$    &     $g_2'\left(X_2^{2*}\right)>f_3'\left(X_2^{2*}\right)$      \\
$E_{11}^{11}$    &     $g_2'\left(X_2^{2*}\right)>f_3'\left(X_2^{2*}\right)$      \\ \hline
\end{tabular}
\end{center}
\end{table}
\begin{itemize}[leftmargin=*]
\item For $E_{00}^{00}=(0,0)$, the two eigenvalues of $J_3$ are
$$
a_{33}=\mu_1\left(S_1^{in}\right)-D_2 \quad\mbox{and}\quad a_{44}=\mu_2\left(S_2^{in}\right)-D_2,
$$
which are negative if and only if the two stability conditions in \cref{TableStabModelAM2S2S} hold.
\item For $E_{00}^{0i}=\left(0,~X_2^{2i*}\right)$, $i=1,2$, where $X_2^{2i*}=\left(S_2^{in}-\lambda_2^{2i}\right)/k_3$, the eigenvalues of $J_3$ are written as follows:
$$
a_{33}=\mu_1\left(S_1^{in}\right)-D_2\quad \mbox{and}\quad a_{44}=-k_3\mu_2'\left(\lambda_2^{2i}\right)X_2^{2i*}.
$$
From \cref{Hypoth2}, $E_{00}^{02}$ is unstable when it exists since $a_{44}$ is positive.
However, from \cref{Hypoth1}, $E_{00}^{01}$ is LES if and only if $S_1^{in}<\lambda_1^2$, since $a_{44}$ is negative.
\item For $E_{00}^{10}=\left(X_1^{2*}, 0\right)$, where $X_1^{2*}=\left(S_1^{in}-\lambda_1^2\right)/k_1$, the eigenvalues of $J_3$ are
$$
a_{33}=-k_1\mu_1'\left(\lambda_1^2\right)X_1^{2*} \quad\mbox{and}\quad a_{44}=\mu_2\left(S_2^{in}+\frac{k_2}{k_1}\left(S_1^{in}-\lambda_1^2\right)\right)-D_2.
$$
From \cref{Hypoth1}, $a_{33}$ is negative. From \cref{Hypoth2}, it follows that $E_{00}^{10}$ is LES if and only if the stability condition in \cref{TableStabModelAM2S2S} holds.
\item For $E_{00}^{1i}=\left(X_1^{2*}, X_2^{2i*}\right)$, where $X_1^{2*}=\frac{S_1^{in}-\lambda_1^2}{k_1}$ and $X_2^{2i*}=\frac{1}{k_3} \left(S_2^{in}+\frac{k_2}{k_1}\left(S_1^{in}-\lambda_1^2\right)-\lambda_2^{2i}\right)$, the eigenvalues of $J_3$ are
$$
a_{33}=-k_1\mu_1'\left(\lambda_1^2\right)X_1^{2*} \quad\mbox{and}\quad a_{44}=-k_3\mu_2'\left(\lambda_2^{2i}\right)X_2^{2i*}, \quad i=1,2.
$$
From \cref{Hypoth1}, $a_{33}$ is negative. From \cref{Hypoth2} and the expressions of ${X_2^{2j}}^*$, the steady state $\mathcal{E}_{00}^{12}$ is unstable when it exists since $a_{44}$ is positive. However, $E_{00}^{11}$ is LES since $a_{44}$ is negative.
\item For $E_{10}^{10}=\left(X_1^{2*}, 0\right)$, where $X_1^{2*}$ is the unique solution of equation $f_1(X_1^2)=g_1(X_1^2)$, the eigenvalues of $J_3$ are
$$
a_{33}=-D_2\frac{X_1^{1*}}{X_1^{2*}}-k_1\mu_1'\left(S_1^{in}-k_1X_1^{2*}\right)X_1^{2*} \quad\mbox{and}\quad a_{44}=\mu_2\left(S_2^{in}+k_2X_1^{2*}\right)-D_2.
$$
Using the definitions of $f_1$ and $g_1$ in \cref{TabFunc2}, a straightforward calculation shows that
$$
a_{33}=-\left[g_1'\left(X_1^{2*}\right)-f_1'\left(X_1^{2*}\right)\right]X_1^{2*}.
$$
Recall, from the proof of \cref{lemH1}, that the function $x\mapsto g_1(x)-f_1(x)$ is increasing. Hence, $a_{33}$ is negative. Using \cref{Hypoth2}, $E_{10}^{10}$ is LES if and only if $a_{44}$ is negative, that is, the stability condition in \cref{TableStabModelAM2S2S} holds.
\item For $E_{10}^{1i}=\left(X_1^{2*},X_2^{2i*}\right)$, $i=1,2$, where $X_1^{2*}$ is the unique solution of equation $f_1(x)=g_1(x)$ and $X_2^{2i*}=\frac{1}{k_3}\left(S_2^{in}+k_2X_1^{2*}-\lambda_2^{2i}\right)$, the eigenvalues of $J_3$ are
$$
a_{33}=-\left[g_1'\left(X_1^{2*}\right)-f_1'\left(X_1^{2*}\right)\right]X_1^{2*} \quad\mbox{and}\quad a_{44}=-k_3\mu_2'\left(\lambda_2^{2i}\right)X_2^{2i*}.
$$
Similarly to the steady state $E_{10}^{10}$, the term $a_{33}$ is negative. Thus, $E_{10}^{12}$ is unstable when it exists since $a_{44}$ is positive, while $E_{10}^{11}$ is LES when it exists since $a_{44}$ is negative.
\item For $E_{01}^{01}=\left(0,X_2^{2*}\right)$, where $X_2^{2*}$ is the unique solution of equation $f_2\left(X_2^2\right)=g_2\left(X_2^2\right)$ (see \cref{propmultip}), the eigenvalues of $J_3$ are given by
$$
a_{33}=\mu_1\left(S_1^{in}\right) - D_2 \quad\mbox{and}\quad a_{44}=-D_2\frac{X_2^{1*}}{X_2^{2*}}-k_3\mu_2'\left(S_2^{in}-k_3X_2^{2*}\right)X_2^{2*},
$$
where $X_2^{1*}=\left(S_2^{in}-\lambda_2^{11}\right)/k_3$. From \cref{Hypoth1}, $a_{33}$ is negative if and only if $S_1^{in}<\lambda_1^2$. Using the definitions of $f_2$ and $g_2$ in \cref{TabFunc2}, we can easily calculate
$$
a_{44}=-\left[g_2'\left(X_2^{2*}\right)-f_2'\left(X_2^{2*}\right)\right]X_2^{2*}.
$$
Recall that $X_2^{1*}>x_1^m:=\left(S_2^{in}-S_2^m\right)/k_3$ so that the function $x\mapsto g_2(x)-f_2(x)$ is increasing (see the proof of \cref{lemH2}). Hence, $a_{44}$ is negative. Therefore, the stability condition of $E_{01}^{01}$ in \cref{TableStabModelAM2S2S} hold.
\item For $E_{01}^{11}=\left(X_1^{2*},X_2^{2*}\right)$, where $X_1^{2*} = \frac{1}{k_1}\left(S_1^{in}-\lambda_1^2\right)$ and $X_2^{2*}$ is a solution of $f_3(x)=g_2(x)$, the eigenvalues of $J_3$ are given by
$$
a_{33}=-k_1\mu_1'\left(\lambda_1^2\right)X_1^{2*} \quad\mbox{and}\quad a_{44}=-\left[g_2'\left(X_2^{2*}\right)-f_3'\left(X_2^{2*}\right)\right]X_2^{2*},
$$
From \cref{Hypoth1}, $a_{33}$ is negative. Consequently, $E_{10}^{11}$ is LES if and only if $a_{44}$ is negative, that is, the stability condition in \cref{TableStabModelAM2S2S} holds.
\item For $E_{11}^{11}=\left(X_1^{2*},X_2^{2*}\right)$, where $X_1^{2*}$ is a solution of equation $f_1(x)=g_1(x)$ and $X_2^{2*}$ is a solution of equation $f_3(x)=g_2(x)$,
 the eigenvalues of $J_3$ are given by
$$
a_{33}=-\left[g_1'\left(X_1^{2*}\right)-f_1'\left(X_1^{2*}\right)\right]X_1^{2*} \quad\mbox{and}\quad a_{44}=-\left[g_2'\left(X_2^{2*}\right)-f_3'\left(X_2^{2*}\right)\right]X_2^{2*},
$$
Similarly to $E_{10}^{10}$, $a_{33}$ is negative. As for $E_{01}^{11}$, $E_{11}^{11}$ is LES if and only if $a_{44}$ is negative, that is, the stability condition in \cref{TableStabModelAM2S2S} holds.
\end{itemize}
\end{proof}
\begin{proof}[\bf{Proof of \cref{propPositionCurv}}]
The function $D \mapsto \lambda_1^1(D)$ is defined, positive and increasing on the interval $[0,rm_1)$. It tends to infinity as $D$ tends to $rm_1$. From \cref{TabSurf}, it follows that
$\gamma_0$ and $\gamma_6$ intersect if $rm_1>D_1^*$, that is, $m_1>\mu_2\left(S_2^{in}\right)$. In this case, we have $D = D_1^* = r \mu_2\left(S_2^{in}\right)$ and $S_1^{in}=\lambda_1^1\left(D_1^*,r\right)$, that is, $\gamma_0 \cap \gamma_6 = P_1$.
Since $D_1^*<\min(D_1^m,rm_1)$, then the functions $F_{11}$ and $F_{12}$ are defined in $D_1^*$.
At this values, we have $\mu_1\left(S_1^{in}\right)=\mu_2\left(S_2^{in}\right)=D_1$. If $S_2^{in} < S_2^m$, we obtain $S_2^{in}=\lambda_2^{11}$; otherwise $S_2^{in}=\lambda_2^{12}$.
From the definition of the functions $F_{11}$ and $F_{12}$ in \cref{TabFunc1}, it follows that $S_1^{in}=\lambda_1^1$ in these two cases. Consequently, $\gamma_0 \cap \gamma_6 \cap \gamma_4 = P_1$ if $S_2^{in} < S_2^m$ and $\gamma_0 \cap \gamma_6 \cap \gamma_5 = P_1$, otherwise.
Similar arguments apply to the second part of the proposition.
\end{proof}
\begin{proof}[\bf{Proof of \cref{PropF21F11}}]
Under \cref{LemmaPositionLambdai}, we have for all $r\in(0,1/2)$
$$
\lambda_1^2(D,r)<\lambda_1^1(D,r)\quad\mbox{and}\quad \lambda_2^{21}(D,r)<\lambda_2^{11}(D,r).
$$
From the definition of $F_{11}$ and $F_{21}$ in \cref{TabFunc1}, $F_{21}\left(D,r,S_2^{in}\right)<F_{11}\left(D,r,S_2^{in}\right)$. Similarly, $F_{21}\left(D,r,S_2^{in}\right)>F_{11}\left(D,r,S_2^{in}\right)$, for all $r\in(1/2,1)$.
\end{proof}

Let $\phi_i$, $i=1,2$ be the function defined in \cref{TabFunc1}. To describe the existence condition $\phi_i>0$ of the steady state $\mathcal{E}_{10}^{1i}$ in \cref{TableCondExis} and the stability condition $\phi_1<0$ or $\phi_2>0$ of the steady state $\mathcal{E}_{10}^{10}$ in \cref{propstab}, we need the following lemma showing the positivity of $\phi_i$ according to the values of $S_2^{in}$ and $D$.
\begin{lemma}                                           \label{Lemr3}
Let $D\in\left[0,\min\left(rm_1,(1-r)\mu_2\left(S_2^m\right)\right)\right]$ and $\phi_j(D)$ be the function defined in \cref{TabFunc1}.
For all $S_2^{in}>S_2^m$, we have $\phi_1(D)>0$. In addition, when $D>D_2^*=(1-r)\mu_2\left(S_2^{in}\right)$, then $\phi_2(D)>0$.
\end{lemma}
\begin{proof}[\bf{Proof of \cref{Lemr3}}]
For all $S_2^{in}>S_2^m$, we have $S_2^{in}>\lambda_2^{21}(D,r)$, so that $\phi_1(D)>0$ holds. In addition, from the definition of $D_2^*$ in  \cref{TabSurf}, we have $\lambda_2^{22}(D,r)<\lambda_2^{22}(D_2^*,r)=S_2^{in}$ for all $D>D_2^*$, so that $\phi_2(D)>0$ holds.
\end{proof}
In the next remark, we show that the existence condition of $\mathcal{E}_{10}^{1i}$ and the stability condition of $\mathcal{E}_{10}^{10}$ are not always satisfied where the function $\phi_i(D)$ can change their sign as we will see in \cref{FigCondE101i}.
\begin{remark}
For the set of parameters in \cref{TabParamVal}, the condition $\phi_2(D)>0$ is not always verified when $D<D_2^*$ (see \cref{FigCondE101i}(a)).
For the set of parameters in \cref{FigDO3} where $S_2^{in}<S_2^m$, condition $\phi_1(D)>0$ holds (see \cref{FigCondE101i}(b)) and the condition $\phi_2(D)>0$ is not always verified (see \cref{FigCondE101i}(c)).
\end{remark}
\begin{figure}[!ht]
\setlength{\unitlength}{1.0cm}
\begin{center}
\begin{picture}(7.8,4.7)(0,0)
\put(-4,0){{\includegraphics[width=5.3cm,height=4.5cm]{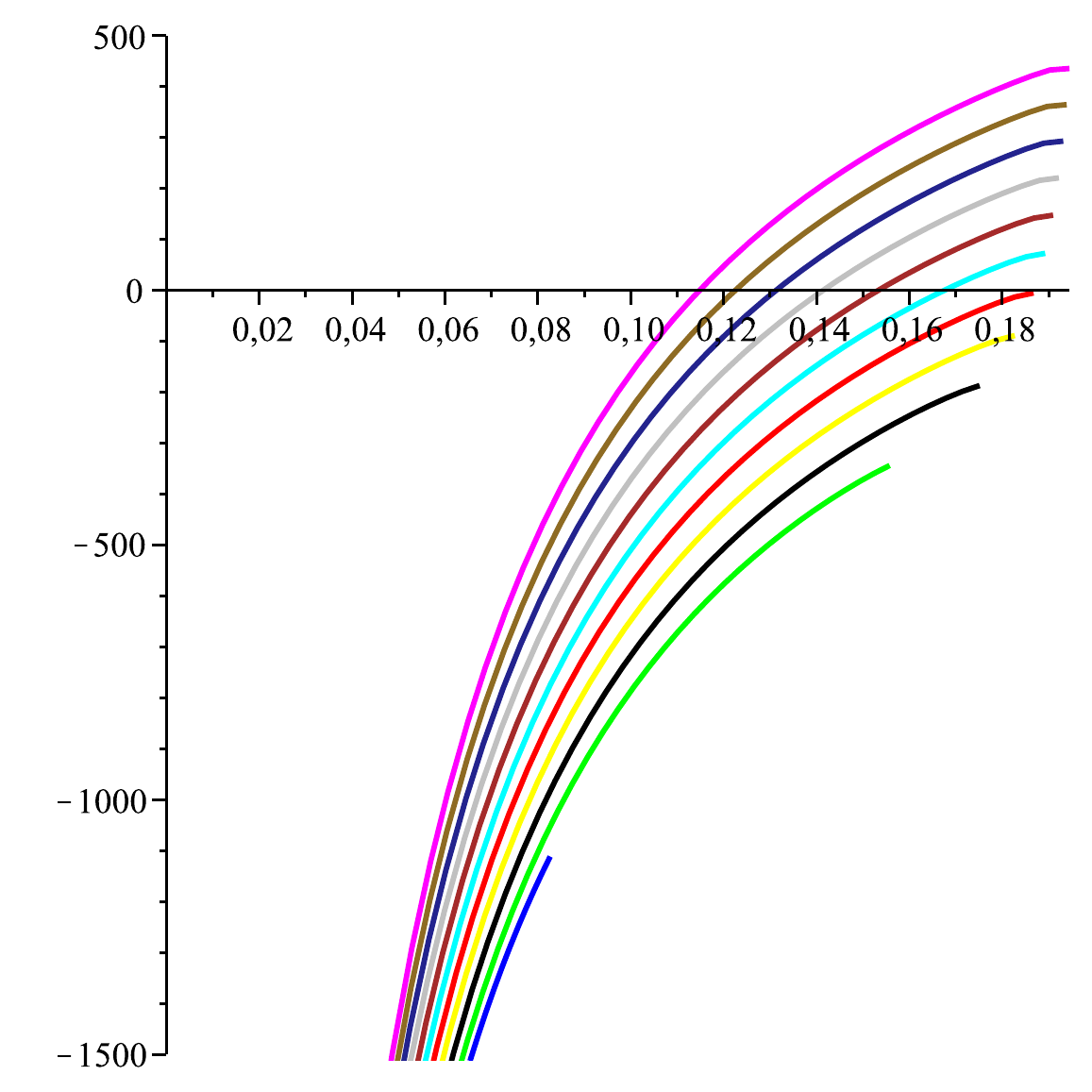}}}
\put(1.1,0){{\includegraphics[width=5.3cm,height=4.5cm]{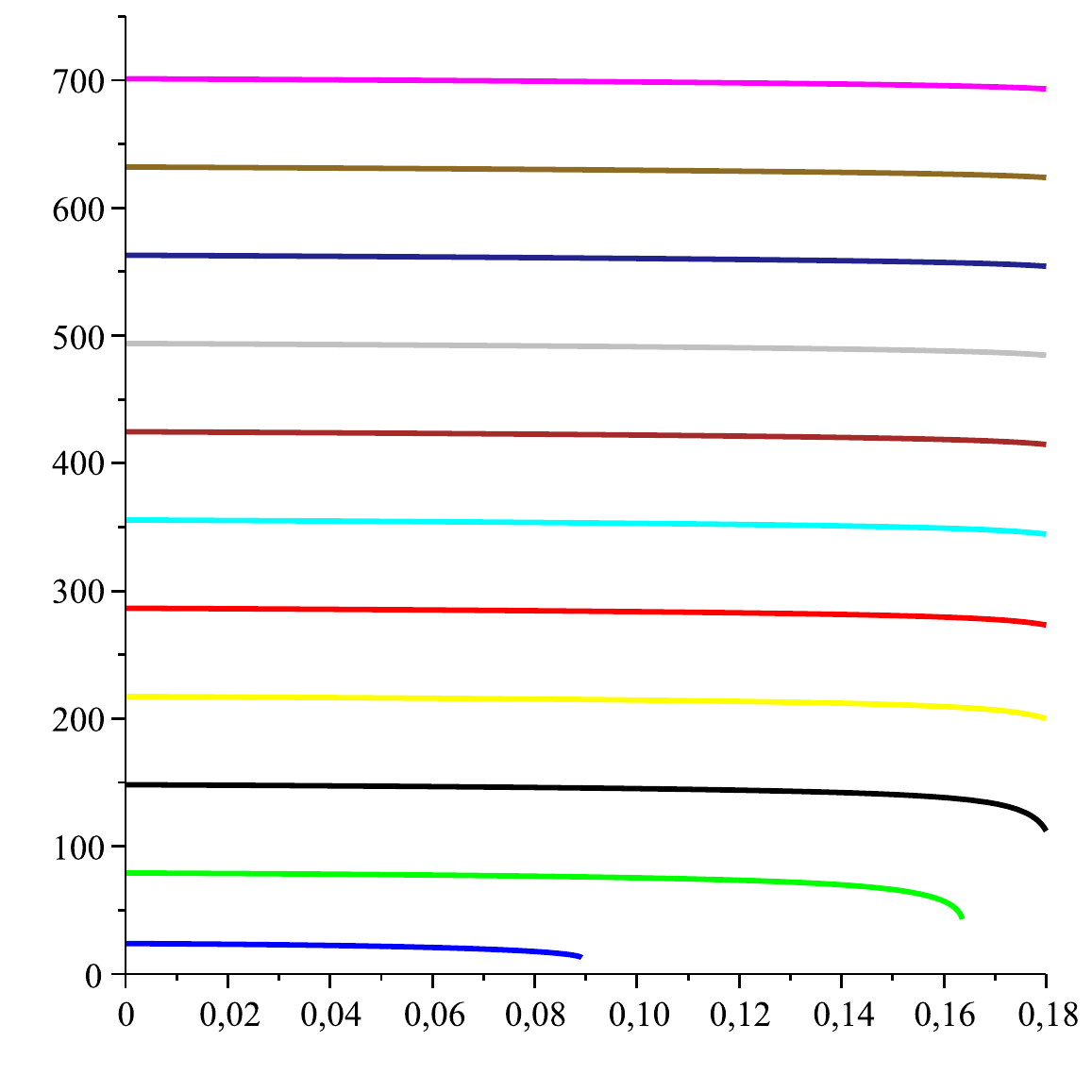}}}
\put(6.5,0){{\includegraphics[width=5.3cm,height=4.5cm]{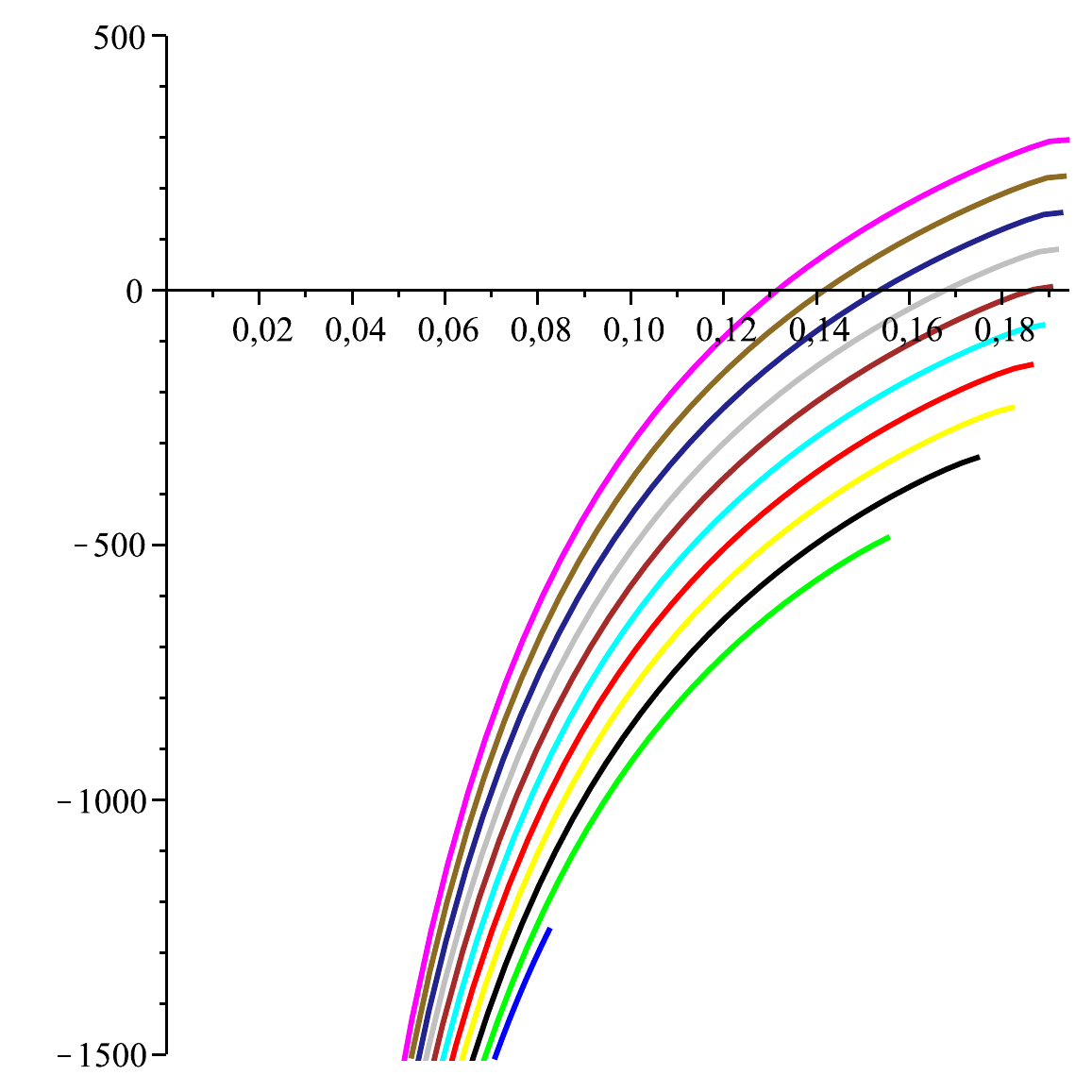}}}
\put(-2,4.5){\sc (a)}
\put(1.15,3.25){\sc $D$}
\put(-3.4,4.5){\sc $\phi_2$}
\put(3.7,4.5){\sc (b)}	
\put(6.3,0.45){\sc $D$}
\put(1.6,4.6){\sc $\phi_1$}
\put(9,4.5){{\sc (c)}}	
\put(11.7,3.25){\sc $D$}
\put(7.2,4.5){\sc $\phi_2$}
\end{picture}
\end{center}
\vspace{-0.4cm}
\caption{Curves of the functions $\phi_1(D)$ and $\phi_2(D)$, for several fixed values of $S_1^{in}=5,25,50,75,100,125,150,175,200,225,250$ and $S_2^m\approx48.74$: (a) $ \phi_2(D)$ when $S_2^{in}=150>S_2^m$ (b) $ \phi_1(D)$ when $S_2^{in}=10<S_2^m$; (c) $ \phi_2(D)$ when $S_2^{in}=10<S_2^m$.} \label{FigCondE101i}
\end{figure}
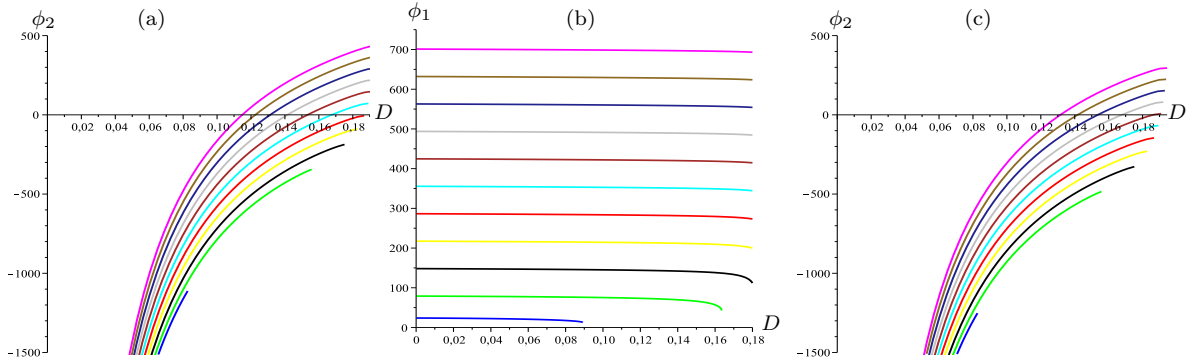
\begin{proof}[\bf{Proof of \cref{propstabstabDO1}}]
We determine the operating diagram of \cref{FigDO1} in the case $r\in(0,1/2)$, $S_2^{in}>S_2^m$, and $m_1>\mu_2\left(S_2^m\right)$.
However, the cases in items (2) and (3) can be treated similarly and are left to the
reader using \cref{CurvFigDO1,FigDO34Cruv}, respectively. In addition, the only difference between items (4) and (1) is the stability of the steady states $\mathcal{E}_{00}^{00}$, $\mathcal{E}_{00}^{01}$, and $\mathcal{E}_{01}^{01}$, where the minimum of the break-even concentrations $\lambda_1^i$ and $\lambda_2^{ij}$ is given by
$$
\min\left(\lambda_1^1,\lambda_1^2\right)=\lambda_1^1,\quad\min\left(\lambda_2^{11},\lambda_2^{21}\right)=\lambda_2^{11}\quad\mbox{and}\quad\max\left(\lambda_2^{12},\lambda_2^{22}\right)=\lambda_2^{12},
$$
using \cref{LemmaPositionLambdai} for $r>1/2$. Inversely, in item (1) for $r<1/2$, we have
\begin{equation}                             \label{lambdaMmRinfDemi}
\min\left(\lambda_1^1,\lambda_1^2\right)=\lambda_1^2,\quad\min\left(\lambda_2^{11},\lambda_2^{21}\right)=\lambda_2^{21} \quad\mbox{and}\quad  \max\left(\lambda_2^{12},\lambda_2^{22}\right)=\lambda_2^{22}.
\end{equation}
From \cref{propstab}, the six steady states $\mathcal{E}_{00}^{02}$, $\mathcal{E}_{00}^{12}$, $\mathcal{E}_{10}^{12}$, $\mathcal{E}_{02}^{01}$, $\mathcal{E}_{02}^{11}$ and $\mathcal{E}_{12}^{11}$ are unstable, when they exist.
Next, we establish the existence and stability of each steady state in the various regions of the operating diagram in the case of item (1).
\begin{figure}[!ht]
\setlength{\unitlength}{1.0cm}
\begin{center}
\begin{picture}(7,6.3)(0,0)
\put(-4.5,0){\rotatebox{0}{\includegraphics[width=7.5cm,height=5.5cm]{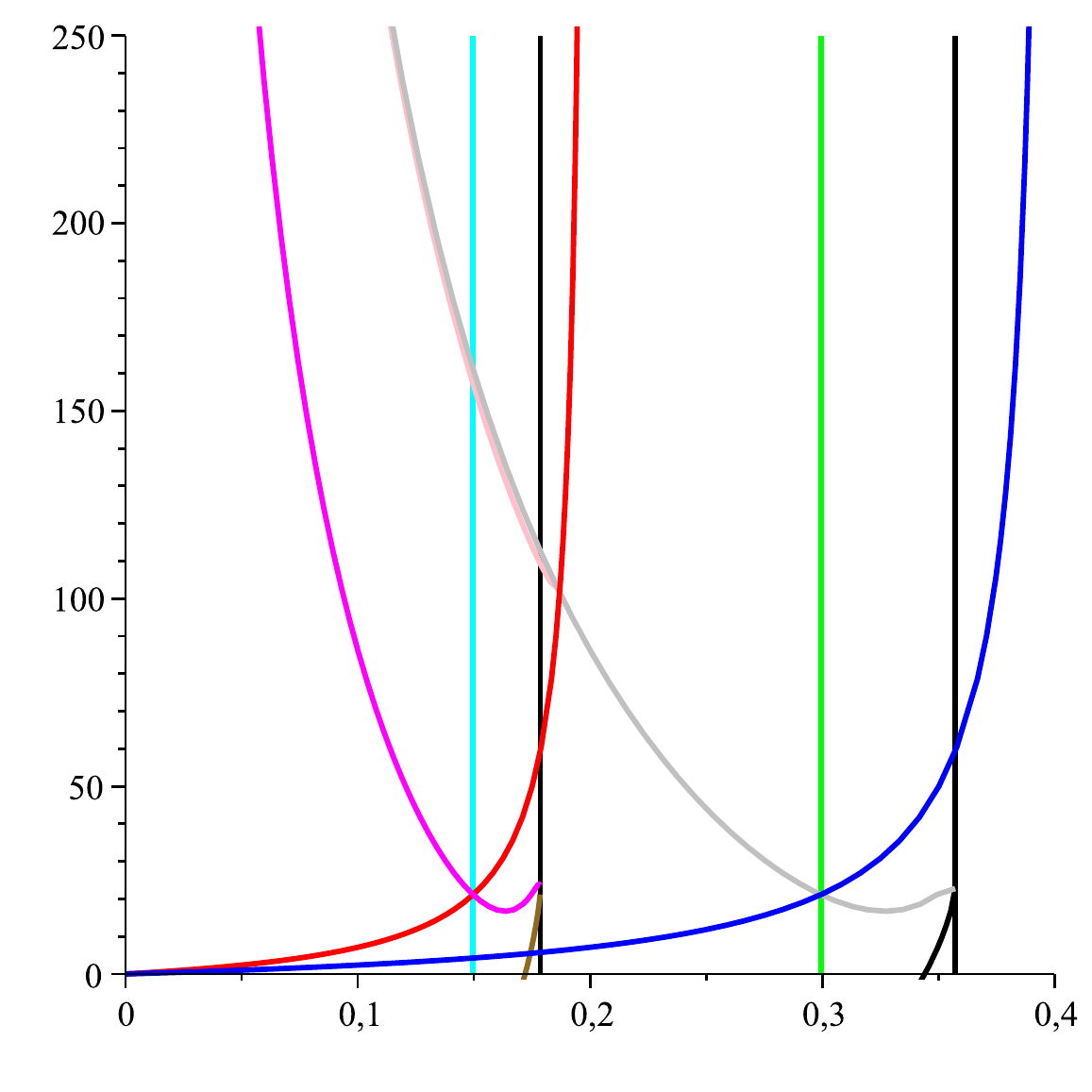}}}
\put(3.5,0){\rotatebox{0}{\includegraphics[width=7.5cm,height=5.5cm]{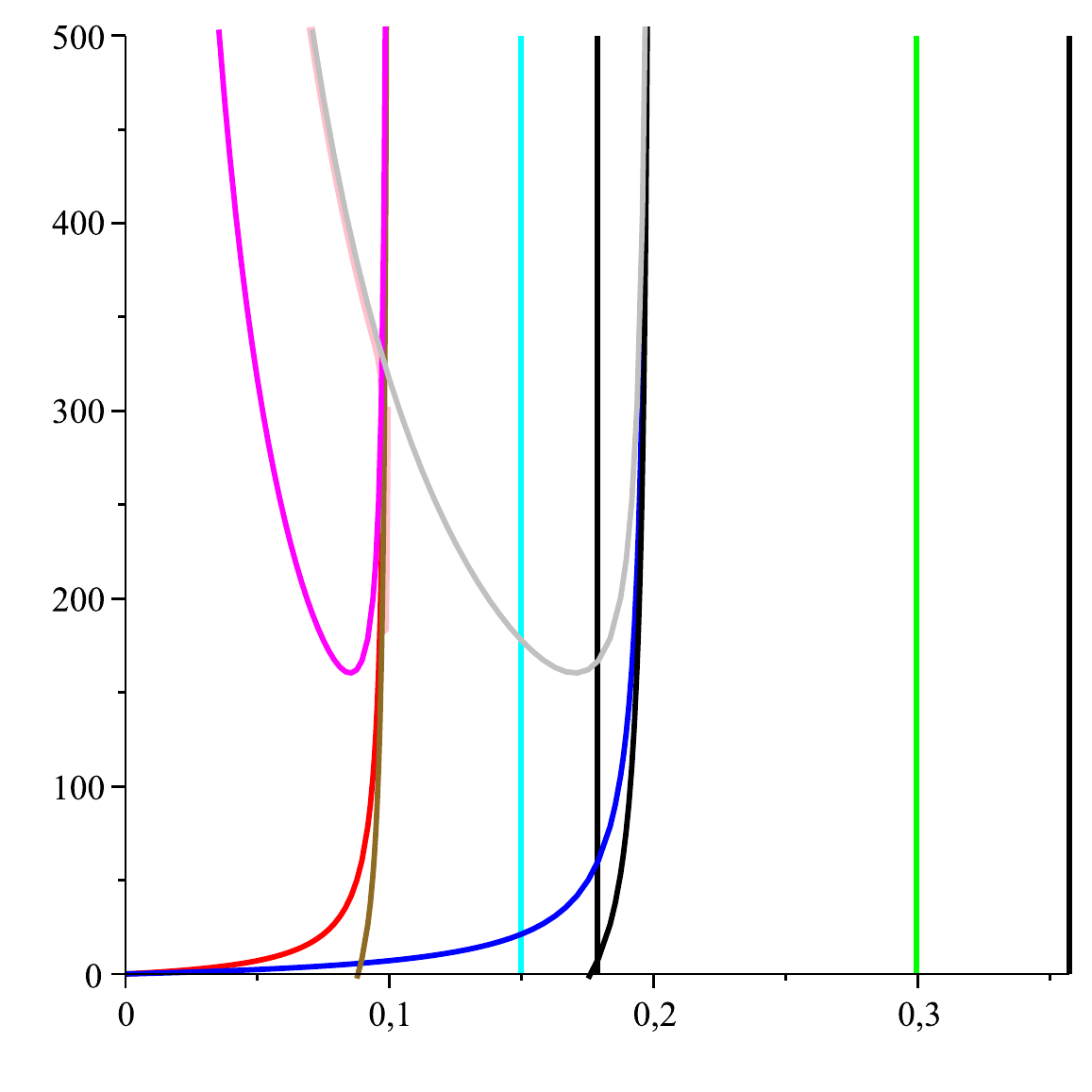}}}
\put(-1,6){\sc $(a)$}
\put(-4,5.5){\sc $S_1^{in}$}
\put(2.8,0.5){\sc $D$}
\put(-0.8,5.5){\sc { $\gamma_0$}} 
\put(-1.1,5.5){\sc { $\gamma_8$}} 
\put(-1.55,5.5){\sc { $\gamma_6$}} \put(-2.15,5.5){\sc {\color{pink} $\gamma_{15}$}}	
\put(-0.4,2.1){\sc { $\gamma_3$}} 
\put(0.8,5.5){\sc { $\gamma_7$}} 
\put(-3,5.5){\sc { $\gamma_5$}}
\put(2.4,5.5){\sc { $\gamma_1$}}
\put(1.7,5.5){\sc { $\gamma_9$}}
\put(1.8,0.35){\sc $\nwarrow$}
\put(1.9,0.1){\sc { $\gamma_2 $}}
\put(-0.9,0.35){\sc $\nwarrow$}
\put(-0.8,0.1){\sc { $ \gamma_4 $}}		
\put(6.8,6){{\sc $(b)$}}
\put(4.2,5.5){\sc $S_1^{in}$}
\put(10.9,0.5){\sc $D$}
\put(5.85,5.5){\sc { $\gamma_0$}}
\put(7.3,5.5){\sc { $\gamma_8$}}
\put(6.75,5.5){\sc { $\gamma_6$}}
\put(5.4,5.5){\sc {\color{pink} $\gamma_{15}$}}	
\put(6.3,3){\sc { $\gamma_3$}}
\put(9.5,5.5){\sc { $\gamma_7$}}
\put(4.7,5.5){\sc { $\gamma_5$}}
\put(7.7,5.5){\sc { $\gamma_1$}}
\put(10.55,5.5){\sc { $\gamma_9$}}
\put(7.5,0.35){\sc $\nwarrow$}
\put(7.7,0.1){\sc { $\gamma_2 $}}
\put(6,0.35){\sc $\nwarrow$}
\put(6.1,0.1){\sc { $ \gamma_4 $}}		
\end{picture}
\vspace{-0.4cm}
\caption{The curves $\gamma_i,i=1-9$ and $\gamma_{15}$, when $r=1/3\in(0,1/2)$, $S_2^{in}=150>S_2^m\approx48.740$: (a) $m_1=0.6>\mu_2\left(S_2^m\right)\approx0.535$; (b) $m_1=0.3<\mu_2\left(S_2^m\right)\approx0.535$.}\label{CurvFigDO1}
\end{center}
\end{figure}
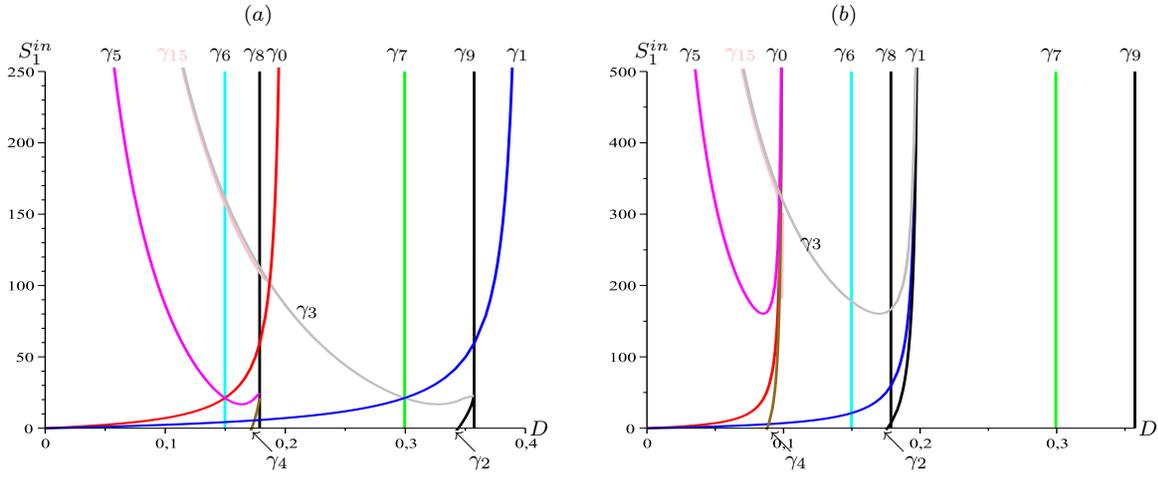
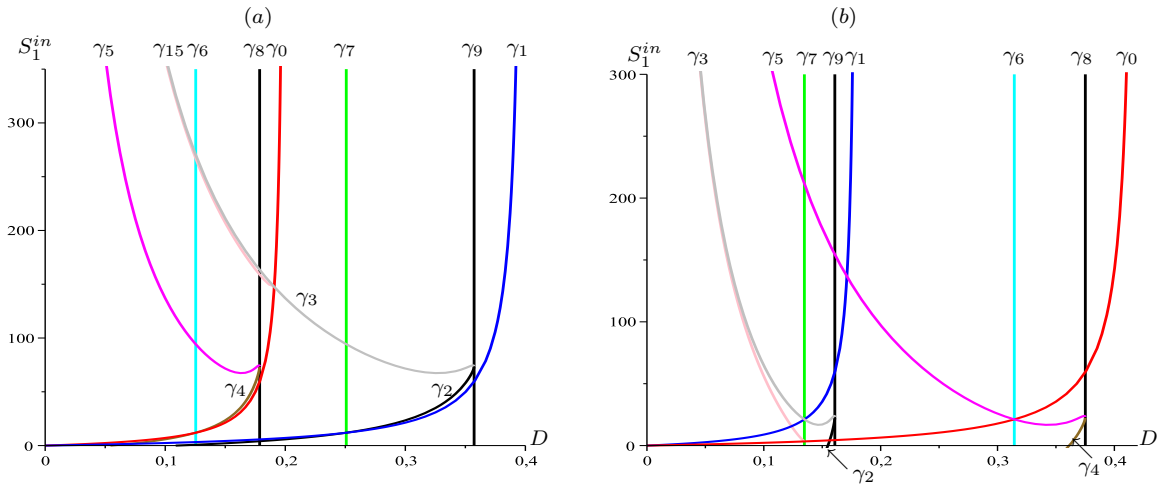
\begin{figure}[!ht]
\setlength{\unitlength}{1.0cm}
\begin{center}
\begin{picture}(7,6.3)(0,0)
\put(3.5,0){\rotatebox{0}{\includegraphics[width=7.5cm,height=5.5cm]{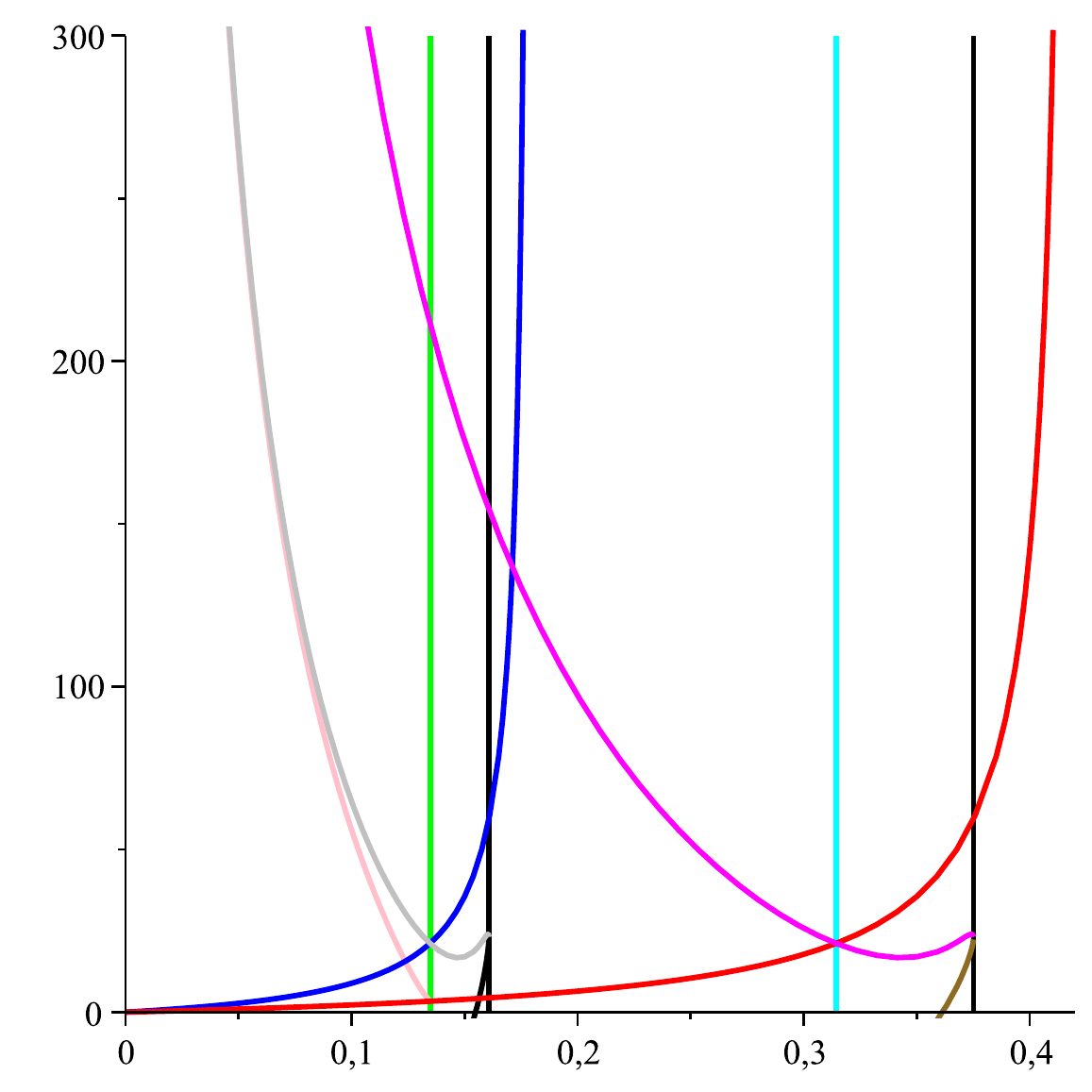}}}
\put(-4.5,0){\rotatebox{0}{\includegraphics[width=7.5cm,height=5.5cm]{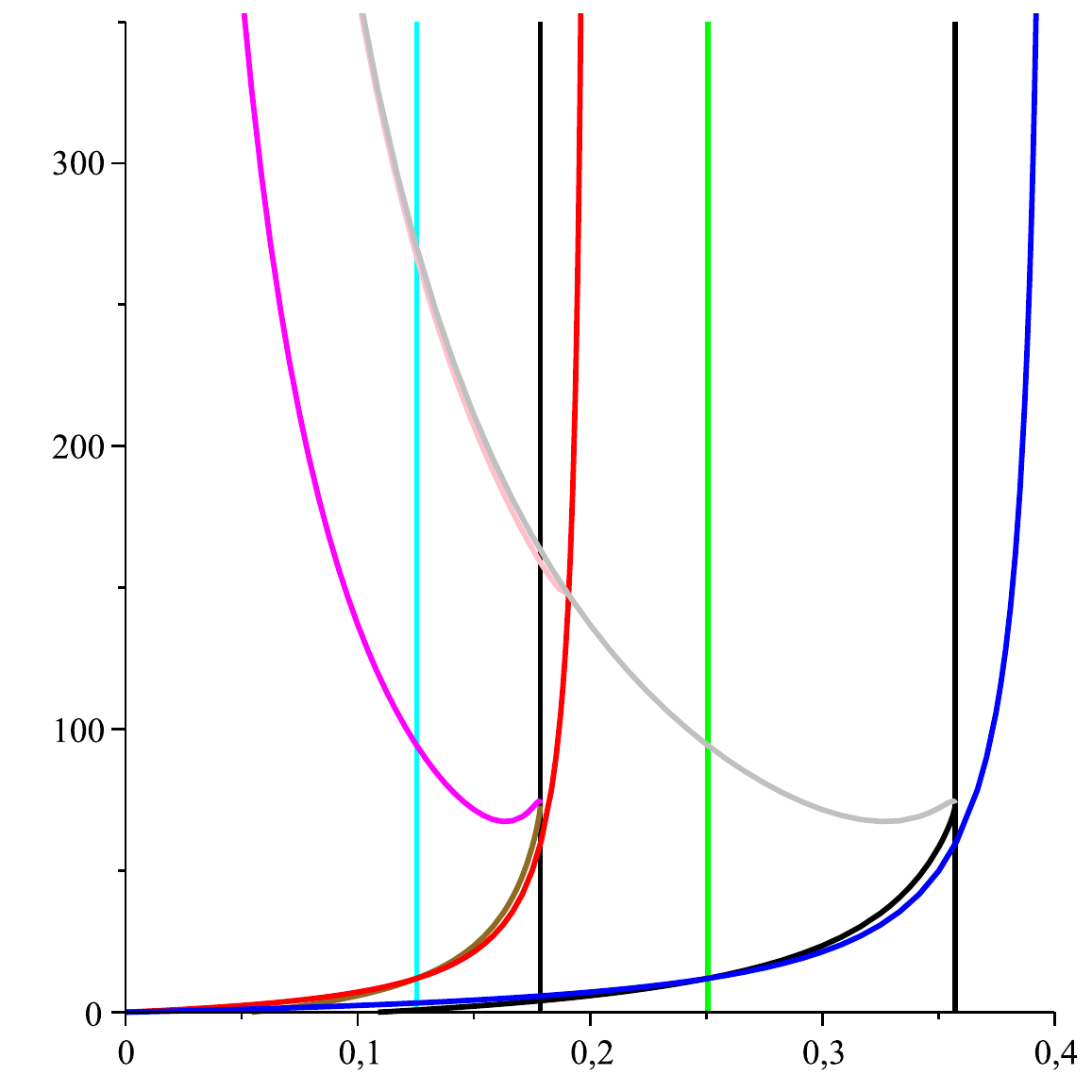}}}
\put(-1,6){\sc $(a)$}
\put(-4,5.6){\sc $S_1^{in}$}
\put(2.8,0.4){\sc $D$}
\put(-0.8,5.6){\sc { $\gamma_0$}}
\put(-1.1,5.6){\sc { $\gamma_8$}}
\put(-1.85,5.6){\sc { $\gamma_6$}} \put(-2.3,5.6){\sc { $\gamma_{15}$}}	
\put(-0.4,2.3){\sc { $\gamma_3$}}
\put(0.1,5.6){\sc { $\gamma_7$}}
\put(-3.1,5.6){\sc { $\gamma_5$}}
\put(2.4,5.6){\sc { $\gamma_1$}}
\put(1.8,5.6){\sc { $\gamma_9$}}
\put(1.4,1.1){\sc { $\gamma_2 $}}
\put(-1.35,1.1){\sc { $ \gamma_4 $}}	
\put(6.8,6){{\sc $(b)$}}
\put(4.1,5.5){\sc $S_1^{in}$}
\put(10.9,0.4){\sc $D$}
\put(6.9,5.5){\sc { $\gamma_1$}}
\put(6.6,5.5){\sc { $\gamma_9$}}
\put(6.25,5.5){\sc { $\gamma_7$}} 	
\put(5.8,5.5){\sc { $\gamma_5$}}
\put(9,5.5){\sc { $\gamma_6$}}
\put(4.8,5.5){\sc { $\gamma_3$}}
\put(10.5,5.5){\sc { $\gamma_0$}}
\put(9.9,5.5){\sc { $\gamma_8$}}
\put(10,0.3){\sc $\nwarrow$}
\put(10,0.05){\sc { $\gamma_4 $}}
\put(6.75,0.15){\sc $\nwarrow$}
\put(7,-0.05){\sc { $ \gamma_2 $}}	
\end{picture}
\vspace{-0.1cm}
\caption{The curves $\gamma_i,i=1-9$ and $\gamma_{15}$: (a) case when $r= 1/3\in(0,1/2)$, $S_2^{in}=10<S_2^m\approx48.740$ and $m_1=0.6>\mu_2\left(S_2^m\right)\approx0.535$; (b) case when $r= 0.7\in(0,1/2)$, $S_2^{in}=150>S_2^m\approx48.740$ and $m_1=0.6>\mu_2\left(S_2^m\right)\approx0.535$. }\label{FigDO34Cruv}
\end{center}
\end{figure}
\begin{itemize}[leftmargin=*]
\item $\mathcal{E}_{00}^{00}$ always exists. From \cref{TableCondStab} and \cref{lambdaMmRinfDemi}, it is LES if and only if,
$$
S_1^{in}<\lambda_1^2\quad\mbox{and}\quad\left(S_2^{in}<\lambda_2^{21}\quad\mbox{or}\quad S_2^{in}>\lambda_2^{22}\right).
$$
From \cref{TabSurf}, the first stability condition of $\mathcal{E}_{00}^{00}$ holds, for all $\left(D,S_1^{in}\right)$ in the region bounded by the curve $\gamma_1$ and located below this curve.
The second stability condition is equivalent to $D_2^*:=(1-r)\mu_2\left(S_2^{in}\right)<D$, that is, this condition holds in the regions at the right of the vertical line $\gamma_8$.
Using the definition of regions in \cref{TabRegionCondDO123}, $\mathcal{E}_{00}^{00}$ is LES in $\mathcal{J}_0$ and $\mathcal{J}_3$ (see \cref{FigDO1}). Hence, we obtain the first column in \cref{TabRegDOJi} corresponding to $\mathcal{E}_{00}^{00}$.
\item For $\mathcal{E}_{00}^{01}$, the existence condition is given in \cref{TableCondExis} by $S_2^{in}>\lambda_2^{21}$. Since $S_2^{in}>S_2^m$, then $S_2^{in}>\lambda_2^{21}$ for all $D<D_2^m=(1-r)\mu_2\left(S_2^m\right)$, that is, for all $\left(D,S_1^{in}\right)$ in the regions bounded by the vertical line $\gamma_9$ and located at the left of this line. Using \cref{TabRegionCondDO123}, $\mathcal{E}_{00}^{01}$ exists in $\mathcal{J}_i$, $i=2-20$ (see \cref{FigDO1}). From \cref{TableCondStab} and \cref{lambdaMmRinfDemi}, $\mathcal{E}_{00}^{01}$ is LES if and only if,
$$
S_1^{in}<\lambda_1^2 \quad\mbox{and}\quad \left( S_2^{in}<\lambda_2^{11} \quad\mbox{or}\quad S_2^{in}>\lambda_2^{12} \right).
$$
Thus, the first stability condition holds in the region below the curve $\gamma_1$. The second stability condition is equivalent to $D_1^*:=r\mu_2\left(S_2^{in}\right)<D$, that is, this condition holds in the regions at the right of the vertical line $\gamma_6$. Using \cref{TabRegionCondDO123}, $\mathcal{E}_{00}^{01}$ is LES in $\mathcal{J}_3$, $\mathcal{J}_4$ and $\mathcal{J}_{14}$ (see \cref{FigDO1}). Consequently, we obtain the second column in \cref{TabRegDOJi} corresponding to $\mathcal{E}_{00}^{01}$.
\item $\mathcal{E}_{00}^{02}$ exists if and only if $S_2^{in}>\lambda_2^{22}$, that is, for all $D_2^*<D<D_2^m$. Thus, the existence condition of $\mathcal{E}_{00}^{02}$ holds in the regions at the left of the vertical line $\gamma_9$ and the right of the vertical line $\gamma_8$. Using \cref{TabRegionCondDO123}, $\mathcal{E}_{00}^{02}$ exists in $\mathcal{J}_2$ and $\mathcal{J}_3$.
\item $\mathcal{E}_{00}^{10}$ exists in the regions above the curve $\gamma_1$ (see \cref{FigDO1}) as the existence condition is $S_1^{in}>\lambda_1^2$. Using the definition of regions in \cref{TabRegionCondDO123}, $\mathcal{E}_{00}^{10}$ exists in $\mathcal{J}_i,i=1,2,5-13,16-20$.
From \cref{TableCondStab}, the inequalities \cref{lambdaMmRinfDemi} and the functions $F_{21}$ and $F_{21}$ defined in \cref{TabFunc1}, the stability conditions
of $\mathcal{E}_{00}^{10}$ become
$$
S_1^{in}<\lambda_1^1 \quad\mbox{and}\quad \left(S_2^{in}<\lambda_2^{11} \quad\mbox{or}\quad  S_2^{in}>\lambda_2^{12}\right) \quad\mbox{and}\quad \left(S_1^{in} <  F_{21} \quad\mbox{or}\quad S_1^{in} >  F_{22}\right).
$$
Thus, the first stability condition of $\mathcal{E}_{00}^{10}$ holds in the regions below the curve $\gamma_0$.
Similarly to $\mathcal{E}_{00}^{01}$, the second stability condition holds in the regions at the right of the vertical line $\gamma_6$.
The third stability condition can be written as
$$
\mu_2\left(S_2^{in}+\frac{k_2}{k_1}\left(S_1^{in}-\lambda_1^2\right)\right)<\frac{D}{1-r}
$$
or equivalently
$$
D > D_2^m:=(1-r)\mu_2\left(S_2^m\right) \quad \mbox{or} \quad S_2^{in}+\frac{k_2}{k_1}\left(S_1^{in}-\lambda_1^2\right) > \lambda_2^{22} \quad (\mbox{or also } S_1^{in} >  F_{22})
$$
since $S_2^{in} > S_2^m$. Therefore, the third stability condition holds in the region at the right of the vertical line $\gamma_9$ or above the curve $\gamma_3$ (see \cref{FigDO1}). From \cref{TabRegionCondDO123}, it follows that $\mathcal{E}_{00}^{10}$ is LES in $\mathcal{J}_i,i=1,2,8$.
\item For $\mathcal{E}_{00}^{1i},i=1,2$, the existence conditions become
$$
S_1^{in}>\max\left(\lambda_1^2,F_{2i}\right),
$$
where $F_{2i}$ is defined for all $D < D_2^m:=(1-r)\mu_2\left(S_2^m\right)$ (see \cref{TabFunc1}), that is, this condition holds for all $\left(D, S_1^{in}\right)$ in the regions at the left of the vertical line $\gamma_9$.
From \cref{propPositionCurv}, we have $F_{21}<\lambda_1^2$ and $\lambda_1^2 < F_{22}$, for all $D<D_2^*$ and inversely $\lambda_1^2 > F_{22}$, for all $D>D_2^*$.
Hence, $\mathcal{E}_{00}^{11}$ exists in the regions above the curve $\gamma_1$ and at the left of the vertical line $\gamma_9$.
However, $\mathcal{E}_{00}^{12}$ exists either in the regions above the curve $\gamma_3$ when $\left(D,S_1^{in}\right)$ is on the left of the vertical line $\gamma_8$
either in the regions above the curve $\gamma_1$ when $\left(D,S_1^{in}\right)$ is on the right of the vertical line $\gamma_8$ and the left of the vertical line $\gamma_9$ (see \cref{FigDO1}).
From \cref{TabRegionCondDO123}, $\mathcal{E}_{00}^{12}$ exists in $\mathcal{J}_i$, $i=2,8-10,20$ while $\mathcal{E}_{00}^{11}$ exists in $\mathcal{J}_i,i=2,5-13,16-20$.
The stability conditions of $\mathcal{E}_{00}^{11}$ are given by
$$
S_1^{in}<\lambda_1^1 \quad \mbox{ and } \quad \left(S_2^{in}<\lambda_2^{11} \quad\mbox{or}\quad S_2^{in}>\lambda_2^{12}\right).
$$
Thus, the first stability condition of $\mathcal{E}_{00}^{10}$ holds in the regions below the curve $\gamma_0$.
Similarly to $\mathcal{E}_{00}^{01}$, the second stability condition of $\mathcal{E}_{00}^{11}$ holds in the regions at the right of the vertical line $\gamma_6$.
From \cref{TabRegionCondDO123}, $\mathcal{E}_{00}^{11}$ is LES in $\mathcal{J}_i$, $i=2,5,8,13$.
\item For $\mathcal{E}_{10}^{10}$, the existence conditions become $S_1^{in}>\lambda_1^1.$
Thus, $\mathcal{E}_{10}^{10}$ exists in the regions above the curve $\gamma_0$ (see \cref{FigDO1}). Using the definition of regions in \cref{TabRegionCondDO123}, $\mathcal{E}_{10}^{10}$ exists in $\mathcal{J}_i,i=6,7,9-12,17-20$. The stability conditions of $\mathcal{E}_{10}^{10}$ become
$$
\left[ D_1 > \mu_2\left(S_2^m\right) \mbox{ or } \left(D_1 < \mu_2\left(S_2^m\right) \mbox{ and }  \left(S_1^{in} <  F_{11} \mbox{ or } S_1^{in} >  F_{12}\right) \right)\right]
\mbox{ and } \left[ \phi_1<0 \mbox{ or } \phi_2>0\right].
$$
From \cref{propPositionCurv}, we have $\lambda_1^1>F_{11}$ so that the condition $S_1^{in} < F_{11}$ does not hold since the existence condition of $\mathcal{E}_{10}^{10}$ is $S_1^{in}>\lambda_1^1$.
Hence, the first stability condition of $\mathcal{E}_{10}^{10}$ holds in the regions at the right of the vertical line $\gamma_8$ or on the left of $\gamma_8$ and above the maximum of the two curves $\gamma_0$ and $\gamma_5$ (see \cref{CurvFigDO1}(a)).
According to \cref{Lemr3}, condition $\phi_1<0$ is not satisfied.
\cref{FigCondE101i}(a) shows that the condition $\phi_2(D)>0$ holds in the right of the root of equation $\phi_2(D)=0$ for several values of $S_1^{in}$. Hence, the second stability condition of $\mathcal{E}_{10}^{10}$ holds in the regions at the right of the curve $\gamma_{15}$.
From \cref{TabRegionCondDO123}, $\mathcal{E}_{10}^{10}$ is LES in $\mathcal{J}_i,i=7,9,10,11,19,20$.
\item For $\mathcal{E}_{10}^{1i},i=1,2$, the existence conditions become
$$
S_1^{in}>\lambda_1^1\quad\mbox{and}\quad \phi_i>0.
$$
From \cref{Lemr3}, condition $\phi_1>0$ is always verified.
Thus, $\mathcal{E}_{10}^{11}$ exists in the regions above the curve $\gamma_0$ while $\mathcal{E}_{10}^{12}$ exists in the regions on the left of the curve $\gamma_0$ and the right of the curve $\gamma_{15}$.
From \cref{TabRegionCondDO123}, $\mathcal{E}_{10}^{11}$ exists in $\mathcal{J}_i,i=6,7,9-12,17-20$ while $\mathcal{E}_{10}^{12}$ exists in $\mathcal{J}_i,i=7,9,11,19,20$. The stability condition
of $\mathcal{E}_{10}^{11}$ becomes
$$
S_1^{in} <  F_{11} \quad\mbox{ or }\quad S_1^{in} >  F_{12}.
$$
Similarly to $\mathcal{E}_{10}^{10}$, the stability conditions of $\mathcal{E}_{10}^{11}$ holds in the regions at the right of the vertical line $\gamma_8$ or on the left of $\gamma_8$ and above the maximum of the two curves $\gamma_0$ and $\gamma_5$.
From \cref{TabRegionCondDO123}, $\mathcal{E}_{10}^{11}$ is LES in $\mathcal{J}_i,i=6,7,9-12,18,19,20$.

\item For $\mathcal{E}_{0i}^{01},i=1,2$, the existence conditions become $S_2^{in}>\lambda_2^{1i}$. Since $S_2^{in}>S_2^m$, then the condition $S_2^{in}>\lambda_2^{21}$ holds for all $D<D_1^m$. Thus, the existence condition of $\mathcal{E}_{01}^{01}$ holds in the regions at the left of the vertical line $\gamma_8$. However, the existence condition of $\mathcal{E}_{02}^{01}$ holds in the regions on the left of the curve $\gamma_8$ and the right of the curve $\gamma_6$.
From \cref{TabRegionCondDO123}, $\mathcal{E}_{01}^{01}$ exists in $\mathcal{J}_i,i=10-20$ while $\mathcal{E}_{02}^{01}$ exists in $\mathcal{J}_i,i=10-14$.
The stability condition of $\mathcal{E}_{01}^{01}$ becomes $S_1^{in}<\lambda_1^2$, that is, this condition holds in the regions below the curve $\gamma_1$. From \cref{TabRegionCondDO123}, it follows that $\mathcal{E}_{01}^{01}$ is LES in $\mathcal{J}_{14}$ and $\mathcal{J}_{15}$.
\item For $\mathcal{E}_{0i}^{11}$, $i=1,2$, the existence conditions become
$$
S_2^{in}>\lambda_2^{1i} \quad\mbox{and}\quad S_1^{in}>\lambda_1^2.
$$
The second existence condition of $\mathcal{E}_{0i}^{11}$ holds in the regions above the curve $\gamma_1$.
Since $S_2^{in}>S_2^m>\lambda_2^{11}$ for all $D<D_1^m$, the first existence condition of $\mathcal{E}_{01}^{11}$ holds in the regions at the left of the vertical line $\gamma_8$.
The first existence condition of $\mathcal{E}_{02}^{11}$ holds in the regions on the left of the curve $\gamma_8$ and the right of the curve $\gamma_6$.
From \cref{TabRegionCondDO123}, it follows that $\mathcal{E}_{01}^{11}$ exists in $\mathcal{J}_i,i=10-13,16-20$ while $\mathcal{E}_{02}^{01}$ exists in $\mathcal{J}_i,i=10-13$. The stability conditions of $\mathcal{E}_{01}^{11}$ is given by
$$
S_1^{in}<\lambda_1^1\quad\mbox{and}\quad g_2'(X_2^{2*})>f_3'(X_2^{2*})
$$
The first stability condition holds in the region below the curve $\gamma_0$.
For the set of parameters in \cref{TabParamVal} corresponding to the operating diagram in \cref{FigDO1}, the equation $\phi_3(x):=f_3(x)-g_2(x)=0$ has unique solution, that is, $g_2'(X_2^{2*})>f_3'(X_2^{2*})$ holds.
Since the operating parameters $D$ and $S_1^{in}$ are variable, the numerical simulations show the uniqueness of the solution of equation $\phi_3(x)=0$ for several values of $D$ and $S_1^{in}$ as shown in \cref{FigCondStabE0111}. From \cref{TabRegionCondDO123}, it follows that $\mathcal{E}_{01}^{11}$ is LES in $\mathcal{J}_{13}$ and $\mathcal{J}_{16}$.
\begin{figure}[!ht]
\setlength{\unitlength}{1.0cm}
\begin{center}
\begin{picture}(7.8,4.7)(0,0)	
\put(-4.3,0){{\includegraphics[width=5.3cm,height=4.5cm]{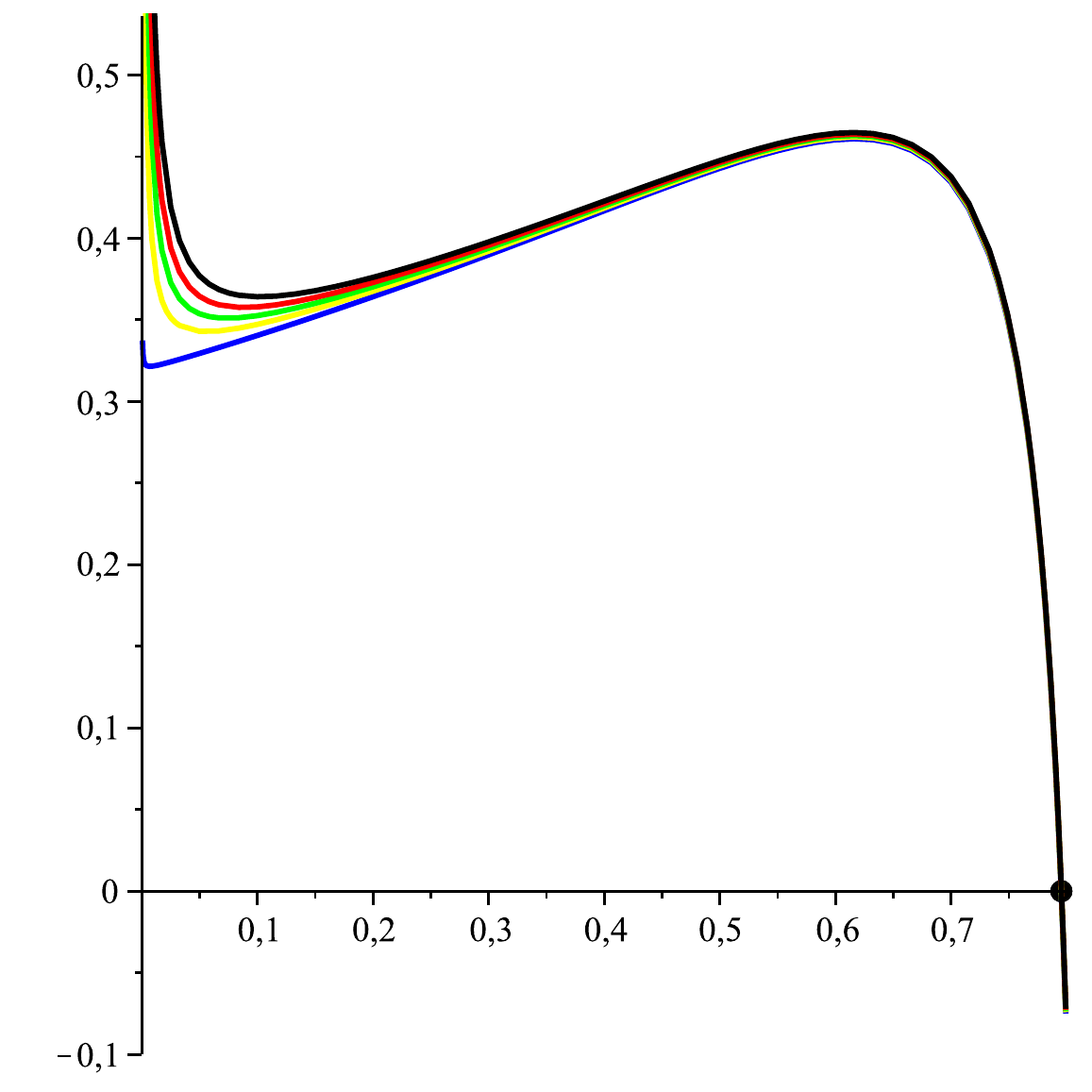}}}	
\put(1.3,0){{\includegraphics[width=5.3cm,height=4.5cm]{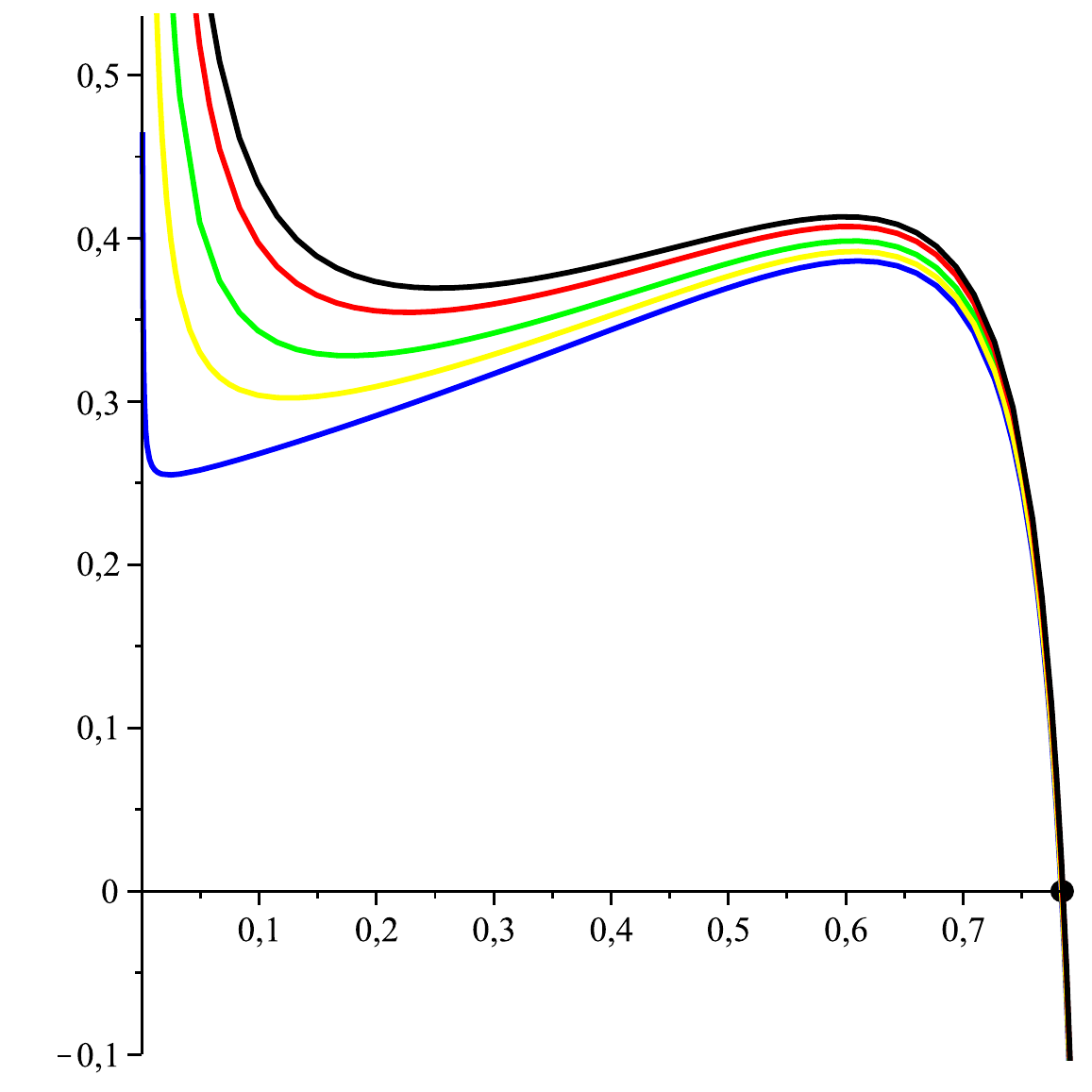}}}
\put(6.5,0){{\includegraphics[width=5.3cm,height=4.5cm]{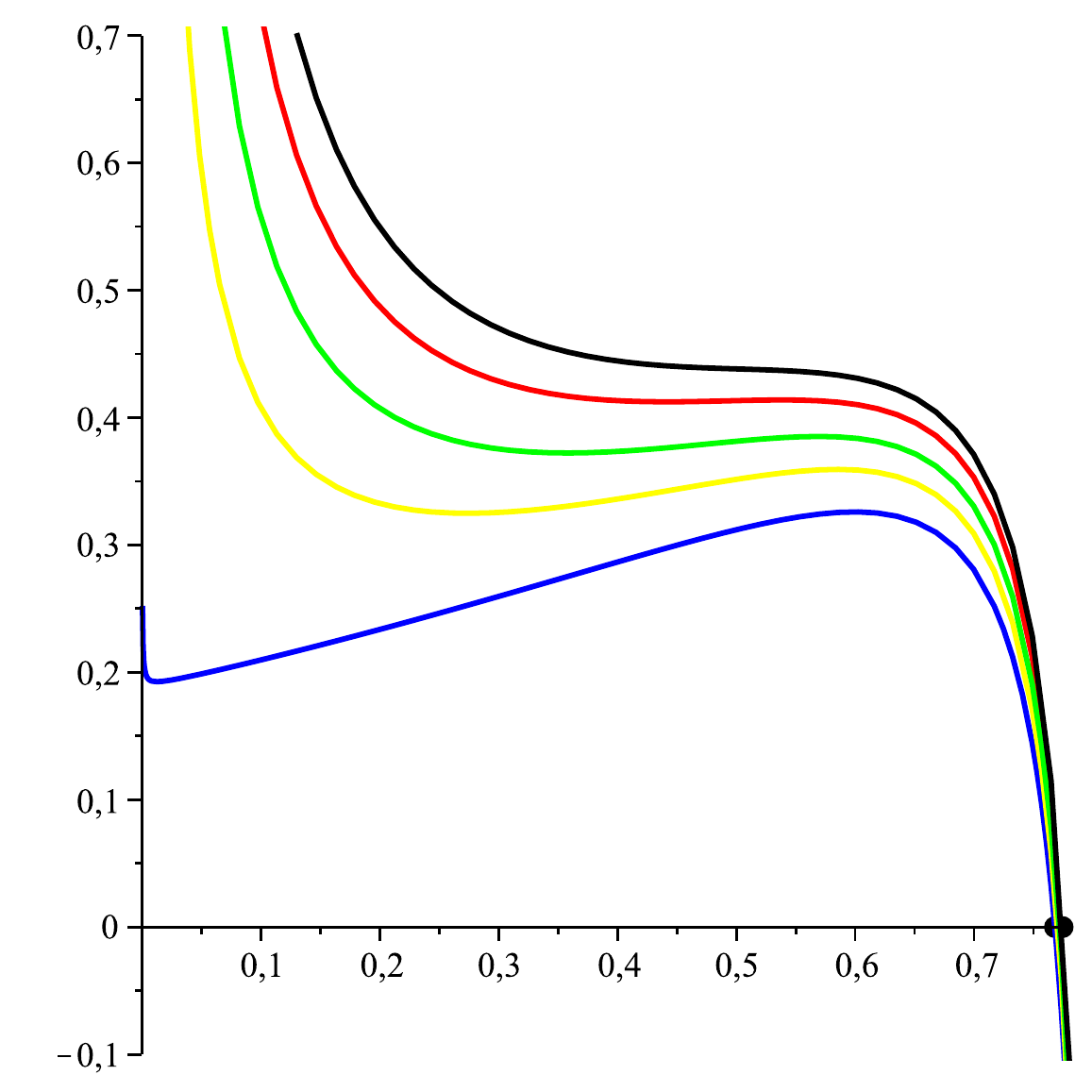}}}
\put(-2,4.5){\sc (a)}
\put(1,0.8){\sc $x$}
\put(-3.7,4.6){\sc $\phi_3$}
\put(4,4.5){\sc (b)}
\put(6.6,0.8){\sc $x$}
\put(1.8,4.6){\sc $\phi_3$}
\put(9,4.5){{\sc (c)}}
\put(7,4.6){\sc $\phi_3$}
\put(11.8,0.6){\sc $x$}
\end{picture}
\end{center}
\vspace{-0.4cm}
\caption{Number of positive solutions of equation $\phi_3(x)=0$ corresponding to the steady state $\mathcal{E}_{01}^{11}$, when $S_2^{in}=150>S_2^m\approx48.740$, and for various values of $D<D_1^{*}\approx0.149$ and $S_1^{in}\in\left[\lambda_1^2,\lambda_1^1\right]$: (a) $D=0.05$ and $S_1^{in}\in \{1.02,1.4,1.7,2,2.35\}$ where $\lambda_1^2 = 1.014$ and $\lambda_1^1=2.366$; (b) $D=0.1$ and $S_1^{in}\in \{2.4,3.4,4.5,6,7\}$ where $\lambda_1^2 = 2.366$ and $\lambda_1^1=7.10$; (c) $D=0.14$ and $S_1^{in}\in \{3.83,7.8,10.8,14,16.5\}$ where $\lambda_1^2 = 3.82$ and $\lambda_1^1=16.56$.} \label{FigCondStabE0111}
\end{figure}
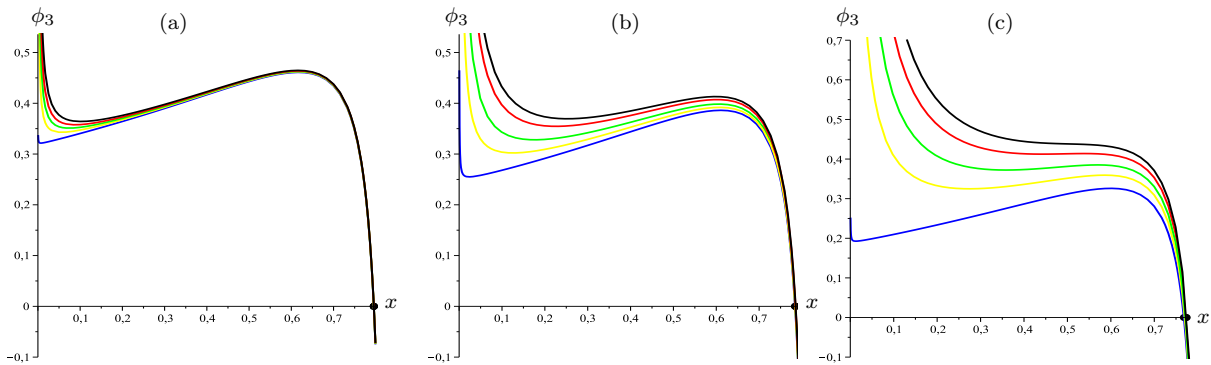
\item For $\mathcal{E}_{1i}^{11},i=1,2$, the existence conditions become
$$
S_1^{in}>\max\left(\lambda_1^1,F_{1i}\right),
$$
where $F_{1i}$ is defined for all $D < \min\left(r_1m_1, D_1^m\right)=D_1^m$ (see \cref{TabFunc1}), that is, this last condition holds for all $\left(D, S_1^{in}\right)$ in the regions at the left of the vertical line $\gamma_8$. From \cref{propPositionCurv}, we have $\lambda_1^1 > F_{11}$ so that the existence condition of $\mathcal{E}_{11}^{11}$ holds in the regions above the curve $\gamma_0$.
However, the existence condition of $\mathcal{E}_{12}^{11}$ holds in the regions on the left of $\gamma_8$ and above the maximum of the two curves $\gamma_0$ and $\gamma_5$ (see \cref{CurvFigDO1}(a)).
From \cref{TabRegionCondDO123}, we see that $\mathcal{E}_{11}^{11}$ exists in $\mathcal{J}_i,i=10,11,12,17-20$ while $\mathcal{E}_{02}^{01}$ exists in $\mathcal{J}_i,i=10-13$.
Recall that the stability condition of $\mathcal{E}_{11}^{11}$ is
$$
g_2'(X_2^{2*})>f_3'(X_2^{2*}).
$$
Similarly to the steady state $\mathcal{E}_{01}^{11}$, the equation $\phi_3(x):=f_3(x)-g_2(x)=0$ has unique solution (that is, $g_2'(X_2^{2*})>f_3'(X_2^{2*})$) as we see in \cref{FigCondStabE1111} for the set of parameters in \cref{TabParamVal} corresponding to the operating diagram in \cref{FigDO1}. Thus, $\mathcal{E}_{11}^{11}$ is LES when it exists.
\end{itemize}
\end{proof}
\begin{figure}[!ht]
\setlength{\unitlength}{1.0cm}
\begin{center}
\begin{picture}(7.8,4.7)(0,0)	
\put(-4.3,0){{\includegraphics[width=5.3cm,height=4.5cm]{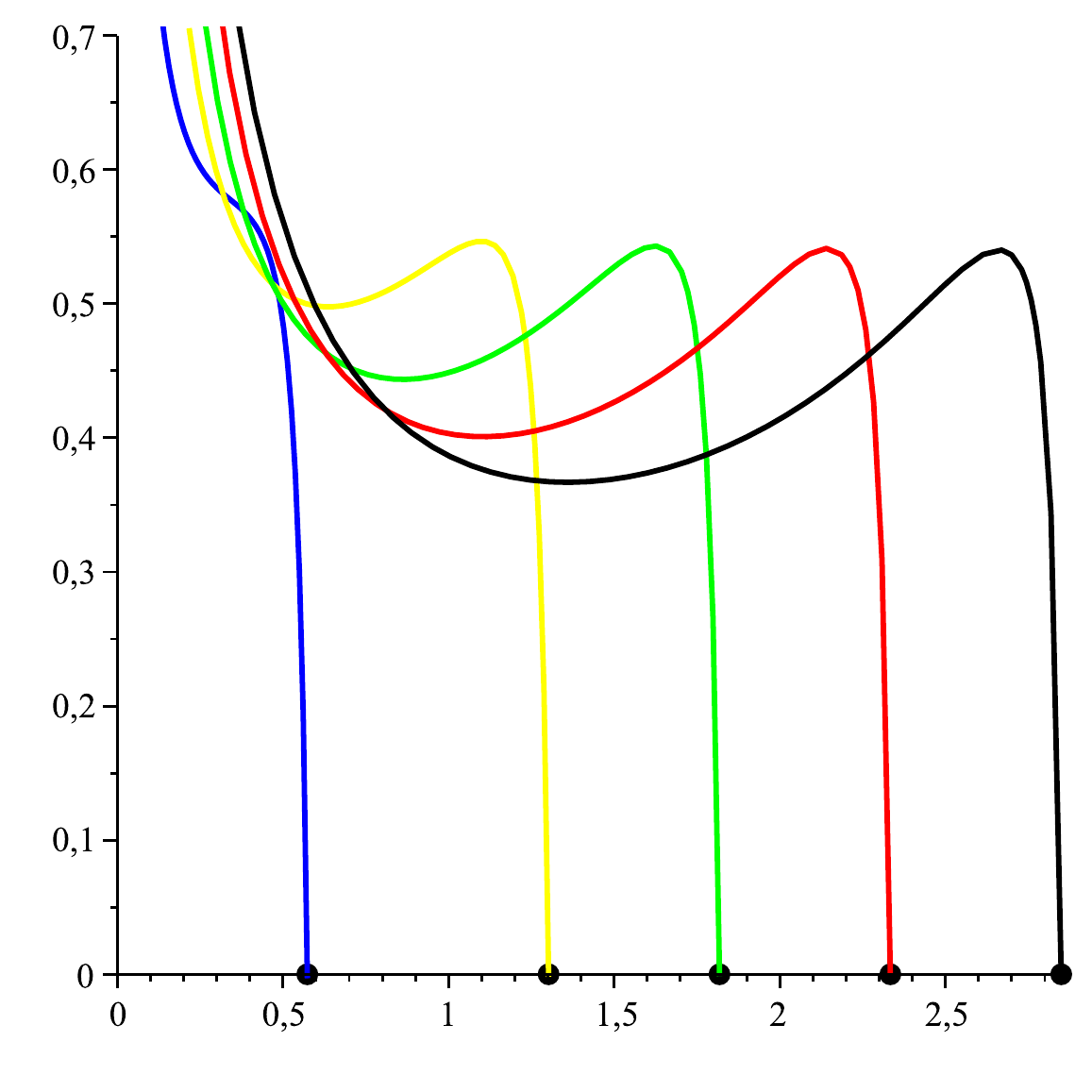}}}	
\put(1.3,0){{\includegraphics[width=5.3cm,height=4.5cm]{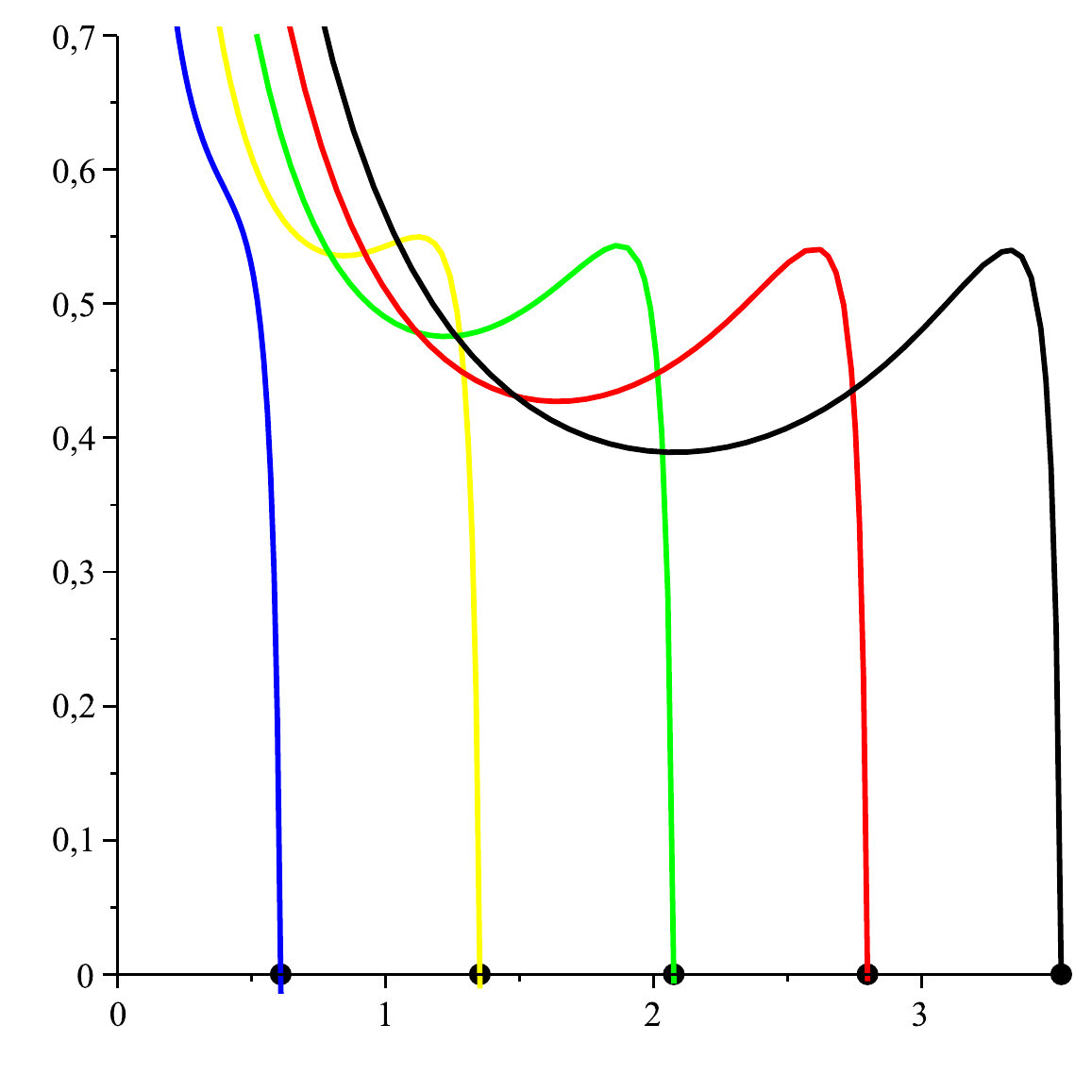}}}
\put(6.5,0){{\includegraphics[width=5.3cm,height=4.5cm]{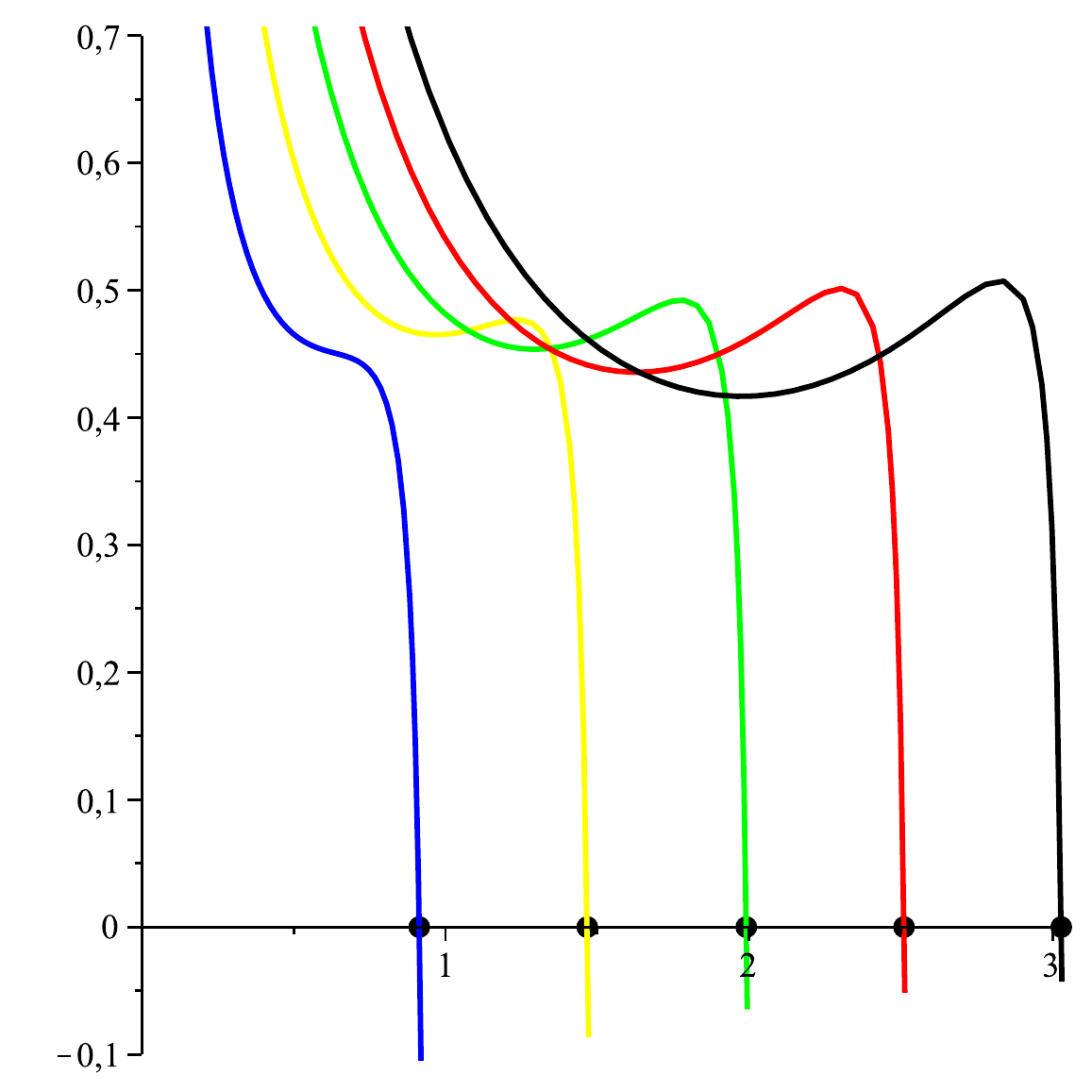}}}
\put(-2,4.5){\sc (a)}
\put(1,0.5){\sc $X$}
\put(-3.9,4.5){\sc $\phi_3$}
\put(4,4.5){\sc (b)}
\put(6.6,0.5){\sc $X$}
\put(1.6,4.5){\sc $\phi_3$}
\put(9,4.5){{\sc (c)}}	
\put(11.8,0.6){\sc $X$}
\put(7,4.5){\sc $\phi_3$}
\end{picture}
\end{center}
\vspace{-0.5cm}
\caption{Number of positive solutions of equation $\phi_3(x)=0$ corresponding to the steady state $\mathcal{E}_{11}^{11}$ when $S_2^{in}=150>S_2^m\approx48.740$, and for various values of $S_1^{in}>\lambda_1^1$ and $D\in \{0.05,0.1,0.17\}$: (a) $D=0.05$ and $S_1^{in}\in \{2.37,72,122,172,222\}$ where $\lambda_1^1\approx2.366$; (b) $D=0.1$ and $S_1^{in}\in \{7.2,77,147,217,287\}$ where $\lambda_1^1\approx 7.1$; (c) $D=0.17$ and $S_1^{in}\in \{40.24, 90, 140, 190, 240\}$ where $\lambda_1^1\approx40.233$.} \label{FigCondStabE1111}
\end{figure}
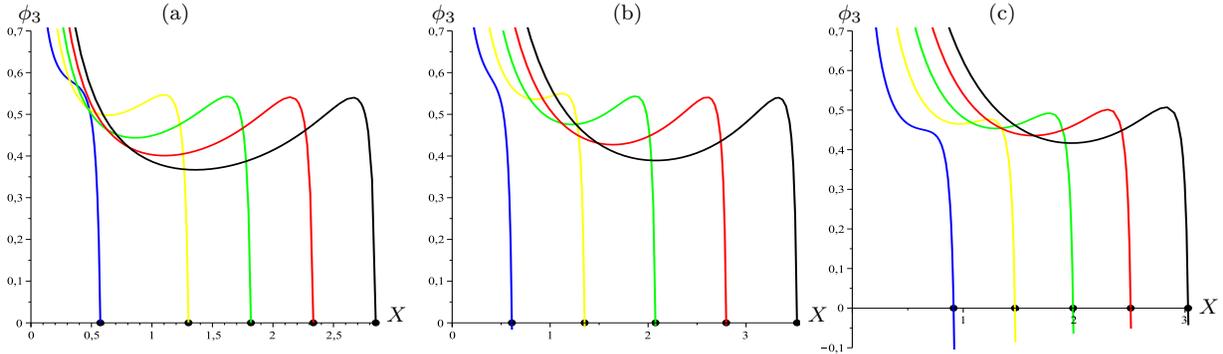
\begin{proof}[\bf{Proof of \cref{PropPosiCurvDfix}}]
For $D$ fixed in $(0,r\mu_2(S_2^m))$ and $r\in(0,1)$, we have $S_{21}^{in*}=\lambda_2^{11}(D,r)$. Therefore,
$$
\lambda_1^1(D,r)+\frac{k_1}{k_2}\left(\lambda_2^{11}(D,r)-S_{21}^{in*}\right)=\lambda_1^1(D,r).
$$
Using \cref{TabSurf}, we see that the curves $\gamma_0$, $\gamma_4$, $\gamma_{10}$ intersect at the same point. In the same way, as ${S_{23}^{in*}}=\lambda_2^{12}(D,r)$. Therefore,
$$
\lambda_1^1(D,r)+\frac{k_1}{k_2}\left(\lambda_2^{12}(D,r)-{S_{23}^{in*}}\right)=\lambda_1^1(D,r).
$$
Consequently, the curves $\gamma_0$, $\gamma_5$, $\gamma_{12}$ intersect at the same point. For $S_2^{in}<S_{21}^{in*}$, we have
$\lambda_2^{11}(D,r)<S_2^{in}$ and $\lambda_2^{12}(D,r)>S_2^{in}$.
Then, the first item holds. Similarly, the result holds when $S_{21}^{in*}<S_2^{in}<{S_{23}^{in*}}$ and $S_2^{in}>{S_{23}^{in*}}$.
\end{proof}

Using the definitions of the auxiliary functions in \cref{TabSurf}, we determine in the following proposition the relative positions of $\gamma_i$, $i=6,7$ and the curve $\gamma_5$.
\begin{lemma}                                           \label{Lem2D}
Fix $r\in(0,1)$ and $D\in \left(0,(1-r)\mu_2(S_2^m)\right)$. Let $S_{2i}^{in*}$, $i=1,\ldots,4$ be the critical values of $S_2^{in}$ defined in \cref{TabSurf}. Let $\phi_j\left(D,r,S_1^{in},S_2^{in}\right)$, $j=1,2$ be the functions defined in \cref{TabFunc1}. If $S_2^{in}>{S_{22}^{in*}}$, then $\phi_1>0$. In addition, if $S_2^{in}>{S_{24}^{in*}}$, then $\phi_2>0$.
\end{lemma}
\begin{proof}[\bf{Proof of \cref{Lem2D}}]
From the expression of $S_{22}^{in*}$ in \cref{TabSurf}, if $S_2^{in}>{S_{22}^{in*}}$, it follows that $S_2^{in}>\lambda_2^{21}(D,r)$ so that $\phi_1>0$ holds. For $S_2^{in}>{S_{24}^{in*}}$. Similarly, if $S_2^{in}>\lambda_2^{22}(D,r)$, then $\phi_2>0$ holds.
\end{proof}
\begin{proof}[\bf{Proof of \cref{PropDO-Dfix}}]
Similarly to the proof of the operating diagram in \cref{propstabstabDO1} and using \cref{PropPosiCurvDfix} and \cref{Lem2D}, we obtain the result.
\end{proof}
\section{Tables}                              \label{Sec_AppTableVal}
In this Appendix, we first provide in \cref{TabParamVal} all the values of the parameters used in the numerical simulations.
\begin{table}[htbp]
{\small
\caption{Parameter values used for systems \cref{ModelAM2S2C,ModelAM2SR4} when the growth rates $\mu_1$ and $\mu_2$ are given by \cref{SpeciFunc}.}\label{TabParamVal}
\vspace{-0.1cm}
\begin{center}
\begin{tabular}{@{\hspace{1mm}}l@{\hspace{1mm}}
                                  @{\hspace{1mm}}l@{\hspace{1mm}} @{\hspace{1mm}}l@{\hspace{1mm}} @{\hspace{1mm}}l@{\hspace{1mm}} @{\hspace{1mm}}l@{\hspace{1mm}}
                                  @{\hspace{1mm}}l@{\hspace{1mm}} @{\hspace{1mm}}l@{\hspace{1mm}} @{\hspace{1mm}}l@{\hspace{1mm}} @{\hspace{1mm}}l@{\hspace{1mm}}
                                  @{\hspace{1mm}}l@{\hspace{1mm}} @{\hspace{1mm}}l@{\hspace{1mm}}
                                  }
\hline
 Parameter    & $m_1$ & $k_{S_1}$ & $m_2$ & $k_{S_2}$ & $k_I$ & $k_1$ & $k_2$ & $k_3$\\
              & $(d^{-1})$ & $(g/l)$ & $(d^{-1})$ & $(mmol/l)$ & $(mmol/l)$ &  $(mmol/g)$ & $(mmol/g)$ & $(mmol/g)$  \\ \hline\hline
\cref{FigDO1,FigDO3,FigDO4,FigDO5Dfix,Figsoleq12,Figsoleq3}
              & $0.6$ & $7.1$ & $0.74$ & $9.28$ & $256$ & $42.14$ & $116.5$ & $268$ \\
\cref{FigDO2} & $0.3$ & $7.1$ & $0.74$ & $9.28$ & $256$ & $42.14$ & $116.5$ & $268$ \\
\hline
\end{tabular}
\end{center}
}
\end{table}

Next, we first define in \cref{TabRegionCondDO123} the regions $\mathcal{J}_i$, $i=0,\dots,31$ of the operating diagram of \cref{FigDO1,FigDO2,FigDO3} in the two-dimensional plane $\left(D, S_1^{in}\right)$ when $r$ and $S_2^{in}$ are fixed so that $r\in(0,1/2)$. Then, we define in \cref{TabRegionCondDO4} the regions $\mathcal{J}_i$, $i=32,\dots,46$ of the operating diagram of \cref{FigDO4} in the two-dimensional plane $\left(D, S_1^{in}\right)$ when $r$ and $S_2^{in}$ are fixed so that $r\in(1/2,1)$. Then, we define in \cref{TabRegionCondDO5} the regions $\mathcal{J}_i$, $i=46,\dots,69$ of the operating diagram of \cref{FigDO5Dfix} in the two-dimensional plane $\left(S_2^{in}, S_1^{in}\right)$ when $r$ and $D$ are fixed so that $r\in(0,1/2)$ and $D\in\left(0,\mu_2\left(S_2^m\right)\right)$.


\begin{table}[htbp]
{\small
\begin{center}
\caption{Definitions of the regions $\mathcal{J}_i$, $i=0,\dots,31$ of the operating diagram of \cref{FigDO1,FigDO2,FigDO3} when $r\in(0,1/2)$. For $i,j=1,2$, the functions $D_i^m$ and $D_i^*$ are defined in \cref{TabSurf} while $F_{ij}$ and $\phi_2$ are defined in \cref{TabFunc1}.}
 \label{TabRegionCondDO123}
 \vspace{-0.1cm}
\begin{tabular}{c} 
 Condition on \\
\begin{tabular}{@{\hspace{0mm}}l@{\hspace{0mm}} @{\hspace{0mm}}l@{\hspace{0mm}}  @{\hspace{0mm}}l@{\hspace{0mm}} @{\hspace{0mm}}l@{\hspace{0mm}} }	
\hline
 $S_2^{in}$ &  $S_1^{in}$ and $\phi_2$ &  $D$ & Region \\
 \hline\hline
\begin{tabular}{l}
$[0,+\infty[$
\end{tabular} &
\begin{tabular}{l}
$\left(\lambda_1^1,F_{22}\right)$ and $ \phi_2<0$ \\
$\left(\lambda_1^1,F_{22}\right)$ and $\phi_2>0$\\
$\left(F_{22},\lambda_1^1\right)$  \\
$\left(\max\left(\lambda_1^1, F_{22}\right),+\infty\right)$ \\
$\left[0,\lambda_1^2\right)$ \\
$\left(\lambda_1^2,\lambda_1^1\right)$  \\
$\left(\lambda_1^1,\min\left(F_{12},F_{22}\right)\right)$ and $\phi_2<0$ \\
$\left(F_{12},F_{22}\right)$ and $\phi_2<0$ \\
$\left(F_{12},F_{22}\right)$ and $\phi_2>0$\\
$\left(\max(F_{12},F_{22}),+\infty\right)$
\end{tabular} &
\begin{tabular}{l}
$\left(D_1^m,\min\left(r_1m_1,D_2^m\right)\right)$\\
$\left(D_1^m,\min\left(r_1m_1,D_2^m\right)\right)$\\
$\left(D_1^m,\min\left(r_2m_1,D_2^*,D_2^m\right)\right)$\\
$\left(D_1^m,\min\left(r_1m_1,r_2m_1,D_2^m\right)\right)$\\
$\left(0,D_1^*\right)$\\
$\left(0,\min\left(r_2m_1,D_1^*\right)\right)$\\
$\left(0,\min\left(r_1m_1,D_1^*,D_2^m\right)\right)$\\
$\left(0,\min\left(r_1m_1,D_1^*,D_1^m,D_2^m\right)\right)$\\
$\left(0,\min\left(r_1m_1,D_1^*,D_1^m,D_2^m\right)\right)$\\
$\left(0,\min\left(r_1m_1,r_2m_2,D_1^*,D_1^m,D_2^m\right)\right)$\\
\end{tabular} &
\begin{tabular}{l}
$\mathcal{J}_6$\\
$\mathcal{J}_7$\\
$\mathcal{J}_{8}$\\
$\mathcal{J}_{9}$\\
$\mathcal{J}_{15}$\\
$\mathcal{J}_{16}$\\
$\mathcal{J}_{17}$\\
$\mathcal{J}_{18}$\\
$\mathcal{J}_{19}$\\
$\mathcal{J}_{20}$\\
\end{tabular}
\\\hline
\begin{tabular}{l}
$\left(S_2^m,+\infty\right)$
\end{tabular}&
\begin{tabular}{l}
$\left[0,\lambda_1^2\right)$\\
$\left[\lambda_1^2,+\infty\right)$\\
$\left[\lambda_1^2,+\infty\right)$\\
$\left[0,\lambda_1^2\right)$\\
$\left[0,\lambda_1^2\right)$  \\
$\left(\lambda_1^2,\min\left(\lambda_1^1, F_{22}\right)\right)$  \\
$\left(F_{22},+\infty\right)$\\
$\left(\lambda_1^1,F_{22}\right)$ and $\phi_2>0$\\
$\left(\lambda_1^1,F_{22}\right)$ and $\phi_2<0$\\
$\left(\lambda_1^2,\max\left(\lambda_1^1,F_{22}\right)\right)$\\
$\left[0,\lambda_1^2\right)$\\
$\left(\max\left(\lambda_1^1,\lambda_1^2\right),\min\left(F_{12},F_{22}\right)\right)$ and $\phi_2>0$\\
$\left(\max\left(\lambda_1^1,F_{22}\right),F_{12}\right)$\\
$\left(F_{22},\lambda_1^1\right)$\\
$\left(F_{22},+\infty\right)$
\end{tabular}&
\begin{tabular}{l}
$\left(D_2^m,+\infty\right)$\\
$\left(D_2^m,+\infty\right)$\\
$\left(D_2^*,D_2^m\right)$\\
$\left(D_2^*,D_2^m\right)$\\
$\left(D_1^m,D_2^*\right)$\\
$\left(D_1^m,r_1m_1\right)$\\
$\left(D_1^*,\min\left(r_2m_1,D_1^m,D_2^m\right)\right)$\\
$\left(D_1^*,\min\left(r_1m_1,D_1^m,D_2^m\right)\right)$\\
$\left(D_1^*,\min\left(r_1m_1,D_1^m,D_2^m\right)\right)$\\
$\left(D_1^*,\min\left(r_2m_1,D_1^m\right)\right)$\\
$\left(D_1^*,D_1^m\right)$\\
$\left(0,\min\left(r_1m_1,r_2m_1,D_2^m\right)\right)$\\
$\left(0,\min\left(r_1m_1,r_2m_1,D_2^m\right)\right)$\\
$\left(0,\min\left(r_2m_1,D_1^*,D_2^m\right)\right)$\\
$\left(D_1^*,\min\left(r_2m_1,D_1^m,D_2^m\right)\right)$
\end{tabular}&
\begin{tabular}{l}
$\mathcal{J}_0$\\
$\mathcal{J}_1$\\
$\mathcal{J}_2$\\
$\mathcal{J}_3$\\
$\mathcal{J}_4$\\
$\mathcal{J}_5$\\
$\mathcal{J}_{10}$\\
$\mathcal{J}_{11}$\\
$\mathcal{J}_{12}$\\
$\mathcal{J}_{13}$\\
$\mathcal{J}_{14}$\\
$\mathcal{J}_{21}$\\
$\mathcal{J}_{22}$\\
$\mathcal{J}_{23}$\\
$\mathcal{J}_{24}$
\end{tabular}
\\\hline
\begin{tabular}{l}
$\left[0,S_2^m\right)$
\end{tabular}&
\begin{tabular}{l}
$\left[0,\lambda_1^2\right)$\\
$\left(\lambda_1^2,F_{21}\right)\cup\left[\lambda_1^2,+\infty\right)$\\
$\left(F_{21},F_{22}\right)$\\
$\left(F_{22},+\infty\right)$\\
$\left(F_{22},+\infty\right)$\\
$\left(F_{12},F_{22}\right)$ and $\phi_2>0$\\
$\left(F_{12},F_{22}\right)$ and $\phi_2<0$\\
$\left(F_{11},F_{12}\right)$\\
$\left(\lambda_1^1,F_{11}\right)$\\
$\left(\lambda_1^2,\min\left(\lambda_1^1, F_{22}\right)\right)$  \\
$\left[0,\lambda_1^2\right)$  \\
\end{tabular}&
\begin{tabular}{l}
$\left(D_2^*,+\infty\right)$\\
$\left(D_2^m,+\infty\right)\cup\left(D_2^*,+\infty\right)$\\
$\left(D_2^*,\min\left(r_2m_1,D_2^m\right)\right)$\\
$\left(D_2^*,\min\left(r_2m_1,D_2^m\right)\right)$\\
$\left(D_1^*,\min\left(r_2m_1,D_1^m,D_2^m\right)\right)$\\
$\left(D_1^*,\min\left(r_1m_1,D_1^m\right)\right)$\\
$\left(D_1^*,\min\left(r_1m_1,D_1^m\right)\right)$\\
$\left(D_1^*,\min\left(r_1m_1,D_1^m\right)\right)$\\
$\left(D_1^*,\min\left(r_1m_1,D_1^m\right)\right)$\\
$\left(D_1^*,\min\left(r_2m_1,D_2^*\right)\right)$\\
$\left(D_1^*,D_2^*\right)$\\
\end{tabular}&
\begin{tabular}{l}
$\mathcal{J}_0$\\
$\mathcal{J}_1$\\
$\mathcal{J}_{25}$\\
$\mathcal{J}_{26}$\\
$\mathcal{J}_{27}$\\
$\mathcal{J}_{28}$\\
$\mathcal{J}_{29}$\\
$\mathcal{J}_{30}$\\
$\mathcal{J}_{31}$\\
$\mathcal{J}_5$\\
$\mathcal{J}_4$\\
\end{tabular}
\\\hline
\end{tabular}\end{tabular}
\end{center}
}
\end{table}
\begin{table}[htbp]
{\small
\begin{center}
\caption{Definitions of the regions $\mathcal{J}_i$, $i=0,20-22,32-46$ of the operating diagram of \cref{FigDO4} in the case $r\in(1/2,1)$ and $S_2^{in}\in\left(S_2^m,+\infty\right)$. The functions $D_i^m$ and $D_i^*$, $i=1,2$ are defined in \cref{TabSurf} while the functions $F_{ij}$ and $\phi_2$, $i,j=1,2$ are defined in \cref{TabFunc1}.}\label{TabRegionCondDO4}
 \vspace{-0.3cm}
\begin{tabular}{c} 
 Condition on \\
\begin{tabular}{@{\hspace{1mm}}l@{\hspace{1mm}} @{\hspace{1mm}}l@{\hspace{1mm}}  @{\hspace{1mm}}l@{\hspace{1mm}} }	
\hline
$S_1^{in}$ and $\phi_2$ &  $D$  & Region \\
 \hline\hline
\begin{tabular}{l}
$\left[0,\lambda_1^1\right)$\\
$\left[\lambda_1^1,+\infty\right)$\\
$\left[\lambda_1^1,+\infty\right)$\\
$\left[0,\lambda_1^1\right)$\\
$\left[0,\lambda_1^1\right)$\\
$\left(\lambda_1^1,\min\left(\lambda_1^2, F_{12}\right)\right)$  \\
$\left(\lambda_1^2,F_{12}\right)$\\
$\left(F_{12},\lambda_1^2\right)$\\
$\left(\max\left(\lambda_1^2, F_{12}\right),+\infty\right)$\\
$\left(F_{12},+\infty\right)$\\
$\left(\lambda_1^2,F_{12}\right)$\\
$\left(\lambda_1^1,\lambda_1^2\right)$\\
$\left[0,\lambda_1^1\right)$\\
$\left[0,\lambda_1^1\right)$\\
$\left(\lambda_1^1,\lambda_1^2\right)$ and $\phi_2<0$ \\
$\left(\lambda_1^1,\lambda_1^2\right)$ and $\phi_2>0$ \\
$\left[\lambda_1^2,F_{22}\right)$ and $\phi_2<0$ \\
$\left[\lambda_1^2,F_{22}\right)$ and $\phi_2>0$ \\
$\left(F_{22},F_{12}\right)$\\
$\left(F_{12},+\infty\right)$
\end{tabular}&
\begin{tabular}{l}
$\left(D_1^m,+\infty\right)$\\
$\left(D_1^m,r_1m_1\right)$\\
$\left(D_1^*,\min\left(r_1m_1,D_1^m\right)\right)$\\
$\left(D_1^*,D_1^m\right)$\\
$\left(D_2^m,D_1^*\right)$\\
$\left(D_2^m,r_1m_1\right)$\\
$\left(D_2^m,r_2m_1\right)$\\
$\left[0,\min\left(r_1m_1,D_1^*,D_1^m\right)\right)$\\
$\left(D_2^m,\min\left(r_1m_1,r_2m_1,D_1^m\right)\right)$\\
$\left(D_2^*,\min\left(r_1m_1,D_1^m,D_2^m\right)\right)$\\
$\left(D_2^*,\min\left(r_2m_1,D_2^m\right)\right)$\\
$\left(D_2^*,\min\left(r_1m_1,D_2^m\right)\right)$\\
$\left(D_2^*,D_2^m\right)$\\
$\left(0,D_2^*\right)$\\
$\left(0,\min\left(r_1m_1,D_2^*,D_2^m\right)\right)$\\
$\left(0,\min\left(r_1m_1,D_2^*,D_2^m\right)\right)$\\
$\left(0,\min\left(r_1m_1,r_2m_1,D_2^m\right)\right)$\\
$\left(0,\min\left(r_1m_1,r_2m_1,D_2^m\right)\right)$\\
$\left(0,\min\left(r_2m_1,D_2^*,D_2^m\right)\right)$\\
$\left(0,\min\left(r_1m_1,D_2^*,D_1^m\right)\right)$\\
\end{tabular}&
\begin{tabular}{l}
$\mathcal{J}_0$\\
$\mathcal{J}_{32}$\\
$\mathcal{J}_{33}$\\
$\mathcal{J}_{34}$\\
$\mathcal{J}_{35}$\\
$\mathcal{J}_{36}$\\
$\mathcal{J}_{37}$\\
$\mathcal{J}_{38}$\\
$\mathcal{J}_{39}$\\
$\mathcal{J}_{40}$\\
$\mathcal{J}_{41}$\\
$\mathcal{J}_{42}$\\
$\mathcal{J}_{43}$\\
$\mathcal{J}_{15}$\\
$\mathcal{J}_{44}$\\
$\mathcal{J}_{45}$\\
$\mathcal{J}_{46}$\\
$\mathcal{J}_{21}$\\
$\mathcal{J}_{22}$\\
$\mathcal{J}_{20}$
\end{tabular}
\\\hline
\end{tabular}\end{tabular}
\end{center}}
\end{table}
\begin{table}[htbp]
{\small
\begin{center}
\caption{Definitions of the regions $\mathcal{J}_i$, $i=0,1,4,5,15,27,46-69$ of the operating diagram of \cref{FigDO5Dfix} in the plane $\left(S_2^{in}, S_1^{in}\right)$ when $r\in(0,1/2)$ and $D\in\left(0,\mu_2\left(S_2^m\right)\right)$. For $i,j=1,2$, the functions $D_i^m$ and $D_i^*$ are defined in \cref{TabSurf} while the functions $F_{ij}$ and $\phi_2$ in \cref{TabFunc1}.}\label{TabRegionCondDO5}
 \vspace{-0.2cm}
\begin{tabular}{c} 
 Condition on \\
\begin{tabular}{@{\hspace{1mm}}l@{\hspace{1mm}} @{\hspace{1mm}}l@{\hspace{1mm}}  @{\hspace{1mm}}l@{\hspace{1mm}} @{\hspace{1mm}}l@{\hspace{1mm}} }	
\hline
 $S_2^{in}$ &  $S_1^{in}$ and $\phi_2$  & Region \\
 \hline\hline
\begin{tabular}{l}
$\left[0,\lambda_2^{21}\right)$
\end{tabular}&
\begin{tabular}{l}
$\left[0,\lambda_1^2\right)$\\
$\left(\lambda_1^2, F_{21}\right)$\\
$\left(F_{21},\lambda_1^1\right)$\\
$\left(\lambda_1^1, F_{11}\right)$\\
$\left(F_{11}, F_{12}\right)$\\
$\left(F_{12},F_{22}\right)$ and $\phi_2<0$\\
$\left(F_{12},F_{22}\right)$ and $\phi_2>0$\\
$\left(F_{22},+\infty\right)$
\end{tabular}&
\begin{tabular}{l}
$\mathcal{J}_0$\\
$\mathcal{J}_1$\\
$\mathcal{J}_{64}$\\
$\mathcal{J}_{65}$\\
$\mathcal{J}_{66}$\\
$\mathcal{J}_{67}$\\
$\mathcal{J}_{68}$\\
$\mathcal{J}_{69}$\\
\end{tabular}
\\\hline
\begin{tabular}{l}
$\left(\lambda_2^{21},\lambda_2^{11}\right)$
\end{tabular}&
\begin{tabular}{l}
$\left[0,\lambda_1^2\right)$\\
$\left(\lambda_1^2, \lambda_1^1\right)$\\
$\left(\lambda_1^1, F_{11}\right)$\\
$\left(F_{11}, F_{12}\right)$\\
$\left(F_{12},F_{22}\right)$ and $\phi_2<0$\\
$\left(F_{12},F_{22}\right)$ and $\phi_2>0$\\
$\left(F_{22},+\infty\right)$\\
\end{tabular}&
\begin{tabular}{l}
$\mathcal{J}_4$\\
$\mathcal{J}_5$\\
$\mathcal{J}_{63}$\\
$\mathcal{J}_{62}$\\
$\mathcal{J}_{61}$\\
$\mathcal{J}_{60}$\\
$\mathcal{J}_{27}$\\
\end{tabular}
\\\hline
\begin{tabular}{l}
$\left(\lambda_2^{11},\lambda_2^{12}\right)$
\end{tabular}&
\begin{tabular}{l}
$\left[0,\lambda_1^2\right)$\\
$\left(\lambda_1^2, \lambda_1^1\right)$\\
$\left(\lambda_1^1, F_{12}\right)$\\
$\left(F_{12},F_{22}\right)$ and $\phi_2<0$\\
$\left(F_{12},F_{22}\right)$ and $\phi_2>0$\\
$\left(F_{22},+\infty\right)$\\
\end{tabular}&
\begin{tabular}{l}
$\mathcal{J}_{15}$\\
$\mathcal{J}_{56}$\\
$\mathcal{J}_{46}$\\
$\mathcal{J}_{57}$\\
$\mathcal{J}_{58}$\\
$\mathcal{J}_{59}$\\
\end{tabular}\\\hline
\begin{tabular}{l}
$\left(\lambda_2^{12},\lambda_2^{22}\right)$
\end{tabular}&
\begin{tabular}{l}
$\left(\max\left(\lambda_1^1,F_{22}\right),+\infty\right)$\\
$\left(\lambda_1^1,F_{22}\right)$ and $\phi_2>0$\\
$\left(\lambda_1^1,F_{22}\right)$ and $\phi_2<0$\\
$\left(F_{22},\lambda_1^1\right)$\\
$\left(\lambda_1^2,\min\left(\lambda_1^1,F_{22}\right)\right)$  \\
$\left[0,\lambda_1^2\right)$
\end{tabular}&
\begin{tabular}{l}
$\mathcal{J}_{50}$\\
$\mathcal{J}_{51}$\\
$\mathcal{J}_{52}$\\
$\mathcal{J}_{53}$\\
$\mathcal{J}_{54}$\\
$\mathcal{J}_{55}$\\
\end{tabular}
\\\hline
\begin{tabular}{l}
$\left(\lambda_2^{22},+\infty\right)$
\end{tabular}&
\begin{tabular}{l}
$\left[0,\lambda_1^2\right)$\\
$\left(\lambda_1^2, \lambda_1^1\right)$\\
$\left(\lambda_1^1,+\infty\right)$
\end{tabular}&
\begin{tabular}{l}
$\mathcal{J}_{47}$\\
$\mathcal{J}_{48}$\\
$\mathcal{J}_{49}$\\
\end{tabular}
\\\hline
\begin{tabular}{l}
\end{tabular}&
\begin{tabular}{l}

\end{tabular}&
\begin{tabular}{l}

\end{tabular}
\end{tabular}\end{tabular}
\end{center}}
\end{table}
 
\end{document}